\documentclass[letterpaper,11pt]{amsart}
\usepackage[latin1]{inputenc}
\usepackage[greek,english]{babel}
\usepackage[OT2,OT1]{fontenc}
\usepackage[all]{xy}
\usepackage{amscd}
\usepackage{amssymb}
\usepackage{amsthm}
\usepackage{enumitem}
\usepackage{mathrsfs,bbm}
\usepackage[pdftex]{graphicx} %xcolor
\usepackage{comment}
\usepackage[all]{xy}
\usepackage{amscd}
\usepackage{amssymb,amsmath,latexsym}
\usepackage{amsthm}
\usepackage{enumitem}
\usepackage{mathrsfs,bbm}
\usepackage[latin1]{inputenc}
\usepackage[pdftex]{graphics}
\usepackage{tikz-cd}
\usepackage[margin=30truemm]{geometry}
\newtheoremstyle{component}{}{}{}{}{\itshape}{.}{.5em}{\thmnote{#3}{#1}}
\theoremstyle{plain}
\newtheorem{thm}{Theorem}[section]
\newtheorem{pro}[thm]{Proposition}
\newtheorem{lem}[thm]{Lemma}

\theoremstyle{definition}
\newtheorem{defn}[thm]{Definition}
\newtheorem{conj}[thm]{Conjecture}

\theoremstyle{remark}
\newtheorem{rem}[thm]{Remark}

\newtheorem{thmN}{Theorem}

\newtheorem{thmA}{Theorem}

\newtheorem*{imcF}{{\bf Iwasawa Main Conjecture for $\psi$}}
\theoremstyle{component}
\newtheorem*{component}{}
\newenvironment{claim}[1]{\par\addvspace{\baselineskip}\noindent\underline{\bfseries
Claim}:\space#1}{\par\addvspace{\baselineskip}}
\def\ker{\mathrm{Ker}}
\DeclareFontEncoding{OT2}{}{}
\DeclareFontSubstitution{OT2}{cmr}{m}{n}
\DeclareFontFamily{OT2}{cmr}{\hyphenchar\font45}
\DeclareFontShape{OT2}{cmr}{m}{n}{%
<5><6><7><8><9>gen*wncyr%
<10><10.95><12><14.4><17.28><20.74><24.88>wncyr10}{}
\DeclareFontShape{OT2}{cmr}{b}{n}{%
<5><6><7><8><9>gen*wncyb%
<10><10.95><12><14.4><17.28><20.74><24.88>wncyb10}{}
\DeclareMathAlphabet{\mathcyr}{OT2}{cmr}{m}{n}
\DeclareMathAlphabet{\mathcyb}{OT2}{cmr}{b}{n}
\SetMathAlphabet{\mathcyr}{bold}{OT2}{cmr}{b}{n}
\makeatletter
\def\theequation{\thesection.\@arabic\c@equation}
\@addtoreset{equation}{section}
\makeatother

\allowdisplaybreaks
\title[On the $p$-adic Artin $L$-functions over CM fields]{On the $p$-adic $L$-function and \linebreak Iwasawa Main Conjecture for an Artin motive \linebreak over a CM field}
\author{Takashi Hara}
\author{Tadashi Ochiai}
\address[T.~Hara]{College of Liberal Arts \\
Tsuda University \\
2-1-1 Tsuda-machi, Kodaira, Tokyo 187-8577, Japan}
\address[T.~Ochiai]{
Department of Mathematics \\ 
Institute of Science Tokyo \\ 
2-12-1 Ookayama,
Meguro-ku, Tokyo 152-8551, Japan
}
\email[T.~Hara]{t-hara@tsuda.ac.jp}
\email[T.~Ochiai]{ochiai.t.1998@m.isct.ac.jp}
\usepackage{}	% required for `\align' (yatex added)
\usepackage{amsmath}	% required for `\align' (yatex added)
\begin{document}
\begin{abstract}
For an algebraic Hecke character defined on a CM field $F$ of degree $2d$, Katz constructed a $p$-adic $L$-function of $d+1+\delta_{F,p}$ variables in his innovative paper published in 1978, where $\delta_{F,p}$ denotes the Leopoldt defect for $F$ and $p$. We shall generalise the result of Katz under several technical conditions (containing the absolute unramifiedness of $F$ at~$p$), and construct a $p$-adic Artin $L$-function of $d+1+\delta_{F,p}$ variables, which interpolates critical values of the Artin $L$-function associated to a $p$-unramified Artin representation of the absolute Galois group $G_F$ of $F$. 
Our construction is an analogue over a CM field of Greenberg's construction over a totally real field, but there appear new difficulties which do not matter in Greenberg's case.  
\end{abstract}

\maketitle
%%%%%%%%%%%%%%%%%%%%%%%%%%%%%%%%%%%%%%%%%%%%
\tableofcontents

%%%%%%%%%%%%%%%%%%%%%%%%%%%%%%%%%%%%%%%%%%%%
\section{Introduction} \label{sc:Introduction}
%%%
%%%%%%%%%%%%%%%%%%%%%%%%%%%%%%%%%%%%%%%%%%%%
%%%
%%%
The purpose of the present article is to construct {\em $p$-adic Artin $L$-functions} %need revise
for (non-commutative) Artin representations defined over CM fields and verify their integrality admitting the (abelian) Iwasawa main conjecture for CM fields. As a byproduct, we obtain equality of the Iwasawa main conjecture for such Artin representations. Before giving a detailed account of our results, we introduce several basic notation. The absolute Galois group of a number field $K$ is denoted as $G_K$ throughout this article. 
Let $F$ be a CM field of degree $2{d}$ and $F^+$
its maximal totally real subfield. We write $\mathfrak{r}_F$ and $D_{F^+}$ for the ring of integers of $F$ and the absolute discriminant of $F^+$, respectively. We fix an odd prime number $p$, an algebraic closure $\overline{\mathbb{Q}}$ of the rational number field $\mathbb{Q}$, an embedding $\iota_\infty \colon \overline{\mathbb{Q}} \hookrightarrow \mathbb{C}$ of $\overline{\mathbb{Q}}$ into the complex number field $\mathbb{C}$, and 
an isomorphism $\boldsymbol{\iota}\colon \mathbb{C}\xrightarrow{\, \sim \,}\mathbb{C}_p$ between $\mathbb{C}$ and the completion $\mathbb{C}_p$ of a fixed algebraic closure $\overline{\mathbb{Q}}_p$ of the $p$-adic number field $\mathbb{Q}_p$. 
Suppose that every prime ideal of $F^+$ lying above $(p)$ splits completely in the quadratic extension $F /F^+$. We fix a {\em $p$-ordinary CM type} $\Sigma_{F}$ of $F$, whose choice amounts to choosing one of $\mathfrak{P}$ or $\mathfrak{P}^c$ for each prime ideal $\mathfrak{p}$ of $F^+$ lying above $(p)$ if $\mathfrak{p}$ is decomposed as  $\mathfrak{p}\mathfrak{r}_F=\mathfrak{P}\mathfrak{P}^c$; 
 see Section \ref{ssc:setting_Hecke} for its precise definition. 
Let $F_{\max}$ be the composition of all $\mathbb{Z}_p$-extensions of $F$ (in $\overline{\mathbb{Q}}$). Due to global class field theory, the Galois group $\Gamma_{F,\, \max}$ of $F_{\max}/F$ is known to be a free $\mathbb{Z}_p$-module of rank
equal to or greater than $d+1$; the equality holds if and only if the Leopoldt conjecture for $F$ and $p$ holds true.

\medskip
Let us review historical background on Iwasawa theory for CM fields. We first take a look at construction of $p$-adic Hecke $L$-functions for ($p$-ordinary) CM fields. Under certain assumptions, Katz \cite[Theorem (5.3.0)]{katz} has constructed a $p$-adic Hecke $L$-function over $F$ as an element of the Iwasawa algebra of $\Gamma_{F,\, \max}$. His $p$-adic $L$-function interpolates the values at $0$ of the Hecke $L$-functions for various algebraic Hecke characters with appropriate infinity type and conductor dividing a power of $p$. 
Later Hida and Tilouine \cite[Theorem II.]{HT-katz} generalised Katz's construction, and relaxed constraints on the conductors of algebraic Hecke characters appearing in the interpolation property. 

Now let $\psi^\mathrm{gal} \colon G_F \rightarrow \overline{\mathbb{C}}_p^\times$ be a finite character which is at most tamely ramified at every finite place of $F$ lying above $(p)$, and  suppose that the field $F_\psi$ corresponding to the kernel of $\psi^\mathrm{gal}$ is linearly disjoint from $F_{\max}$ over $F$. The $p$-adic Hecke $L$-function over $F$ is constructed for such a {\em branch character}. We take a finite flat extension $\mathcal{O}$ of $\mathbb{Z}_p$ containing the image of $\psi^\mathrm{gal}$, and define  $\widehat{\mathcal{O}}^\mathrm{ur}$ as the composition of $\mathcal{O}$ and $\widehat{\mathbb{Z}}_p=W(\overline{\mathbb{F}}_p)\; (\subset \mathcal{O}_{\mathbb{C}_p})$, the Witt ring of the algebraic closure $\overline{\mathbb{F}}_p$ of $\mathbb{F}_p$. Let $\psi \colon \mathbb{A}_F^\times/F^\times \rightarrow \overline{\mathbb{Q}}^\times$ be the algebraic Hecke character of $F$ corresponding to $\psi^\mathrm{gal}$ via global class field theory. Recall that, for any algebraic Hecke character $\chi\colon \mathbb{A}_F^\times/F^\times \rightarrow \overline{\mathbb{Q}}^\times$ of conductor $\mathfrak{f}_\chi$, the {\em Hecke $L$-function $L(\chi,s)$ associated to $\chi$} is defined as the meromorphic continuation of the Euler product $\prod_{\mathfrak{l}\nmid \mathfrak{f}_\chi} (1-\chi_{\mathfrak{l}}(\varpi_\mathfrak{l})\mathcal{N}\mathfrak{l}^{-s})^{-1}$, where the product is taken over all prime ideals of $F$ relatively prime to $\mathfrak{f}_\chi$. Here, for a prime ideal $\mathfrak{l}$ of $F$, we write  $\varpi_\mathfrak{l}$ for a uniformiser of the $\mathfrak{l}$-adic completion $F_\mathfrak{l}$ of $F$, and $\mathcal{N} \mathfrak{l}$ for the absolute norm of $\mathfrak{l}$. Throughout the present article, we are especially concentrated on the case where $F$ is {\em absolutely unramified at $(p)$}. In the case, we may slightly improve the results of Katz and Hida--Tilouine as follows. 

\begin{thmN}[=Theorem~\ref{thm:KHT}, Katz, Hida--Tilouine] \label{theorem:p-adic_Hecke_L}
Assume that $F$ is absolutely unramified at $(p)$ and take a branch character $\psi^\mathrm{gal}$ as above. 
Then there exists a unique element $L_{p, \Sigma_F}(\psi )$ of $\widehat{\mathcal{O}}^\mathrm{ur} [[\Gamma_{F,\, \max}]]$ satisfying
\begin{equation} \label{eq:interpolation_katz}
 \frac{\eta^{\mathrm{gal}}(L_{p, \Sigma_{F}}(\psi))}
{\Omega_{\mathrm{CM},p,F}^{w_\eta \mathsf{t}+2\mathsf{r}_\eta}} = \dfrac{(\mathfrak{r}_F^\times : \mathfrak{r}_{F^+}^\times)}{2^d\sqrt{\lvert D_{F^+}\rvert}}
 i^{\lvert -w_\eta \mathsf{t}-\mathsf{r}_\eta\rvert }
\prod_{v \in \Sigma_{F,p}} \mathrm{Eul}_v(\psi\eta,0)
\frac{\Lambda (\psi\eta ,0 )}{\Omega_{\mathrm{CM},\infty,F}^{w_\eta \mathsf{t}+2\mathsf{r}_\eta}}
\end{equation}
with
\begin{align*}
& \mathrm{Eul}_v(\psi\eta,0)=
\begin{cases}
 L_{v^c}(\psi \eta,0)^{-1}L_v((\psi\eta)^\vee,1)^{-1} & \text{if $\psi\eta$ is unramified at $v$}, \\
\epsilon((\psi \eta)_v,\mathbf{e}_{F,v},dx_v)^{-1} & \text{if $\psi\eta$ is ramified at $v$},
\end{cases} \\ 
& 
\Lambda (\psi\eta ,s )  = L_\infty (\psi\eta, s) L(\psi\eta,s)
\end{align*}
for each algebraic Hecke character $\eta\colon \mathbb{A}_F^\times/F^\times \rightarrow \overline{\mathbb{Q}}^\times$ satisfying the following two conditions:
\begin{enumerate}[label={\upshape (\roman*)}]
\item  
the Galois character $\eta^{\mathrm{gal}}\colon G_F\longrightarrow \mathbb{C}_p^\times$ 
corresponding to $\eta$ factors through $\Gamma_{F,\, \max}$; 
\item the infinity type of $\eta$ satisfies both $-w_\eta-r_{\eta,\sigma}\leq -1$ and $r_{\eta,\sigma}\geq 0$ for every $\sigma\in \Sigma_F$. 
\end{enumerate}
\end{thmN}

There are several notational remarks on the statement of the theorem above. We set $\mathsf{t}=(1,1,\ldots,1)\in \mathbb{Z}^{\Sigma_F}$, and let $(w_\eta,\mathsf{r}_\eta)\in \mathbb{Z}\times \mathbb{Z}^{\Sigma_F}$ denote the infinity type of  $\eta$ defined as $\eta\bigl((x_\sigma)_{\sigma\in \Sigma_F}\bigr)=\prod_{\sigma\in \Sigma_F}x_\sigma^{w_\eta+r_{\eta,\sigma}}\overline{x}_\sigma^{-r_{\eta,\sigma}}$. We use multi-index notation in \eqref{eq:interpolation_katz}; refer to Remark~\ref{rem:multi_index} for details.
The symbols $\epsilon((\psi\eta)_v,\mathbf{e}_{F,v},dx_v)$ and $L_\infty(\psi\eta,s)$ respectively denote Deligne's local constant at the $p$-adic place $v$ and the archimedean local factor (or the gamma factor) of $L(\psi\eta,s)$, whose explanations are given in Section~\ref{ssc:p-adic_Hecke}.   The product $\Lambda(\psi\eta,s)$ of $L_\infty(\psi\eta,s)$ and $L(\psi\eta,s)$ is called the {\em completed Hecke $L$-function}, which satisfies the {\em functional equation} $\Lambda(\psi \eta,s)=\epsilon(\psi\eta,s)\Lambda((\psi\eta)^{-1},1-s)$ for an exponential function $\epsilon(\psi \eta,s)$  (the global epsilon factor).
The modified complex CM period $\Omega_{\mathrm{CM},\infty,F} \in \mathbb{C}^{\times,\Sigma_F}$ and the normalised $p$-adic CM period $\Omega_{\mathrm{CM},p,F} \in 
(\widehat{\mathcal{O}}^\mathrm{ur})^{\times, \Sigma_F}$ %need revise
will be defined in Definition \ref{definition:complex_CM_period&p-adic_CM_period}. 

\medskip
Next we shall review the (abelian) Iwasawa main conjecture for CM fields.
Let  $ M_{\Sigma_F}$  denote the maximal abelian pro-$p$ extension of $F_{\max}$ which is unramified outside all the finite places lying above $\Sigma_{F,p}$. 
Then the $\psi$-isotypic quotient $\mathrm{Gal}( M_{\Sigma_F}F_\psi/ F_{\max} F_\psi)_\psi$ 
of the $\Sigma_{F,p}$-ramified Iwasawa module $\mathrm{Gal}( M_{\Sigma_F}F_\psi/ F_{\max} F_\psi)\otimes_{\mathbb{Z}_p}\mathcal{O}$ is a finitely generated torsion 
${\mathcal{O}}[[\Gamma_{F,\, \max}]]$-module as is explained in  \cite[Theorem 1.2.2 (iii)]{HT-aIMC}, and thus its characteristic ideal $\mathrm{char}_{\mathcal{O}[[\Gamma_{F,\, \max}]]}(\mathrm{Gal}( M_{\Sigma_F}F_\psi / F_{\max} F_\psi)_\psi)$ is defined as a nontrivial ideal of $\mathcal{O}[[\Gamma_{F,\, \max}]]$. Under these settings, the multi-variable Iwasawa main conjecture for CM fields is formulated as follows.
 
\begin{imcF}
Let $\psi^\mathrm{gal}$ be a branch character chosen as above.
 Then 
\begin{align} \label{eq:imc_CM_abelian}
\mathrm{char}_{\widehat{\mathcal{O}}^\mathrm{ur}[[\Gamma_{F,\, \max}]]}(\mathrm{Gal}( M_{\Sigma_F}F_\psi / F_{\max} F_\psi)_\psi \widehat{\otimes}_{\mathcal{O}} \widehat{\mathcal{O}}^\mathrm{ur} ) = (L_{p, \Sigma_F}(\psi )) 
\end{align}
holds as an equality of ideals of $\widehat{\mathcal{O}}^\mathrm{ur}[[\Gamma_{F,\, \max}]]$.
\end{imcF}

Under several assumptions, Ming-Lun Hsieh has recently verified in \cite[Theorem 8.16]{Hsieh14} that there exists an inclusion 
\begin{align*}
 \mathrm{char}_{\widehat{\mathcal{O}}^\mathrm{ur} [[\Gamma_{F,\, \max}]]}(\mathrm{Gal}( M_{\Sigma_F}/F_{\max} F_\psi )_\psi \widehat{\otimes}_{\mathcal{O}} \widehat{\mathcal{O}}^\mathrm{ur} ) \subset (L_{p, \Sigma_F}(\psi ))
\end{align*}
predicted by the Iwasawa main conjecture for $\psi$  \eqref{eq:imc_CM_abelian} when the Leopoldt conjecture for $F$ and $p$  holds true, adopting the method of Eisenstein congruences on the unitary group $\mathrm{GU}(2,1)$. 
Note that the Iwasawa algebra $\widehat{\mathcal{O}}^\mathrm{ur} [[\Gamma_{F,\, \max}]]$ is isomorphic to the ring of formal power series in $d+1$ variables under the Leopoldt conjecture for $F$ and $p$. 
Meanwhile, Hida obtains, under other technical assumptions, an equality between the algebraic characteristic ideal and the analytic ideal \cite[{\sc Theorem}]{hida2006} in the anticyclotomic Iwasawa algebra, which is a $d$-variable quotient of the Iwasawa algebra $\widehat{\mathcal{O}}^\mathrm{ur} [[\Gamma_{F,\, \max}]]$. A standard specialisation argument combined with these two results implies the desired equality \eqref{eq:imc_CM_abelian} in $\widehat{\mathcal{O}}^\mathrm{ur} [[\Gamma_{F,\, \max}]]$ if the branch character $\psi^{\mathrm{gal}}$ is anticyclotomic. 
When the CM field $F$ under consideration is a composite of a totally real field and an imaginary quadratic field, 
Hsieh has another result \cite[Theorem 8.18]{Hsieh14} on the Iwasawa main conjecture for a certain branch character $\psi^\mathrm{gal}$ by using Rubin's equality of the  
two-variable Iwasawa main conjecture for an imaginary quadratic field \cite[Theorem 4.1 (i)]{Rubin91} in place of the anticyclotomic main conjecture. 

\bigskip
In the present article, we shall construct a $p$-adic Artin $L$-function $L_{p, \Sigma_F}(M(\rho ))$ associated to an Artin representation $\rho\colon G_F\rightarrow \mathrm{Aut}_{E}\, V_\rho$ when $F$ is absolutely unramified at $(p)$, the field $F_\rho$ corresponding to the kernel of $\rho$ is also a CM field, and $\rho$ is unramified at every finite place of $F$ lying above $(p)$. 
Since the Galois character $\psi^{\mathrm{gal}}$ of finite order is an Artin representation of degree $1$, Theorem A below can be regarded as an extension of Katz, Hida and Tilouine's theorem stated above under such a restricted situation.
 
\begin{thmA}[a part of Theorem \ref{theorem:integrality_of_CMp-adicL_Artin}] \label{thm:MainA}
Suppose that $F$ is absolutely unramified at $(p)$. Let $\rho$ be an Artin representation of $G_F$ unramified at every finite place of $F$ lying above $(p)$, and suppose that 
the field $F_\rho $ corresponding to the kernel of $\rho$ is also a CM field. 
Assume that the Iwasawa main conjecture is true for any intermediate field $K$ of the extension $F_\rho /F$ and any branch character $\psi^\mathrm{gal}$ of $G_K$ factoring through $\mathrm{Gal}(F_\rho/K)$. 
Then there exists a unique element $L_{p, \Sigma_F }(M(\rho) )$ of $\widehat{\mathcal{O}}^\mathrm{ur} [[\Gamma_{F,\, \max}]]$ satisfying
\begin{equation} \label{equation:interpolation_property_TheoremA}
\!\!\!\!\!\! \frac{\eta^\mathrm{gal} (L_{p,\Sigma_F} (M(\rho )))}{\bigl(\Omega^{w_\eta \mathsf{t}+2\mathsf{r}_\eta}_{\mathrm{CM}, p,F} \bigr)^{r(\rho )}} = 
 i^{r(\rho)\lvert -w_\eta \mathsf{t}-\mathsf{r}_\eta\rvert} \prod_{v\in \Sigma_{F,p}} \! \!\!\mathrm{Eul}_v (\rho\otimes \eta,0) 
 \dfrac{\Lambda (\rho\otimes \eta ,0)}{\bigl(\Omega^{w_\eta\mathsf{t}+2\mathsf{r}_\eta}_{\mathrm{CM}, \infty, F} \bigr)^{r(\rho )}}
\end{equation}
with
\begin{align*}
&  \mathrm{Eul}_v(\rho\otimes \eta,0) 
=\begin{cases}
 L_{v^c}(\rho\otimes \eta,0)^{-1}L_v((\rho\otimes \eta)^\vee, 1)^{-1} & \text{if $\eta$ is unramified at $v$}, \\
\epsilon((\rho \otimes \eta)_v,\mathbf{e}_{F,v},dx_v)^{-1} & \text{if $\eta$ is ramified at $v$},  %need correction
\end{cases} \\
& \Lambda (\rho\otimes \eta ,s) = L_\infty(\rho\otimes \eta,s)L(\rho \otimes \eta, s)
\end{align*}
for each algebraic Hecke character $\eta$ satisfying the conditions (i) and (ii) appearing in {\em Theorem}. 
\end{thmA}

Note that, since both $F$ and $F_\rho$ are CM fields, each intermediate field $K$ of $F_\rho/F$ is also a CM field. Here $L(\rho\otimes \eta,s)$ is the Hasse--Weil $L$-function of the pure motive $M(\rho)\otimes M(\eta)$, and $\epsilon((\rho\otimes \eta)_v,\mathbf{e}_{F,v}, dx_v)$ is Deligne's local constant at $v\in \Sigma_{F,p}$; see \eqref{equation:L_factor} and \eqref{equation:local_constant} for their precise definitions. The completed Artin $L$-function $\Lambda(\rho\otimes \eta,s)$ is defined as the product of $L(\rho\otimes \eta,s)$ and the archimedean $L$-factor $L_\infty(\rho\otimes \eta,s)$, quite similarly to $\Lambda(\psi \eta,s)$. According to our construction of the $p$-adic Artin $L$-function $L_{p,\Sigma_F}(M(\rho))$, the Iwasawa main conjecture for $\rho$ shall be simultaneously verified.

\begin{thmA}[a part of Theorem \ref{theorem:integrality_of_CMp-adicL_Artin}] \label{thm:MainB}
Retain the notation and assumptions from Theorem A. Then 
\begin{align*}
 \mathrm{char}_{\widehat{\mathcal{O}}^\mathrm{ur}[[\Gamma_{F,\, \max}]]}(\mathrm{Gal}( M_{\Sigma_F}/F_\rho F_{\max} )_\rho \widehat{\otimes}_{\mathcal{O}}\widehat{\mathcal{O}}^\mathrm{ur}  )= (L_{p, \Sigma_F }(M(\rho) ))
\end{align*}
holds as an equality of ideals of $\widehat{\mathcal{O}}^\mathrm{ur}[[\Gamma_{F,\, \max}]]$, where the subscript $\rho$ denotes the isotypic quotient of type $\rho$.

\end{thmA}

Unfortunately we can hardly apply Hsieh's results \cite[Theorem 8.17, Theorem 8.18]{Hsieh14} to fulfill the assumption of Theorems~\ref{thm:MainA} and \ref{thm:MainB} on the (abelian) Iwasawa main conjecture for intermediate fields of $F_\rho/F$; see Lemma~\ref{lem:CM_ext} and Remark~\ref{rem:failure_ab_IMC} for details.

\medskip

\subsubsection*{{\bf Strategy of the proof}}
Let us briefly explain our strategy to prove Theorems A and B. 
One of the key ingredients is {\em Brauer's induction theorem} \cite[Theorem (15.9)]{CR81} appearing in representation theory of finite groups. 
Indeed, it provides us with a virtual decomposition  
\begin{align} \label{eq:Brauer_rho}
\rho= \sum^s_{j=1} a_j \, \mathrm{Ind}^{G_F}_{G_{F_j}}\,  \psi_j^\mathrm{gal}  
\end{align} 
of the Artin representation $\rho$ under consideration, where $a_j$ is an integer, $F_j$ is an intermediate field of $F_\rho /F$, and $\psi_j^\mathrm{gal}\colon G_{F_j} \rightarrow \mathbb{C}^\times$ is a character of finite order factoring through $\mathrm{Gal}(F_\rho/F_j)$ for each $j$.   
Using \eqref{eq:Brauer_rho}, we formally define our $p$-adic Artin $L$-function 
$ L_{p,\Sigma_F}(\rho )$ 
as the product 
$\prod_{j=1}^s 
\mathrm{pr}_j (L_{p,\Sigma_{F_j}}(\psi_j))^{a_j}$ 
up to multiples of minor constants,
 where $\psi_j$ denotes the algebraic Hecke character on $F_j$ corresponding to $\psi^\mathrm{gal}_j$ via global class field theory, and  $\mathrm{pr}_j$ is the composition
\begin{align*}
 \widehat{\mathcal{O}}^\mathrm{ur}[[\Gamma_{F_j,\,\max}]]\twoheadrightarrow  \widehat{\mathcal{O}}^\mathrm{ur}[[\mathrm{Gal}(F_jF_{\max}/F_j)]] \hookrightarrow \widehat{\mathcal{O}}^\mathrm{ur}[[\Gamma_{F,\, \max}]].
\end{align*}
 Note that  \eqref{eq:Brauer_rho} is just a {\em virtual} decomposition, and the integers $a_j$ can be {\em negative}.
Hence $L_{p,\Sigma_F}(\rho )$ is constructed as an element of the {\em field of fractions} of $\widehat{\mathcal{O}}^\mathrm{ur}[[\Gamma_{F,\, \max}]]$, and we do not know a priori whether it is contained in $\widehat{\mathcal{O}}^\mathrm{ur}[[\Gamma_{F,\, \max}]]$ or not; that is, the {\em integrality} of the $p$-adic Artin $L$-function is not clear at this stage. Meanwhile, we can verify that, on the algebraic side, the characteristic ideal has a quite similar decomposition
\begin{align}
 \mathrm{char}_{\widehat{\mathcal{O}}^\mathrm{ur}[[\Gamma_{F,\, \max}]]} & (\mathrm{Gal}( M_{\Sigma_F}/F_\rho F_{\max} )_\rho  \widehat{\otimes}_\mathcal{O}\widehat{\mathcal{O}}^\mathrm{ur} ) \label{equation:descent_argument_algebraic_side_introduction}\\
&= \displaystyle \prod_{j=1}^s 
\mathrm{pr}_j \bigl( \mathrm{char}_{\widehat{\mathcal{O}}^\mathrm{ur}[[\Gamma_{F_j,\,\max}]]}(\mathrm{Gal}( M_{\Sigma_F}/F_j F_{\max} )_{\psi_j}  \widehat{\otimes}_{\mathcal{O}}\widehat{\mathcal{O}}^\mathrm{ur}) \bigr)^{a_j}.  \nonumber
\end{align}
Although the product in the right-hand side of \eqref{equation:descent_argument_algebraic_side_introduction} is just a fractional ideal, 
 the left-hand side of \eqref{equation:descent_argument_algebraic_side_introduction} is indeed an {\em integral} ideal of $\widehat{\mathcal{O}}^\mathrm{ur}[[\Gamma_{F,\, \max}]]$ by definition. Therefore, admitting the Iwasawa main conjecture for {\em every} branch character $\psi_j^\mathrm{gal}$ appearing in \eqref{eq:Brauer_rho}, 
we may conclude that the fractional principal ideal 
$( L_{p,\Sigma_F}(\rho ))$ must be integral in $\widehat{\mathcal{O}}^\mathrm{ur}[[\Gamma_{F,\, \max}]]$. We also emphasise that, in order to check that our naively constructed $p$-adic Artin $L$-function $L_{p,\Sigma_F}(M(\rho) )$ satisfies the interpolation property \eqref{equation:interpolation_property_TheoremA}, we need to put together the interpolation formula of each $L_{p,\Sigma_{F_j}}(\psi_j)$ into the desired form 
\eqref{equation:interpolation_property_TheoremA}. This tedious task will be achieved in Section~\ref{section:matching}. 

\subsubsection*{{\bf Remarks in comparison with preceding results}}
The prototype of the present work is Ralph Greenberg's remarkable study of $p$-adic Artin $L$-functions over totally real number fields $F^+$ \cite{Greenberg83, Greenberg14}. 
He has constructed a one-variable $p$-adic $L$-function associated to a totally even Artin representation $\rho$ of $G_{F^+}$. The construction based upon the Brauer induction principle has already appeared in \cite{Greenberg83}, and we owe the main strategy of the proof of Theorems A and B to \cite{Greenberg83}. There appear, however, several noteworthy difficulties which did not matter in the case over a totally real field. 
We shall explain technical difficulties appearing in our generalisation of Greenberg's work to the case over a CM field. 

\par 
Firstly, the Brauer induction argument in Greenberg's case is covered by one-variable objects, but in the present work, 
the number of variables changes drastically throughout the Brauer induction argument. Here let us admit Leopoldt's conjecture for simplicity. 
Then $p$-adic Hecke $L$-functions $L_{p,\Sigma_{F^+_j}}(\psi_j)$ appearing in the construction of the $p$-adic Artin $L$-function are of one variable when $F^+_j$ is totally real. Contrastingly in our CM situation, 
the $p$-adic Hecke $L$-functions $L_{p,\Sigma_{F_j}}(\psi_j)$ are of $d[F_j:F]+1$ variables for each $j$. Thus, in order to take the product of $p$-adic Hecke $L$-functions, we need to specialise each $L_{p,\Sigma_{F_j}}(\psi_j)$ into the $(d+1)$-variable one via the projection $\mathrm{pr}_j$, but the characteristic ideals appearing in the right-hand side of \eqref{equation:descent_argument_algebraic_side_introduction} do not behave compatibly in general with respect to such specialisation procedures.
To overcome this difficulty, we adopt techniques developed in our previous work \cite{haraochiai} on the cyclotomic Iwasawa main conjecture for Hilbert modular forms with complex multiplication. In particular, inductive descent arguments developed there play crucial roles in the proof of the key decomposition formula \eqref{equation:descent_argument_algebraic_side_introduction} on the algebraic side; see Section~\ref{ssc:integrality} for details.
 
\par 
Secondly, verification of the desired interpolation property \eqref{equation:interpolation_property_TheoremA} becomes much harder than that dealt with in \cite{Greenberg83}. Indeed, only $L$-values concern the interpolation properties of the $p$-adic Artin $L$-functions for totally real fields, and thus one readily verifies it by just using inductivity of the Artin $L$-functions. 
However, there are several other factors appearing in the interpolation properties of the $p$-adic Artin $L$-functions for CM fields; for example, the gamma factor $L_\infty(\rho\otimes \eta,0)$, the modified complex CM periods $\Omega_{\mathrm{CM},\infty,F}$, the normalised $p$-adic CM periods $\Omega_{\mathrm{CM},p,F}$, the modified $p$-Euler factor $\mathrm{Eul}_v(\rho\otimes \eta,0)$ and so on. All of these factors do not appear in \cite{Greenberg83}, and the matching of these factors requires intricate computations.
Amongst them, the matching of Deligne's local constants (or generalised Gauss sums) requires quite delicate arguments and is not straightforward at all. To establish its matching, we settle in Appendix \ref{app:DH} an extension of classical Davenport--Hasse relation \cite[(0.8)]{DH} for Gauss sums to the case where the conductor of the multiplicative character is a {\em power} of a prime element (Theorem~\ref{theorem:D-H_relation_overW_n}), which seems to be of independent interest and will play important roles in the construction of $p$-adic Artin $L$-functions for more general motives in future works.

\subsubsection*{{\bf Notation}}
%
%notation on ideals
We mainly use the fraktur $\mathfrak{r}$ for the ring of integers 
of {\em an algebraic number field} (which is often regarded 
as the base field of a certain motive); 
the calligraphic $\mathcal{O}$ 
is kept to denote the ring of integers for {\em a $p$-adic field} 
(which is often regarded as the coefficient field of the $p$-adic 
realisation of a certain motive).
The absolute norm of a fractional ideal $\mathfrak{a}$ 
of an algebraic number field is denoted by $\mathcal{N}\mathfrak{a}$.
We fix an algebraic closure $\overline{\mathbb{Q}}$ of the rational number field $\mathbb{Q}$ and 
regard all algebraic number fields (that is, all finite extensions of $\mathbb{Q}$) as subfields of $\overline{\mathbb{Q}}$. 
%notation on adele, idele
For an algebraic number field $\mathsf{K}$, let  $\mathbb{A}_\mathsf{K}$ 
(resp.\ $\mathbb{A}_\mathsf{K}^\times$) denote the ring of ad\`eles 
(resp.\ the group of id\`eles) of $\mathsf{K}$. The finite part 
(resp.\ the archimedean part) of the ring of ad\`eles 
$\mathbb{A}_\mathsf{K}$ is denoted by $\mathbb{A}_{\mathsf{K},\mathrm{fin}}$ 
(resp.\ $\mathbb{A}_\mathsf{K}^\infty$). 

%additive characters
We shall fix notion on {\em the standard additive character} throughout this article. 
For each finite prime $v$ of $K$, we define $\mathbf{e}_{\mathsf{K},v}\colon  \mathsf{K}_{v} \longrightarrow \mathbb{C}^{\times}$ to be 
\begin{align} \label{eq:standard_additive}
\mathbf{e}_{\mathsf{K},v} (x )=\exp(-2\pi \sqrt{-1} \mathrm{Tr}_{\mathsf{K}/\mathbb{Q}}(\tilde{x}))
\end{align}
where $\tilde{x}$ is an arbitrary element of $\bigcup_{n=1}^{\infty} \mathfrak{P}_v^{-n}$ 
(regarded as a $\mathfrak{r}_\mathsf{K}$-submodule of $\mathsf{K}$)
such that $\tilde{x}-x$ is contained in the completion of the ring of integers of $\mathsf{K}$ at $v$. Here $\mathfrak{P}_v$ denotes the prime ideal of $\mathsf{K}$ corresponding to $v$.
 
%notation on p-adic fields
Let $\mathbb{C}_p$ be the $p$-adic completion of the fixed algebraic closure $\overline{\mathbb{Q}}_p$ of $\mathbb{Q}_p$ and $\mathcal{O}_{\mathbb{C}_p}$ its 
ring of integers. For a finite flat extension $\mathcal{O}$ of $\mathbb{Z}_p$, 
we use the symbol ${\widehat{\mathcal{O}}}^\mathrm{ur}$  for the composite ring ${\mathcal{O}} \widehat{\mathbb{Z}}_p^{\mathrm{ur}}$, where $\widehat{\mathbb{Z}}_p^{\mathrm{ur}}=W(\overline{\mathbb{F}}_p)$ denotes the ring of Witt vectors with coefficients in $\overline{\mathbb{F}}_p$. Throughout the present article, we fix an isomorphism $\boldsymbol{\iota}\colon \mathbb{C}\xrightarrow{\, \sim \,}\mathbb{C}_p$.
 
%remark on CFT
We here adopt {\em geometric} normalisation 
of global class field theory. Specifically, for a finite abelian extension $\mathsf{L/K}$ of algebraic number fields, the Artin %reciprocity %for arXiv version 
 map $(-,\mathsf{L/K}) \colon \mathbb{A}^\times_\mathsf{K} \rightarrow \mathrm{Gal}(\mathsf{L/K})$ 
is normalised so that it sends a uniformiser $\varpi_v$ at a finite place $v$ of $\mathsf{K}$ which does not ramify in  $\mathsf{L/K}$ 
to the {\em geometric} Frobenius element
$\mathrm{Frob}_v$ in $\mathrm{Gal}(\mathsf{L/K})$. In other words, 
the congruence %added%for arXiv version 
$a^{(\varpi_{v}, \mathsf{L/K})^{-1}}\equiv a^{q_{v}} \, (\text{mod}\, \mathfrak{P}_v)$ 
holds for each $a$ in $\mathfrak{r}_\mathsf{K}$ where $q_{v}$ denotes 
the cardinality of the residue field at $v$.

%
%sharp 
The absolute Galois group 
$\mathrm{Gal}(\overline{\mathbb{Q}}/\mathsf{K})$ of an algebraic number field $\mathsf{K}$ is denoted by
$G_\mathsf{K}$. For a place $v$ of $\mathsf{K}$, the decomposition group 
and the inertia group at $v$ are denoted by $D_v$ and $I_v$, respectively. 
For a (possibly infinite) abelian Galois extension $\mathsf{L/K}$ 
of $\mathsf{K}$ and the ring of integer $\mathcal{O}$ of 
a finite extension of $\mathbb{Q}_p$, we define 
$\mathcal{O}[[\mathrm{Gal}(\mathsf{L/K})]]^\sharp$ as the
free $\mathcal{O}[[\mathrm{Gal}(\mathsf{L/K})]]$-module of rank $1$
on which $G_\mathsf{K}$ acts via the universal tautological character
\begin{align*}
G_\mathsf{K} \rightarrow
 \mathcal{O}[[\mathrm{Gal}(\mathsf{L/K})]]^\times ; \ g \mapsto
 g\vert_\mathsf{L}.
\end{align*}
We also define $\widehat{\mathcal{O}}^\mathrm{ur}[[\mathrm{Gal}(\mathsf{L/K})]]^\sharp$ in the same manner. Finally, we let $M^\vee$ denote the  Pontrjagin dual  $\mathrm{Hom}_{\mathbb{Z}_p}(M,\mathbb{Q}_p/\mathbb{Z}_p)$ of a $\mathbb{Z}_p$-module $M$.

%%%%%%%%%%%%%%%%%%%%%%%%%%%%%%%%%%%%%%%%%%%%
\section{The $p$-adic Hecke $L$-functions for CM fields} \label{sc:Review_Hecke_L}
%%%
%%%%%%%%%%%% 

In this section, we construct the $p$-adic Hecke $L$-function $L_{p,\Sigma_F}(\psi)$ following Katz \cite{katz}, Hida and Tilouinen \cite{HT-katz}. 

%%%%%%%%%%%%%%%%%%%%%%%%%%%%%%%%%%%%%%%%%%%%%%
\subsection{General settings} \label{ssc:setting_Hecke}
%%%%%%%%%%%%%%%%%%%%%%%%%%%%%%%%%%%%%%%%%%%%%%

Let $p$ be an odd prime number and $F$ a CM field of degree $2d$ with the maximal totally real subfield $F^+$. We use the symbol $c$ for a unique generator of $\mathrm{Gal}(F/F^+)$, namely the complex conjugation on $F$. 
The composition of all $\mathbb{Z}_p$-extensions of $F$ in $\overline{\mathbb{Q}}$ is denoted as $F_{\max}$, and we write $\Gamma_{F,\, \max}$ for the Galois group of $F_{\max}/F$. Then, by using global class field theory, one readily observes  that $\Gamma_{F,\, \max}$ is a free $\mathbb{Z}_p$-module of rank equal to or greater than $d+1$; as is well known, the equality holds  if and only if the Leopoldt conjecture for $F$ and $p$ is true. 

Hereafter we impose the following two assumptions on $F$ and $p$;
\begin{itemize}[leftmargin=5em]
 \item[(unr$_{F,p}$)] the field $F$ is absolutely unramified
              at $(p)$;
 \item[(ord$_{F,p}$)] all places of $F^+$ lying above
              $(p)$ split completely in $F$.
\end{itemize}

By virtue of the ordinarity condition (ord$_{F,p}$), there
exists a {\em $p$-ordinary CM type} $\Sigma_{F}$ of
$F$ (also called a {\em $p$-adic CM type} in several literature),
 which is defined as a subset of the set $I_{F}$ of all embeddings of
 $F$ into the complex number field $\mathbb{C}$ satisfying the following two conditions:
\begin{itemize}
\item[--] we have $I_{F} =\Sigma_{F} \sqcup \Sigma_F^c$ (disjoint union) where 
$\Sigma_{F}^c$ is defined as $\{ \sigma\circ c \in I_F \mid \sigma \in \Sigma _{F}\}$; 
\item[--] we have $\{ \text{places of $F$ lying
above $p$} \} = \Sigma_{F,p} \sqcup \Sigma^c_{F,p}$
         (disjoint union) where $\Sigma_{F,p}$ denotes the set of prime ideals of $F$ induced by $p$-adic embeddings $\boldsymbol{\iota} \circ \sigma \colon F\hookrightarrow \mathbb{C}_p$ for 
all $\sigma$ in $\Sigma_{F}$, and $\Sigma^c_{F,p}$
is defined as $\Sigma^c_{F,p}:=\{\mathfrak{P}^c \mid \mathfrak{P}\in \Sigma_{F,p}\}$.
\end{itemize}
We take a $p$-ordinary CM type $\Sigma_{F}$ of $F$ and fix it once and for all.  

%%%%%%%%%%%%%%%%%%%%%%%%%%%%%%%%%%%%%%%%%%%%%%%%%%%%%%%%
\subsection{Period relations of CM periods} \label{ssc:period_relation}
%%%%%%%%%%%%%%%%%%%%%%%%%%%%%%%%%%%%%%%%%%%%%%%%%%%%%%%%

In this subsection, we recall the definition of
the complex and $p$-adic CM periods, which
appear in the interpolation formulae of the $p$-adic Hecke $L$-functions
for CM fields. For this purpose, we first construct a Hilbert--Blumenthal
abelian variety equipped with complex multiplication by
$\mathfrak{r}_{F}$ as follows. 

Consider a diagonal embedding $\mathfrak{r}_F\hookrightarrow \mathbb{C}^{\Sigma_{F}}\, ; x\mapsto (\sigma(x))_{\sigma \in \Sigma_F}$ of $\mathfrak{r}_F$ into $\mathbb{C}^{\Sigma_F}$ with respect to the fixed $p$-ordinary CM type $\Sigma_{F}$. 
The image $\Sigma_{F}(\mathfrak{r}_{F})$  of this embedding then forms a 
$\mathbb{Z}$-lattice of $\mathbb{C}^{\Sigma_{F}}$, and thus we can define a complex torus $\mathbb{C}^{\Sigma_{F}}/\Sigma_F (\mathfrak{r}_F)$. 
In order to regard
this complex torus as an abelian variety, we equip it with a {\em polarisation} by choosing any element $\delta$ of
$F$ satisfying the following three conditions:
\begin{enumerate}[label=(\arabic*${}_\delta$)]
 \item $\delta$ is relatively prime to $p$;
 \item $\delta$ is purely imaginary, that is,
       it satisfies $\delta^c=-\delta$;
\item the imaginary part $-i\sigma(\delta)$ of
      $\sigma(\delta)\in \mathbb{C}$ is positive for all $\sigma$ in
      $\Sigma_{F}$. Here $i$ denotes the imaginary unit satisfying $i^2=-1$.
\end{enumerate}

Using such $\delta$, we introduce an alternating pairing $\langle \  , \
\rangle_\delta:=(uv^c-u^cv)/2\delta$ on $\mathfrak{r}_{F}$.  
The existence of such an alternating pairing implies that there is an abelian variety $X_{\Sigma_F}(\mathfrak{r}_{F})_{/\mathbb{C}}$ over $\mathbb{C}$ 
whose $\mathbb{C}$-valued points are identified with the complex torus $\mathbb{C}^{\Sigma_{F}}/\Sigma_F (\mathfrak{r}_F)$. The details are as follows.
Let $\mathfrak{c}$ be a fractional ideal of $F^+$ defined as $\mathfrak{c}=\mathrm{Tr}_{F/F^+}(\mathfrak{d}_{F^+}/2\delta)^{-1}$. Here $\mathfrak{d}_{{F}^+}$ denotes the absolute different of $F^+$. Then $\mathfrak{c}$ is relatively prime to $p$ due to the assumptions (unr$_{F,p}$) and $(1_\delta)$, and thus there exists an abelian variety denoted by 
$X_{\Sigma_{F}}(\mathfrak{r}_{F})_{/\mathbb{C}}\otimes_{\mathfrak{r}_{F^+}} \mathfrak{c}$ over $\mathbb{C}$ 
whose $\mathbb{C}$-valued points are identified with the complex torus $\mathbb{C}^{\Sigma_{F}}/\Sigma_F (\mathfrak{c} \mathfrak{r}_F)$,
and the alternating pairing $\langle \  , \
\rangle_\delta$ induces a $\mathfrak{c}$-polarisation $\lambda_\delta \colon
X_{\Sigma_{F}}(\mathfrak{r}_{F})^t_{/\mathbb{C}} \xrightarrow{\sim}
X_{\Sigma_{F}}(\mathfrak{r}_{F})_{/\mathbb{C}}\otimes_{\mathfrak{r}_{F^+}} \mathfrak{c}$ where 
$X_{\Sigma_{F}}(\mathfrak{r}_{F})^t_{/\mathbb{C}}$ denotes the dual abelian variety of $X_{\Sigma_{F}}(\mathfrak{r}_{F})_{/\mathbb{C}}$. 
Note that $2\delta$ is chosen so that its image $2\delta_{\mathfrak{p}}$ under  $F \hookrightarrow F_\mathfrak{P}\xrightarrow{\, \sim \,} F_\mathfrak{p}^+$ generates the absolute different $\mathfrak{d}_{F^+_\mathfrak{p}}$ of $F^+_{\mathfrak{p}}$ for each place $\mathfrak{p}$ of $F^+$ lying above $p$ and a unique $\mathfrak{P}\in \Sigma_{F,p}$ satisfying $\mathfrak{P}\mid \mathfrak{p}$  (see \cite[(5.3.3) and Lemma~(5.7.35)]{katz} for details). 

Next recall that Katz and Hida--Tilouine have endowed the $\mathfrak{c}$-polarised abelian variety
$(X_{\Sigma_{F}}(\mathfrak{r}_F)_{/\mathbb{C}}, \lambda_\delta)$
with a {\em  $\Gamma_{00}(p^\infty)$-level structure} $i_\delta \colon (\mathfrak{d}_{F^+}^{-1}\otimes_\mathbb{Z} \mathbb{G}_m)[p^\infty]\hookrightarrow X(\mathfrak{r}_F)_{/\mathbb{C}}$ {\em with respect to $\delta$} induced by a composite map
\begin{align*}
&(\mathfrak{d}_{F^+}^{-1} \otimes_\mathbb{Z} \mathbb{G}_m)[p^\infty](\mathbb{C}) 
\cong 
\prod_{\mathfrak{p} \mid p\mathfrak{r}_{F^+} } \mathfrak{d}_{F^+_\mathfrak{p}}^{-1}\mathfrak{p}^{-\infty}  / \mathfrak{r}_{F^+_{\mathfrak{p}}} \xrightarrow{\times (2\delta_\mathfrak{p})_{\mathfrak{p}}} 
 \prod_{\mathfrak{p}\mid p\mathfrak{r}_{F^+}} \mathfrak{p}^{-\infty} / \mathfrak{r}_{F^+_{\mathfrak{p}}} 
\\
&\;   \overset{\sim\, }{\longrightarrow} 
\prod_{\mathfrak{P} \in \Sigma_{F,p}}
 \mathfrak{P}^{-\infty}/\mathfrak{r}_{F_{\mathfrak{P}}} 
 \overset{\sim\,}{\longrightarrow} \prod_{\mathfrak{P} \in \Sigma_{F,p}}
 \mathfrak{P}^{-\infty}/\mathfrak{r}_{F}
  \hookrightarrow \mathbb{C}^{\Sigma_F}/\Sigma_{F} (\mathfrak{r}_F )=X_{\Sigma_F}(\mathfrak{r}_F)(\mathbb{C})
\end{align*}
(see \cite[(5.1.11)--(5.1.18)]{katz}). Meanwhile we may canonically identify $\mathfrak{d}_{F^+_\mathfrak{p}}^{-1}$ with $\mathfrak{r}_{F^+_\mathfrak{p}}$ for every $\mathfrak{p}\mid p\mathfrak{r}_{F^+}$ under the assumption (unr${}_{F,p}$). Hence there is another $\Gamma_{00}(p^\infty)$-level structure $i_- \colon (\mathfrak{d}_{F^+}^{-1}\otimes_\mathbb{Z} \mathbb{G}_m)[p^\infty]\hookrightarrow X(\mathfrak{r}_F)_{/\mathbb{C}}$ defined as the composite
\begin{align*}
&(\mathfrak{d}_{F^+}^{-1} \otimes_\mathbb{Z} \mathbb{G}_m)[p^\infty](\mathbb{C}) 
\cong 
\prod_{\mathfrak{p} \mid p\mathfrak{r}_{F^+} } \mathfrak{d}_{F^+_\mathfrak{p}}^{-1}\mathfrak{p}^{-\infty}  / \mathfrak{r}_{F^+_{\mathfrak{p}}} \xrightarrow{\; \times(-1) \;}
 \prod_{\mathfrak{p}\mid p\mathfrak{r}_{F^+}} \mathfrak{p}^{-\infty} / \mathfrak{r}_{F^+_{\mathfrak{p}}} 
\\
&\;   \overset{\sim\, }{\longrightarrow} 
\prod_{\mathfrak{P} \in \Sigma_{F,p}}
 \mathfrak{P}^{-\infty}/\mathfrak{r}_{F_{\mathfrak{P}}} 
 \overset{\sim\,}{\longrightarrow} \prod_{\mathfrak{P} \in \Sigma_{F,p}}
 \mathfrak{P}^{-\infty}/\mathfrak{r}_{F}
  \hookrightarrow \mathbb{C}^{\Sigma_F}/\Sigma_{F} (\mathfrak{r}_F )=X_{\Sigma_F}(\mathfrak{r}_F)(\mathbb{C}).
\end{align*}
Both triplets
 $X_{\Sigma_{F}}(\mathfrak{r}_{F})_\delta=(X_{\Sigma_{F}}(\mathfrak{r}_F)_{/\mathbb{C}},
 \lambda_\delta, i_\delta)$ and $X_{\Sigma_F}(\mathfrak{r}_F)_-=(X_{\Sigma_{F}}(\mathfrak{r}_F)_{/\mathbb{C}},
 \lambda_\delta, i_-)$ then have  models
 $\mathcal{X}_{\Sigma_F}(\mathfrak{r}_{F})_\delta $ and  $\mathcal{X}_{\Sigma_F}(\mathfrak{r}_{F})_-$ over the valuation
 ring $\mathcal{W}=\overline{\mathbb{Q}}\cap \widehat{\mathbb{Z}}_p^{\mathrm{ur}}$
 due to the theory of complex multiplication combined with Serre and Tate's criterion for good reduction and the assumption (unr${}_{F,p}$). 

 The complex uniformisation $\Pi \colon \mathbb{C}^{\Sigma_F} \twoheadrightarrow X(\mathfrak{r}_F)(\mathbb{C})$ of $X_{\Sigma_F}(\mathfrak{r}_F)$, namely the natural quotient map $\mathbb{C}^{\Sigma_F}\twoheadrightarrow \mathbb{C}^{\Sigma_F}/\Sigma_F(\mathfrak{r}_F)$, induces an isomorphism
 \begin{align*}
 \Pi^\ast \colon \mathrm{Fil}^1H^1_{\mathrm{dR}}(X_{\Sigma_{F}}(\mathfrak{r}_{F})_{/\mathbb{C}})
  \overset{\sim\,}{\longrightarrow} \bigoplus_{\sigma \in \Sigma_F} \mathbb{C}\,  {\rm d}u_\sigma, 
 \end{align*}
 where $(u_\sigma)_{\sigma \in \Sigma_F}$ is %denotes %for arXiv version
the coordinate of $\mathbb{C}^{\Sigma_F}$. Now let us define  $\omega_{\mathrm{trans}}(\mathfrak{r}_F)$ as
$(\Pi^{\ast})^{-1} \left( \sum_{\sigma \in \Sigma_F} {\rm d}u_\sigma \right)$.

On the other hand, the $p$-part of the $\Gamma_{00}(p^\infty)$-level structure $i_?$ (for $?\in \{\delta,-\}$) induces an isomorphism 
 $\hat{i}_? \colon ((\mathfrak{d}_{F^+}^{-1}\otimes_\mathbb{Z} \mathbb{G}_m)_{/\widehat{\mathcal{O}}^{\mathrm{ur}}})^{\wedge} \xrightarrow{\sim} (\mathcal{X}_{\Sigma_F}(\mathfrak{r}_{F})
 _{/\widehat{\mathcal{O}}^{\mathrm{ur}}})^{\wedge}$ 
 between the formal completions along the identity sections over $\widehat{\mathcal{O}}^{\mathrm{ur}}$ where 
 $\mathcal{X}_{\Sigma_F}(\mathfrak{r}_{F})
 _{/\widehat{\mathcal{O}}^{\mathrm{ur}}}$ denotes the base extension of $ \mathcal{X}_{\Sigma_F}(\mathfrak{r}_{F})$  to $\widehat{\mathcal{O}}^{\mathrm{ur}}$. Note that, under the assumption (unr${}_{F,p}$), the formal scheme $((\mathfrak{d}_{F^+}^{-1}\otimes_\mathbb{Z} \mathbb{G}_m)_{/\widehat{\mathcal{O}}^{\mathrm{ur}}})^{\wedge}$ is decomposed into the direct product %added %for arXiv version 
$\prod_{\sigma\in \Sigma_F}{\mathbb{G}_m^\wedge}_{/\widehat{\mathcal{O}}^{\mathrm{ur}}}$ corresponding to a natural isomorphism $\mathfrak{d}_{F^+}^{-1}\otimes_\mathbb{Z}\widehat{\mathcal{O}}^\mathrm{ur} \xrightarrow{\, \sim \,} \prod_{\sigma\in \Sigma_F} \widehat{\mathcal{O}}^\mathrm{ur} \, ; x\otimes 1 \mapsto (\boldsymbol{\iota}\circ \sigma(x))_{\sigma\in \Sigma_F}$. In summary, the isomorphism $\hat{i}_?$ above induces
 \begin{align*}
 \hat{i}_?^\ast \colon \mathrm{Fil}^1H^1_{\mathrm{dR}}(\mathcal{X}_{\Sigma_{F}}(\mathfrak{r}_{F})_{/\widehat{\mathcal{O}}^{\mathrm{ur}}})
  \xrightarrow{\, \sim \, } 
  \bigoplus_{\sigma \in \Sigma_{F}} \widehat{\mathcal{O}}^{\mathrm{ur}}\dfrac{{\rm d}T_\sigma}{T_\sigma}
 \end{align*}
 where $T_\sigma$ denotes the formal parameter of the component ${\mathbb{G}_m^\wedge}_{/\widehat{\mathcal{O}}^\mathrm{ur}}$ corresponding to $\sigma\in \Sigma_F$. Define  $\omega^?_{\mathrm{can}}(\mathfrak{r}_F)$ to be $\omega^?_{\mathrm{can}}(\mathfrak{r}_F)=
 (\hat{i}_?^{\ast})^{-1} \left(
 \sum_{\sigma \in \Sigma_F} \dfrac{{\rm d}T_\sigma}{T_\sigma}\right)$.

\begin{defn}\label{definition:complex_CM_period&p-adic_CM_period}
Let us choose and fix a basis $\omega$ of $\mathrm{Fil}^1 H^1_{\mathrm{dR}}(\mathcal{X}_{\Sigma_{F}}(\mathfrak{r}_{F})_{/\mathcal{W}} )$, which is 
a free $\mathfrak{r}_{F^+} \otimes_{\mathbb{Z}}\mathcal{W}$-module of rank one. 
\begin{enumerate}[label=(\arabic*)]
\item 
We define a {\em complex CM period} 
\begin{align*}
C_{\mathrm{CM}, \infty ,F}
      =\bigl(C_{\mathrm{CM},\infty,F,\sigma}\bigr)_{\sigma\in
      \Sigma_{F}}\in \bigl(
      \mathfrak{r}_{F^+}\otimes_{\mathbb{Z}}\mathbb{C}\bigr)^\times 
\end{align*}
to be a constant satisfying 
\begin{align*}
 \omega = C_{\mathrm{CM}, \infty ,F} \, \omega_{\mathrm{trans}}(\mathfrak{r}_F)
\end{align*}
as an equation in 
      $\mathrm{Fil}^1H^1_{\mathrm{dR}}(X_{\Sigma_{F}}(\mathfrak{r}_{F})_{/\mathbb{C}})$. We also define a {\em modified complex CM period} $\Omega_{\mathrm{CM},\infty,F}=(\Omega_{\mathrm{CM},\infty,F,\sigma})_{\sigma\in \Sigma_F}$ by setting $\Omega_{\mathrm{CM},\infty, F,\sigma}:=(2\pi i)^{-1}C_{\mathrm{CM},\infty,F,\sigma}$ for %each %omit %for arXiv version
$\sigma\in \Sigma_F$.

\item 
We define a {\em $\delta$-modified $p$-adic CM period} 
\begin{align*}
 C_{\mathrm{CM},p,F}=C^\delta_{\mathrm{CM}, p ,F} =\bigl(C^\delta_{\mathrm{CM},p,F,\sigma}\bigr)_{\sigma\in
      \Sigma_{F}}\in \bigl(
       \mathfrak{r}_{F^+}\otimes_{\mathbb{Z}}\widehat{\mathcal{O}}^{\mathrm{ur}}\bigr)^\times
\end{align*} 
 to be a constant satisfying
\begin{align*}
 \omega = C^\delta_{\mathrm{CM}, p, F} \,  \omega^\delta_{\mathrm{can}}(\mathfrak{r}_F)
\end{align*}
as an equation in $\mathrm{Fil}^1H^1_{\mathrm{dR}}(\mathcal{X}_{\Sigma_{F}}(\mathfrak{r}_{F})_{/\widehat{\mathcal{O}}^{\mathrm{ur}}})$. We also define a {\em normalised $p$-adic CM period} $\Omega_{\mathrm{CM},p,F}=(\Omega_{\mathrm{CM},p,F,\sigma})_{\sigma\in \Sigma_F}$ as a constant satisfying $\omega=\Omega_{\mathrm{CM},p,F}\omega^-_\mathrm{can}(\mathfrak{r}_F)$.
\end{enumerate}

Note that both of the complex and  $p$-adic periods {\em do depend} on the choice of an
 $(\mathfrak{r}_{F^+}  \otimes_{\mathbb{Z}}\mathcal{W}) $-basis 
$\omega$ of $\mathrm{Fil}^1 H^1_{\mathrm{dR}}(\mathcal{X}_{\Sigma_{F}}(\mathfrak{r}_F)_{/\mathcal{W}})$,
 but the ``ratio of them'' is independent of $\omega$; 
 namely, when $\omega$ is replaced by $a\omega$ for $a\in
 \bigl(\mathfrak{r}_{F^+} \otimes_\mathbb{Z} \mathcal{W}\bigr)^\times$,
 the resulting (complex and $p$-adic) periods are both multiplied by
 the {\em same} constant $a$. 
\end{defn}

\begin{rem}[On CM periods]
Many authors including Katz \cite{katz} and Hida--Tilouine \cite{HT-aIMC} adopt the pair of CM periods $(C_{\mathrm{CM},\infty,F}, C_{\mathrm{CM},p,F})$, whereas we %shall %for arXiv version
use $(\Omega_{\mathrm{CM},\infty,F}, \Omega_{\mathrm{CM},p,F})$ to simplify the interpolation formula of the $p$-adic $L$-function under the assumption (unr${}_{F,p}$). As we shall mention in Remark~\ref{rem:CPC}, the normalisation of $\Omega_{\mathrm{CM},\infty,F}$ is initiated by Coates and Perrin-Riou's conjecture \cite{cp89,coa89}. The normalised $p$-adic CM period $\Omega_{\mathrm{CM},p,F}$ essentially appears in \cite[1.5]{deShalit} where the assumption (unr${}_{F,p}$) is also admitted; indeed $\Omega_{\mathrm{CM},p,F}$ coincides with de Shalit's $-\Omega_p(\Phi)$. We should remark that Chida and Hsieh also consider in \cite[Proposition~3.4]{CH23} a similar modification of the CM periods when $F$ is an imaginary quadratic field. We also emphasise here that, although we implicitly have to choose an auxiliary element $\delta$ to endow the complex torus $X_{\Sigma_F}(\mathfrak{r}_F)$ with a polarisation $\lambda_\delta$, the normalised $p$-adic CM period $\Omega_{\mathrm{CM},p,F}$ {\em does not depend} on the choice of $\delta$, while $C_{\mathrm{CM},\infty,F}=C^\delta_{\mathrm{CM},p,F}$ {\em does depend} on it; by construction, they are related via the equality 
\begin{align} \label{eq:Cp_vs_Omegap}
 C^\delta_{\mathrm{CM},p,F,\sigma}=\boldsymbol{\iota}\circ \sigma(-2\delta)\Omega_{\mathrm{CM},p,F,\sigma} \quad \text{ for every } \sigma\in \Sigma_F.
\end{align} 
\end{rem}

Later we need to compare (complex and $p$-adic) CM periods among various CM fields. Let $F'$ be a CM field which is absolutely unramified at $p$ and contains $F$. Then one readily checks that $\Sigma_{F'}=\{\sigma'\in I_{F'} \mid \sigma'\vert_{F}\in \Sigma_F\}$ is indeed a $p$-ordinary CM type of $F'$. We say that $\Sigma_{F'}$ is {\em induced from $\Sigma_F$}. Now let us choose and fix an $(\mathfrak{r}_{F^+}  \otimes_{\mathbb{Z}}\mathcal{W}) $-basis $\omega$ of $\mathrm{Fil}^1 H^1_{\mathrm{dR}}(\mathcal{X}_{\Sigma}(\mathfrak{r}_{F})_{/\mathcal{W}})$.
Since $\mathrm{Fil}^1 H^1_{\mathrm{dR}}(\mathcal{X}_{\Sigma}(\mathfrak{r}_{F'})_{/\mathcal{W}})$ is isomorphic to
$\mathrm{Fil}^1 H^1_{\mathrm{dR}}(\mathcal{X}_{\Sigma}(\mathfrak{r}_{F})_{/\mathcal{W}} ) \otimes_{\mathfrak{r}_{F}}\mathfrak{r}_{F'}$, 
we can choose $\omega \otimes_{\mathfrak{r}_{F}}\mathfrak{r}_{F'}$ as an $(\mathfrak{r}_{F^{\prime, +}}  \otimes_{\mathbb{Z}}\mathcal{W}) $-basis of $\mathrm{Fil}^1 H^1_{\mathrm{dR}}(\mathcal{X}_{\Sigma}(\mathfrak{r}_{F'})_{/\mathcal{W}} )$.  This observation implies the following period relation, which plays an important role in the proof of Theorem \ref{theorem:CMp-adicL_Artin}.   

\begin{lem}[Period relation, see also {\cite[1.3 (ii)]{deShalit}}]\label{lemma:deShalit}
Let $F$, $F'$, $\omega$ and $\omega\otimes_{\mathfrak{r}_F}\mathfrak{r}_{F'}$ be as above. For each $?\in \{\infty, p\}$, consider the $($complex or $p$-adic$)$ CM periods $\Omega_{\mathrm{CM}, ? ,F}=\bigl(\Omega_{\mathrm{CM}, ? ,F, \sigma}\bigr)_{\sigma \in \Sigma_F} $ for $F$ and $\Omega_{\mathrm{CM}, ? ,F'}=\bigl(\Omega_{\mathrm{CM}, ? ,F', \sigma'}\bigr)_{\sigma'\in \Sigma_{F'}}$ for $F'$ defined with respect to $\omega$ and $\omega\otimes_{\mathfrak{r}_F}\mathfrak{r}_{F'}$, as in Definition $\ref{definition:complex_CM_period&p-adic_CM_period}$. Then we have an equality $\Omega_{\mathrm{CM}, ? ,F', \sigma'}  = \Omega_{\mathrm{CM}, ? ,F, \sigma' \vert_F}$ for every $\sigma'\in \Sigma_{F'}$. 
\end{lem}

%%%%%%%%%%%%%%%%%%%%%%%%%%%%%%%%%%%%%%%%%%%%%%%%%%%
\subsection{Construction of the $p$-adic Hecke $L$-functions for CM fields} \label{ssc:p-adic_Hecke}
%%%%%%%%%%%%%%%%%%%%%%%%%%%%%%%%%%%%%%%%%%%%%%%%%%%

Before stating the interpolation formulae of $p$-adic Hecke $L$-functions for CM fields, we here make a remark concerning {\em purity} of infinity types of algebraic Hecke characters. Let $\eta =(\eta_v)_v \colon \mathbb{A}_F^\times/F^\times \rightarrow \overline{\mathbb{Q}}^\times$ be an algebraic Hecke character of a CM field $F$, and $\Sigma_F$ a $p$-ordinary CM type of $F$. Then there exist an integer $w_\eta\in \mathbb{Z}$ and an integer-valued vector $\mathsf{r}_\eta=(r_{\eta,\sigma})_{\sigma\in \Sigma_F}\in \mathbb{Z}^{\Sigma_F}$ satisfying 
\begin{align} \label{eq:eta_infinity_type}
 \eta_\infty(x_\infty) &=\prod_{\sigma \in \Sigma_F} x_\sigma^{w_\eta+r_{\eta,\sigma}} \overline{x}_\sigma^{-r_{\eta,\sigma}} \qquad \text{for } x_\infty=(x_\sigma)_{\sigma\in \Sigma_F} \in F\otimes_{\mathbb{Q}}\mathbb{R}\cong \mathbb{C}^{\Sigma_{F}}
\end{align}
where $\bar{\cdot}$ is the complex conjugation in $\mathbb{C}$. We refer to  $(w_\eta, \mathsf{r}_\eta) \in \mathbb{Z}\times \mathbb{Z}^{\Sigma_F}$ as the {\em infinity type} of $\eta$. We write for $\eta^\mathrm{gal}\colon G_F\rightarrow \mathbb{C}_p^\times$ the continuous character of $G_F$ corresponding to $\eta$ via global field theory.

\begin{rem}[multi-index notation] \label{rem:multi_index}
 To lighten the notation, we use the following {\em multi-index notation} in Theorem~\ref{thm:KHT}. Let $\mathsf{a}=(a_\sigma)_{\sigma\in \Sigma_F}\in \mathbb{Z}^{\Sigma_F}$ be an integer-valued vector. Then we put $\lvert \mathsf{a}\rvert:=\sum_{\sigma \in \Sigma_F} a_\sigma$, and for a $\mathbb{C}$-valued or $\widehat{\mathcal{O}}^\mathrm{ur}$-valued vector $\mathsf{w}=(w_\sigma)_{\sigma\in \Sigma_F}$, we write  $\mathsf{w}^{\mathsf{a}}$ for the product $\prod_{\sigma\in \Sigma_F} w_\sigma^{a_\sigma}$. 
Finally let $\mathsf{t}\in \Sigma_F$ denote a particular vector whose components are all equal to $1$.  
\end{rem}

Recall from Section~\ref{sc:Introduction} that we call  $\psi^\mathrm{gal}\colon G_F\rightarrow \mathbb{C}^\times$ a {\em branch character} if it is a finite character at most tamely ramified at every prime ideal of $F$ lying above $(p)$, and the field $F_\psi$ corresponding to the kernel of $\psi^\mathrm{gal}$ is linearly disjoint from $F_{\max}$ over $F$. Let $\mathcal{O}$ be a finite flat extension of $\mathbb{Z}_p$ containing the image of $\psi^\mathrm{gal}$ and $\psi\colon \mathbb{A}_F^\times/F^\times \rightarrow \overline{\mathbb{Q}}^\times$ the algebraic Hecke character corresponding to $\psi^\mathrm{gal}$ via global class field theory. Let $L(\psi\eta,s)$ be  the complex Hecke $L$-function of $\psi\eta$ defined by $\prod_{\mathfrak{l}\nmid \mathfrak{f}_\eta} (1-(\psi\eta)_{\mathfrak{l}}(\varpi_\mathfrak{l})\mathcal{N}\mathfrak{l}^{-s})^{-1}$ for $\mathrm{Re}\, s>1$, where we write $\varpi_\mathfrak{l}$ for a uniformiser of $F_\mathfrak{l}$. As is well known, it is meromorphically continued to the whole complex plane $\mathbb{C}$  with a possible simple pole at $s=1$.  
To consider the completed $L$-function, let us introduce the archimedean $L$-factors (or the gamma factors) $L_\infty(\psi\eta,s):=L_\infty(\mathrm{Res}_{F/\mathbb{Q}}\, M(\psi\eta))$ of $L(\psi\eta,s)$. Suppose that the algebraic Hecke character $\eta$ has the infinity type $(w_\eta, \mathsf{r}_\eta)$; namely suppose that $\eta_\infty$ satisfies \eqref{eq:eta_infinity_type}. On the Hodge realisation  $H_{\mathrm{B},\sigma}(M(\psi\eta)_{/F})_\mathbb{C}$ of the pure motive associated to $\psi\eta$ with respect to $\sigma\in \Sigma_F$, the real-valued points of Deligne's torus $\mathbb{S}_m(\mathbb{R})=\mathbb{C}^\times$ acts via $z\mapsto z^{w_\eta+r_{\eta,\sigma}}\bar{z}^{-r_{\eta,\sigma}}$; this implies that the Hodge type of $H_{\mathrm{B},\sigma}(M(\eta)_{/F})_\mathbb{C}$ is $(-w_\eta-r_{\eta,\sigma}, r_{\eta,\sigma})$. Here we adopt the convention in \cite[Remark 3.3]{Deligne_Hodge} concerning the Hodge types. The archimedean local $L$-factor $L_\infty(\psi\eta, s):=\prod_{\sigma\in \Sigma_F}L_\sigma(\psi\eta,s)$ of $L(\eta,s)$ is thus described as 
\begin{align} \label{eq:arch_L_eta}
L_\infty(\psi\eta,s)&= \prod_{\sigma \in \Sigma_F} \Gamma_\mathbb{C}(s+w_\eta+r_{\eta,\sigma}), \qquad \Gamma_\mathbb{C}(s)=2\cdot (2\pi)^{-s}\Gamma(s) 
\end{align}
where $\Gamma(s)$ denotes the usual gamma function; see \cite[Section 5.3]{Deligne} for details. The completed Hecke $L$-function is then defined as $\Lambda(\psi\eta,s):=L_\infty(\psi\eta,s)L(\psi\eta,s)$.

We are now ready to state the existence theorem of the $p$-adic Hecke $L$-function for CM fields, which is deduced from the result of Katz \cite{katz} and Hida--Tilouine \cite{HT-katz}. For a $p$-adic place $v\in \Sigma_{F,p}$, let $\epsilon ((\psi\eta)_v,\mathbf{e}_{F,v},dx_v)$ denote Tate's local constant
 with respect to the standard additive character $\mathbf{e}_{F,v}$ defined as \eqref{eq:standard_additive} and a unique Haar measure $dx_v$  on $F_v$ normalised so that the volume of the ring of integers of $F_v$ equals $1$.

\begin{thm}\label{thm:KHT}

Let $p$ be an odd prime number, ${F}$ a CM field of degree $2d$ and $F^+$ the maximal totally real subfield of $F$. Assume that 
${F}$, ${F}^+$ and $p$ satisfy both  {\upshape (unr$_{F,p}$)} and {\upshape
 (ord$_{F,p}$)}. 
Let $\psi^\mathrm{gal} \colon \mathbb{A}^\times_F/F^\times \rightarrow \mathcal{O}^\times$ be a branch character at most tamely ramified at every prime ideal of $F$ lying above $(p)$. 
Then there exists a unique element $L_{p, \Sigma_F}(\psi )$ of $\widehat{\mathcal{O}}^\mathrm{ur} [[\Gamma_{F,\, \max}]]$ satisfying
\begin{equation} \label{eq:p-adic_Hecke}
 \frac{\eta^{\mathrm{gal}}(L_{p, \Sigma_{F}}(\psi))}
{\Omega_{\mathrm{CM},p,F}^{w_\eta \mathsf{t}+2\mathsf{r}_\eta}} = \dfrac{(\mathfrak{r}_F^\times :\mathfrak{r}_{F^+}^\times)}{2^d\sqrt{\lvert D_{F^+}\rvert}}
i^{\lvert -w_\eta \mathsf{t}-\mathsf{r}_\eta\rvert }
\prod_{v \in \Sigma_{F,p}} \mathrm{Eul}_v(\psi\eta,0)
\frac{\Lambda (\psi\eta ,0 )}{\Omega_{\mathrm{CM},\infty,F}^{w_\eta \mathsf{t}+2\mathsf{r}_\eta}}
\end{equation}
with
\begin{align*}
& \mathrm{Eul}_v(\psi\eta,0)=
\begin{cases}
 L_{v^c}(\psi \eta,0)^{-1}L_v((\psi\eta)^\vee,1)^{-1} & \text{if $\psi\eta$ is unramified at $v$}, \\
\epsilon((\psi \eta)_v,\mathbf{e}_{F,v},dx_v)^{-1} & \text{if $\psi\eta$ is ramified at $v$},
\end{cases} 
\end{align*}
for each algebraic Hecke character $\eta\colon \mathbb{A}_F^\times/F^\times \rightarrow \overline{\mathbb{Q}}^\times$ satisfying the following two conditions$:$
\begin{enumerate}[label={\upshape (\roman*)}]
\item  
the Galois character $\eta^{\mathrm{gal}}\colon G_F\longrightarrow \mathbb{C}_p^\times$ 
corresponding to $\eta$ factors through $\Gamma_{F,\, \max};$ 
\item the infinity type $(w_\eta,\mathsf{r}_\eta)\in \mathbb{Z}\times \mathbb{Z}^{\Sigma_F}$ of $\eta$ satisfies both $-w_\eta-r_{\eta,\sigma}\leq -1$ and $r_{\eta,\sigma}\geq 0$ for every $\sigma\in \Sigma_F$. 
\end{enumerate}
\end{thm}

\begin{proof}
 Let $\mathfrak{f}^{(p)}$ denote the prime-to-$p$ part of the conductor of $\psi^\mathrm{gal}$, and decompose $\mathfrak{f}^{(p)}$ into the product $\mathfrak{FF}_c\mathfrak{I}$ of integral ideals satisfying the following three conditions; 
\begin{itemize}
 \item both $\mathfrak{F}$ and $\mathfrak{F}_c$ are the products of prime ideals which split completely over $F^+$;
\item $\mathfrak{I}$ is the product of prime ideals which inert or ramify over $F^+$;
\item $\mathfrak{F}$ and $\mathfrak{F}_c$ are relatively prime and satisfy $\mathfrak{F}_c^c\supset \mathfrak{F}$. 
\end{itemize}
Furthermore choose a purely imaginary element $\delta$ of ${F}$ so that it is relatively prime to the conductor of $\psi^\mathrm{gal}$ and satisfies all the conditions $(1_\delta)$, $(2_\delta)$ and $(3_\delta)$ at the beginning of Section~$\ref{ssc:period_relation}$. 
Fixing such a decomposition of $\mathfrak{f}^{(p)}$ and $\delta$, Hida and Tilouine have constructed in \cite[Theorem~II]{HT-katz} the $p$-adic $L$-function $L_{p,\Sigma_F,\delta}^{\mathrm{KHT}}(F)$ as a unique element of $\widehat{\mathcal{O}}^{\mathrm{ur}}[[\mathrm{Gal}(F_{\mathfrak{f}^{(p)}p^\infty}/F)]]$, where $F_{\mathfrak{f}^{(p)}p^\infty}$ denotes the ray class field modulo $\mathfrak{f}^{(p)}p^\infty$ of $F$. Meanwhile, for each place $v\mid p\mathfrak{r}_F$ of $F$, let $\delta_v$ denote the image of $\delta \in F$ into the $v$-adic completion $F_v$ of $F$, which we identify with the $v\vert_{F^+}$-adic completion $F^+_{v\vert_{F^+}}$ of $F^+$. Then, since $2\delta_v$ generates the absolute different of $F^+_{v\vert_{F^+}}$ due to \cite[Lemma (5.7.35)]{katz}, it is a $v$-adic unit under the assumption (unr${}_{F,p}$). Let $\mathfrak{f}^{(p)}=\prod_{v \nmid p\infty} \mathfrak{l}_v^{e_v(\mathfrak{f}^{(p)})}$ be the prime ideal decomposition of $\mathfrak{f}^{(p)}$ and set $U_{\mathfrak{f}^{(p)}p^k}:=\prod_{v\nmid p\infty}U_{F_v}^{(e_v(\mathfrak{f}^{(p)}))}\times \prod_{v\mid p\mathfrak{r}_F}U_{F_v}^{(k)}\times \prod_{v\mid \infty} \mathbb{C}^\times$ for each $k\geq 1$, where $U_{F_v}^{(n)}$ denotes the $n$-th higher unit group of $F_v$ for every natural number $n$. We then define $\gamma_{2\delta}^{\Sigma_F}$ as the image of $(2\delta_v)_{v\in \Sigma_{F,p}}$ under the composite map
\begin{align} \label{eq:composite}
 \prod_{v\in \Sigma_{F,p}}\mathfrak{r}_{F_v}^\times \hookrightarrow \mathbb{A}_F^\times \longrightarrow \varprojlim_{k\rightarrow\infty} \mathbb{A}_F^\times/ F^\times U_{\mathfrak{f}^{(p)}p^k} &\xrightarrow{\; \sim \;} \mathrm{Gal}(F_{\mathfrak{f}^{(p)}p^\infty}/F), 
\end{align} 
where the third isomorphism is induced by the global Artin reciprocity map due to global class field theory. 
Now let us consider the $\psi^\mathrm{gal}$-twisting map  
\begin{align*}
 \mathrm{Tw}_{\psi^\mathrm{gal}} \colon \widehat{\mathcal{O}}^{\mathrm{ur}}[[\mathrm{Gal}(F_{\mathfrak{f}^{(p)} p^\infty}/F)]]\rightarrow \widehat{\mathcal{O}}^{\mathrm{ur}}[[\Gamma_{F,\max}]] ; \, g\mapsto \psi^\mathrm{gal}(g) \overline{g}
\end{align*}
where $\overline{g}$ denotes the image of $g$ under the natural surjection $\mathrm{Gal}(F_{\mathfrak{f}^{(p)}p^\infty}/F)\twoheadrightarrow \Gamma_{F,\max}$, and define $L_{p,\Sigma_{F}}(\psi) $ to be $\mathrm{Tw}_{\psi^{\mathrm{gal}}}\bigl(\bigl(\gamma_{2\delta}^{\Sigma_F}\bigr)^{-1}L^{\mathrm{KHT}}_{p,\Sigma_F,\delta}\bigr) \in \widehat{\mathcal{O}}^{\mathrm{ur}}[[\Gamma_{F,\,\max}]]$. We shall check in the rest of the proof that  $L_{p,\Sigma_F}(\psi)$ satisfies the desired interpolation formula \eqref{eq:p-adic_Hecke}, 
from which the uniqueness of $L_{p,\Sigma_F}(\psi)$ readily follows; see \cite[Proposition (4.1.2)]{katz}.

Let $\eta \colon \mathbb{A}_F^\times/F^\times \rightarrow \overline{\mathbb{Q}}^\times$ be an algebraic Hecke character satisfying (i) and (ii) of the statement; we then have  
\begin{align*}
 \eta^\mathrm{gal}\bigl(\mathrm{Tw}_{\psi^\mathrm{gal}}\bigl(\gamma_{2\delta}^{\Sigma_F}\bigr)^{-1}\bigr)=\prod_{v\in \Sigma_{F,p}}(\psi\eta)_v(2\delta_v)^{-1} \prod_{\sigma\in \Sigma_F} (\boldsymbol{\iota}\circ \sigma (2\delta))^{-w_\eta-r_{\eta,\sigma}}
\end{align*}
by the construction of $\gamma^{\Sigma_F}_{2\delta}$ and local-global compatibility of the Artin reciprocity maps; note that the right-hand side is none other than evaluation of the $p$-adic avatar of $\psi\eta$ at $\bigl((2\delta_v)^{-1}\bigr)_{v\in \Sigma_{F,p}}$ (see \cite[Section~2.1.1]{haraochiai} for details). 
Combining this with \eqref{eq:Cp_vs_Omegap}, we can calculate as
\begin{align*}
 \dfrac{\eta^{\mathrm{gal}}(L_{p,\Sigma_F}(\psi))}{\Omega_{\mathrm{CM},p,F}^{w_\eta \mathsf{t}+2\mathsf{r}_\eta}} &= (-1)^{w_\eta d}\prod_{v\in \Sigma_{F,p}}(\psi\eta)_v(2\delta_v)^{-1} \prod_{\sigma\in \Sigma_F}\bigl(\boldsymbol{\iota}\circ \sigma(2\delta)\bigr)^{r_{\eta,\sigma}} \dfrac{(\psi\eta)^{\mathrm{gal}}(L^{\mathrm{KHT}}_{p,\Sigma_F,\delta}(F))}{C^{w_\eta\mathsf{t}+2\mathsf{r}_\eta}_{\mathrm{CM},p,F}}. 
\end{align*}
Applying the interpolation formula of $L^{\mathrm{KHT}}_{p,\Sigma_F}(F)$ proposed at 
\cite[Theorem II]{HT-katz} at the right-hand side, we obtain 
\begin{align*}
 \dfrac{\eta^{\mathrm{gal}}(L_{p,\Sigma_F}(\psi))}{\Omega_{\mathrm{CM},p,F}^{w_\eta \mathsf{t}+2\mathsf{r}_\eta}} &=(\mathfrak{r}_F^\times : \mathfrak{r}_{F^+}^\times) \dfrac{(-1)^{w_\eta d}}{\sqrt{\lvert D_{F^+}\rvert} } (-1)^{w_\eta d} \prod_{v\in \Sigma_{F,p}}(\psi\eta)_v(2\delta_v)^{-1} \prod_{\sigma\in \Sigma_F} \sigma(2\delta)^{r_{\eta,\sigma}} \\
& \qquad \cdot \left\{\prod_{v\in \Sigma_{F,p}} \mathrm{Eul}_v^\delta(\psi\eta,0) \prod_{\sigma\in \Sigma_F} \dfrac{\pi^{r_{\eta,\sigma}}\Gamma(w_\eta+r_{\eta,\sigma})}{\bigl(-i\sigma(\delta)\bigr)^{r_{\eta,\sigma}}}\right\}  \dfrac{L(\psi\eta,0)}{C^{w_\eta\mathsf{t}+2\mathsf{r}_{\eta}}_{\mathrm{CM},\infty,F}}   \\
&= \dfrac{(\mathfrak{r}_F^\times : \mathfrak{r}_{F^+}^\times)}{\sqrt{\lvert D_{F^+}\rvert} } L(\psi\eta,0) \left\{\Omega^{w_\eta\mathsf{t}+2\mathsf{r}_{\eta}}_{\mathrm{CM},\infty,F}\prod_{\sigma\in \Sigma_F} (2\pi i)^{w_\eta+2r_{\eta,\sigma}}\right\}^{-1}   \\
& \qquad \cdot \prod_{v\in \Sigma_{F,p}} (\psi\eta)_v(2\delta_v)^{-1}\mathrm{Eul}_v^\delta(\psi\eta,0)  \prod_{\sigma\in \Sigma_F}(2\pi i)^{r_{\eta,\sigma}}\Gamma(w_\eta+r_{\eta,\sigma}),
\end{align*}
where $\mathrm{Eul}^\delta_v(\psi\eta,0)$ is defined as 
\begin{align*}
 & \mathrm{Eul}_v^\delta(\psi\eta,0)=
\begin{cases}
 L_{v^c}(\psi \eta,0)^{-1}L_v((\psi\eta)^\vee,1)^{-1} & \text{if $\psi\eta$ is unramified at $v$}, \\
\epsilon((\psi \eta)_v,\mathbf{e}_{F,v}((2\delta_v)^{-1}-),dx_v)^{-1} & \text{if $\psi\eta$ is ramified at $v$}.
\end{cases}  
\end{align*}
See also Remark~\ref{rem:Hecke_epsilon} on comparison between $\mathrm{Eul}_v^\delta(\psi\eta,0)$ and the local term appearing in \cite[(5.7.28)]{katz} and \cite[(0.10)]{HT-katz}. One then readily obtains the desired interpolation formula taking the equalities
\begin{align*}
 \prod_{\sigma\in \Sigma_F}(2\pi)^{-w_\eta-r_{\eta,\sigma}}\Gamma(w_\eta+r_{\eta,\sigma})&=2^{-d}L_\infty(\psi\eta,0), \\
\bigl\{(\psi\eta)_v(2\delta_v)\epsilon\bigl((\psi\eta)_v,\mathbf{e}_{F,v}((2\delta_v)^{-1}-), dx_v\bigr)\bigr\}^{-1} &=\epsilon\bigl((\psi\eta)_v,\mathbf{e}_{F,v},dx_v\bigr)^{-1} \\
& \quad \qquad (\text{see \cite[(3.3.3)]{Deligne_constant} for example})
\end{align*}
into accounts.
\end{proof} 

\begin{rem}[Independency of the polarisation parameter $\delta$]
Under the assumption (unr${}_{F,p}$), we have succeeded in getting rid of the dependency of the $p$-adic Hecke $L$-function $L_{p,\Sigma_F}(\psi)$ on the auxiliary element $\delta$ satisfying the conditions $(1_\delta)$, $(2_\delta)$ and $(3_\delta)$ at the beginning of Section~\ref{ssc:period_relation}; note that Katz' and Hida--Tilouine's original $p$-adic Hecke $L$-function $L^{\mathrm{KHT}}_{p,\Sigma_F,\delta}(F)$ {\em does depend} on the choice of $\delta$. It seems hard to remove $\delta$-dependency of the $p$-adic Hecke $L$-function when $F^+$ is ramified at $(p)$. 
\end{rem}

\begin{rem}[Relation with Coates and Perrin-Riou's conjecture]  \label{rem:CPC}
We clarify that our $p$-adic Hecke $L$-function $L_{p,\Sigma_F}(\psi)$ is compatible with Coates and Perrin-Riou's conjecture \cite{cp89,coa89}.
The term $i^{\lvert -w_\eta \mathsf{t}-\mathsf{r}_\eta\rvert}L_\infty(\psi\eta,0)$ is %regarded as %for arXiv version
the modified $\infty$-Euler factor $\mathcal{L}^{(i)}_\infty(\mathrm{Res}_{F/\mathbb{Q}}\,M(\psi\eta))$ introduced in \cite[p.\,103]{coa89}, and $\Omega_{\mathrm{CM},\infty,F}^{w_\eta \mathsf{t}+2\mathsf{r}_\eta}$ amounts to the modified period 
\begin{align*}
 \Omega^{(i)}(\mathrm{Res}_{F/\mathbb{Q}}\, M(\psi\eta))=C^+(\mathrm{Res}_{F/\mathbb{Q}}M(\psi\eta))(2\pi i)^{\tau(\mathrm{Res}_{F/\mathbb{Q}}M(\psi\eta))}
\end{align*}
defined in \cite[p.\,107]{coa89}; indeed $\tau(\mathrm{Res}_{F/\mathbb{Q}}\, M(\psi\eta))$ is calculated as $\lvert -w_\eta \mathsf{t}-\mathsf{r}_\eta\rvert$, and algebraicity of the right-hand side of \eqref{eq:p-adic_Hecke} verified in \cite[(5.3.5)]{katz} implies that Deligne's period $C^+(\mathrm{Res}_{F/\mathbb{Q}}\, M(\psi\eta))$ should coincide up to non-zero algebraic multiples  with $(2\pi i)^{\lvert -\mathsf{r}_\eta\rvert} C_{\mathrm{CM}, \infty, F}^{w_\eta \mathsf{t}+2\mathsf{r}_\eta}$. Furthermore $p$ is ordinary for the motive  $\mathrm{Res}_{F/\mathbb{Q}}\, M(\psi\eta)$ in the sense of \cite[Definition 4.1]{cp89} or \cite[Section~3]{coa89}, and $\prod_{v\in \Sigma_{F,p}}\mathrm{Eul}_v(\psi\eta, s)$ is just the modified $p$-Euler factor $\mathcal{L}^{(i)}_p(\mathrm{Res}_{F/\mathbb{Q}}\, M(\psi\eta))$ introduced in \cite[Section~2]{coa89} (divided by the usual $p$-Euler factor $L_p(\mathrm{Res}_{F/\mathbb{Q}}\, M(\psi\eta))$); compare with \cite[Lemma~3]{coa89}. 
\end{rem}

\begin{rem}[On local epsilon factors] \label{rem:Hecke_epsilon} 
For completeness, we here verify that the product $\prod_{v\in \Sigma_{F,v}}\epsilon\bigl((\psi\eta)_v,\mathbf{e}_{F,v}\bigl((2\delta_v)^{-1}-\bigr),dx_v)^{-1}$ of Tate's local constants coincides with the local term $\mathrm{Local}(\psi\eta,v\text{'s})$ introduced by Katz \cite[(5.7.28)]{katz} and $W_p(\psi\eta)$ introduced by Hida and Tilouine \cite[(0.10)]{HT-katz}. Fix a $p$-adic place $v\in \Sigma_{F,p}$. 
The functional equation of the local constant \cite[(5.7.1)]{Deligne_constant} and the behaviour under unramified twists \cite[(5.5.3)]{Deligne_constant} imply the equality 
\begin{align*}
 \epsilon\bigl((\psi\eta)_v,\mathbf{e}_{F,v}\bigl((2\delta_v)^{-1}-\bigr),dx_v\bigr)^{-1}&=\epsilon((\psi \eta)_v^{-1} \lvert -\rvert_{F_v}, \mathbf{e}_{F,v}\bigl((-2\delta_v)^{-1}-\bigr),dx_v) \\
&=\epsilon\bigl((\psi\eta)_v^{-1},\mathbf{e}_{F,v}\bigl((-2\delta_v)^{-1}-\bigr), dx_v\bigr)  \mathcal{N}v^{-e_v(\psi\eta)},
\end{align*}
where $ \lvert -\rvert_{F_v}$ denotes the normalised valuation on $F_v$ 
and $e_v(\psi\eta)$ is the exponent of the conductor of $(\psi \eta)_v$. Note that $dx_v$ is a self-dual measure since $F_v$ is unramified over $\mathbb{Q}_p$. Applying \cite[(3.4.3.2)]{Deligne_constant} to the right-hand side, we obtain 
\begin{align*}
 \epsilon&\bigl((\psi\eta)_v, \mathbf{e}_{F,v}\bigl((2\delta_v)^{-1}-\bigr), dx_{v}\bigr)^{-1} =\mathcal{N}v^{-e_v(\psi\eta)}   \epsilon\bigl((\psi \eta)_v^{-1}, \mathbf{e}_{F,v}\bigl((-2\delta_v)^{-1}-\bigr), dx_{v}\bigr)   \\
&=\mathcal{N}v^{-e_v(\psi\eta)} \int_{\varpi _v^{-e_v(\psi\eta)}\mathfrak{r}_{F_v}}  (\psi \eta)_v(x) \mathbf{e}_{F,v}\bigl((-2\delta_v)^{-1}x\bigr) \, dx_v \\
&=\mathcal{N}v^{-e_v(\psi\eta)} \int_{\mathfrak{r}_{F_v}} (\psi\eta)_v(\varpi_v^{-e_v(\psi\eta)}x) \mathbf{e}_{F,v}\bigl(-\varpi_v^{-e_v(\psi\eta)}(2\delta_v)^{-1}x\bigr) \, \bigl( \mathcal{N}v^{e_v(\psi\eta)}dx_{v}\bigr) \\
&=\sum_{x\in (\mathfrak{r}_{F,v}/\varpi_v^{e_v(\psi\eta)})^\times}  (\psi\eta)_v(\varpi_v^{-e_v(\psi\eta)}x)\mathbf{e}_{F,v}(-\varpi_v^{-e_v(\psi\eta)}(2\delta_v)^{-1}x) \int_{1+\varpi_v^{e_v(\psi\eta)}\mathfrak{r}_{F_v}} dx_v \\ 
&=\bigl(\mathcal{N}v\;  (\psi\eta)_v(\varpi_v)\bigr)^{-e_v(\psi\eta)} \!\!\!  \sum_{x\in (\mathfrak{r}_{F_v}/\varpi_v^{e_v(\psi\eta)})^\times} \!\!\! (\psi \eta)_v(x) \exp(2\pi i\, \mathrm{Tr}_{F_v/\mathbb{Q}_p} (\varpi_v^{-e_v(\psi\eta)}(2\delta_v)^{-1}x)),
\end{align*}
 which completely concides with the generalised Gauss sum $\rho_v((\psi\eta)_v)$ appearing in \cite[(5.7.15)]{katz} (the local factor $\mathcal{N}v^{-e_v(\psi\eta)}G((2\delta_v),(\psi\eta)_v)$ introduced in \cite[(0.10)]{HT-katz} seems to contain a typo on the sign of the additive character).
\end{rem}

\section[Constructing the $p$-adic Artin $L$-function in the field of fractions of the Iwasawa algebra]{Constructing the $p$-adic Artin $L$-function   in the field of fractions of the Iwasawa algebra}\label{section:matching}
We construct a $p$-adic Artin $L$-function as a unique element of the {\em field of fractions} of the Iwasawa algebra $\widehat{\mathcal{O}}^{\mathrm{ur}}[[\Gamma_{F,\, \max}]]$. The integrality of the $p$-adic Artin $L$-function thus constructed will be discussed in the next section.

\subsection{Local $L$-factors and local $\epsilon$-factors}\label{subsection: localepsilonfactor} 

Before presenting the construction of $p$-adic Artin $L$-functions, we here summarise basic notion on local $L$-factors and local $\epsilon$-factors. For a CM field $F$, let $\eta=(\eta_v)_v \colon \mathbb{A}_F^\times/F^\times\rightarrow \overline{\mathbb{Q}}^\times$ be an algebraic Hecke character,  and consider an Artin representation $\rho\colon G_F \rightarrow {\rm Aut}_{\overline{\mathbb{Q}}}\, V_\rho$  of degree $r(\rho)$. %submit version: and $\rho\colon ...$ be an Artin representation of degree $r(\rho)$
As is well known, both $\eta$ and $\rho$ are defined on a number field $E$ of sufficiently large degree over $\mathbb{Q}$. Furthermore $\eta$ corresponds to a unique continuous character $\eta^\mathrm{gal}\colon G_F\rightarrow \mathcal{E}^\times$ via global class field theory, where $\mathcal{E}$ is an appropriate finite extension of $\mathbb{Q}_p$ containing $\boldsymbol{\iota}\circ\iota_\infty(E)$.  For each finite place $v$ of $F$, let $W_{F_v}\subset G_{F_v}$ (resp.\ ${}^\prime W_{F_v}$) denote the Weil group (resp.\ the Weil--Deligne group) of $F_v$. Following Fontaine's recipe proposed in \cite[Section 2.3.7]{Fontaine94}, we associate with $V_\rho\otimes_{\mathcal{E}} \eta^\mathrm{gal}$ a Weil--Deligne representation $W\widehat{D}_{\mathrm{pst},v}(V_\rho\otimes_{\mathcal{E}} \eta^\mathrm{gal})$ of ${}^\prime W_{F_v}$. It is defined over $\widehat{\mathcal{E}}^\mathrm{ur}$, the completion of the maximal unramified extension of $\mathcal{E}$ if $v$ lies above $p$, and over $\mathcal{E}$ otherwise. 
 Define $W\widehat{D}_{\mathrm{pst},v}(V_\rho\otimes_{\mathcal{E}} \eta^\mathrm{gal})_\mathbb{C}$ as the scalar extension of $W\widehat{D}_{\mathrm{pst},v}(V_\rho\otimes_{\mathcal{E}} \eta^\mathrm{gal})$ to $\mathbb{C}$ via $\boldsymbol{\iota}^{-1}$.
One readily observes by construction that $W\widehat{D}_{\mathrm{pst},v}(V_\rho\otimes_{\mathcal{E}} \eta^\mathrm{gal})_\mathbb{C}\cong V_\rho \otimes_{\mathbb{C}} W\widehat{D}_{\mathrm{pst},v}(\eta^\mathrm{gal})_\mathbb{C}$ is a complex vector space of dimension $r(\rho)$ on which $\gamma\in W_{F_v}$ acts as $\iota_\infty \bigl(\rho(\gamma) \eta_v(\mathrm{rec}_{F,v}^{-1}(\gamma^{\mathrm{ab}}))\bigr)$, where $\eta_v$ is the $v$-component of the Hecke character $\eta=(\eta_v)_v$, $\gamma^\mathrm{ab}$ is the image of $\gamma$ in the abelianisation $W_{F_v}^\mathrm{ab}$ of $W_{F_v}$, and  $\mathrm{rec}_{F,v}\colon F_v^\times \xrightarrow{\, \sim \,} W_{F_v}^\mathrm{ab}$ denotes the local Artin reciprocity map at $v$.

Under these settings, the local $L$-factor and the local $\epsilon$-factor of the %pure %omit %for arXiv version
motive $M(\rho)\otimes_E M(\eta)$ at a finite place $v$ of $F$ are defined as follows: 
\begin{align}
& L_v(\rho\otimes \eta,s) 
:=
 \det\bigl(1-\mathrm{Frob}_v \,  \mathcal{N}v^{-s} \mid \bigl(W\widehat{D}_{\mathrm{pst},v}(V_\rho\otimes_{\mathcal{E}} \eta^\mathrm{gal})_\mathbb{C}\bigr)^{I_v,N=0}\bigr)^{-1},
 \label{equation:L_factor} \\
& \epsilon((\rho\otimes \eta)_v,\mathbf{e}_{F,v},dx_v):=\epsilon(W\widehat{D}_{\mathrm{pst},v}(V_\rho\otimes_{\mathcal{E}} \eta^\mathrm{gal})_\mathbb{C},\mathbf{e}_{F,v},dx_{v}). \label{equation:local_constant}
\end{align}
The right-hand side of \eqref{equation:local_constant} denotes Deligne's local constant defined as in \cite[Th\'eor\`em 4.1]{Deligne_constant} with respect to the standard additive character $\mathbf{e}_{F,v}$ and the normalised Haar measure $dx_v$ of $F_v$. In particular, if both $\eta$ and $\rho$ are unramified at $v$, we have 
\begin{align*}
 L_v(\rho\otimes \eta, s)=\det(1-\eta_v(\varpi_v)\rho(\mathrm{Frob}_v)\mathcal{N}v^{-s})^{-1}
\end{align*}
where $\varpi_v$ denotes a uniformiser of $F_v$ and $\epsilon((\rho\otimes \eta)_v,\mathbf{e}_{F,v},dx_v)$ is equal to $1$. 
We define the Hasse--Weil $L$-function of the motive $M(\rho)\otimes_E M(\eta)$ by 
\begin{align*}
 L(\rho\otimes \eta,s):=\prod_{v \colon \text{finite places of $F$}}L_v(\rho\otimes \eta,s).
\end{align*}

Concerning the archimedean local $L$-factor, let $(w_\eta, \mathsf{r}_\eta)\in \mathbb{Z}\times \mathbb{Z}^{\Sigma_F}$ be the infinity type of $\eta$ as in Section~\ref{ssc:p-adic_Hecke}. Then the Hodge realisation 
of $M(\rho) \otimes_E M(\eta)$ with respect to $\sigma\in \Sigma_F$ is an $r(\rho)$-dimensional $\mathbb{C}$-vector space of Hodge type $\{(-w_\eta-r_{\eta,\sigma},r_{\eta,\sigma})\}_{\sigma\in\Sigma_F}$, and thus the archimedean local $L$-factor $L_\infty(\rho\otimes \eta,s)$ of $M(\rho)\otimes_E M(\eta)$ is described as 
\begin{align} \label{eq:arch_L_rho}
 L_\infty(\rho\otimes \eta,s)=\prod_{\sigma\in \Sigma_F}L_\sigma(\rho\otimes \eta,s)=\prod_{\sigma\in \Sigma_F} \Gamma_\mathbb{C}(s+w_\eta+r_{\eta,\sigma})^{r(\rho)}.
\end{align}
The completed $L$-function is then defined as $\Lambda(\rho\otimes \eta,s)=L_\infty(\rho\otimes\eta,s)L(\rho\otimes \eta,s)$.

\subsection{Artin representations cutting out CM fields}  \label{ssc:artin_CM}

Here after we consider an Artin representation $\rho \colon G_F\rightarrow \mathrm{Aut}_E(V_\rho)$ such that {\em the field $F_\rho$ corresponding to the kernel of $\rho$ is also a CM field}. This condition imposes a rather strict Galois-theoretic constraint, as the following lemma implies. 

\begin{lem} \label{lem:CM_ext}
 Let $K'/K$ be a finite Galois extension of a CM field $K$. Suppose that $K'$ is also a CM field, and let $c$ denote the complex conjugation on $K'$. Then $K'/K^+$ is a Galois extension with Galois group isomorphic to $\mathrm{Gal}(K'/K)\times \langle c\rangle$. In particular, $K^{\prime, +}/K^+$ is also a Galois extension and $\mathrm{Gal}(K'/K)$ is isomorphic to $\mathrm{Gal}(K^{\prime,+}/K^+)$. 
\end{lem}

Roughly speaking, every Galois extension of CM fields is  derived from the corresponding Galois extension of maximal totally real subfields. 

\begin{proof}
The complex conjugation $c$ of the CM field $K'$ is characterised as an automorphism of $K'$ of order $2$ satisfying $\bar{\cdot}\circ \sigma=\sigma\circ c$ for any complex embedding $\sigma\colon K'\hookrightarrow \mathbb{C}$, where $\bar{\cdot}$ denotes the complex conjugation on $\mathbb{C}$; see \cite[Lemma 18.2 (i)]{shimura61} for example. Using this fact, one readily observes that $c$ commutes with any element of $\mathrm{Aut}(K'/\mathbb{Q})$, and that $c\vert_{K}$ coincides with the complex conjugation of the CM field $K$. Then one finds $2[K':K]$ distinct elements of $\mathrm{Aut}(K'/K^+)$, namely $g\in \mathrm{Gal}(K'/K)$ and $c\circ g$ for $g\in \mathrm{Gal}(K'/K)$; note that $g\in \mathrm{Gal}(K'/K)$ acts trivially on $K$ whereas $c$ acts on $K$ nontrivially as the complex conjugation. This implies that $K'/K^+$ is a normal extension, as desired. The other assertions immediately follow from this fact.
\end{proof}

\subsection{Gluing $p$-adic Hecke $L$-functions}

Now let us state the main result of this section. Refer to Remark~\ref{rem:multi_index} for multi-index notation used in the statement.

\begin{thm}\label{theorem:CMp-adicL_Artin}
Let $p$ be an odd prime number, ${F}$ a CM field of degree $2d$ and $F^+$ the maximal totally real subfield of $F$. Assume that 
${F}$, ${F}^+$ and $p$ satisfy both  {\upshape (unr$_{F,p}$)} and {\upshape
 (ord$_{F,p}$)}. 
In addition, let $\rho\colon G_F\rightarrow {\rm Aut}_E(V_\rho)$ be an Artin representation of degree $r(\rho)$ unramified at any prime ideals lying above $(p)$, and assume that the field $F_\rho$ corresponding to the kernel of $\rho$ is a CM field, as in Section $\ref{ssc:artin_CM}$. 
Then, for each $p$-ordinary CM type $\Sigma_F$ of $F$, there exists a unique element $L_{p,\Sigma_F} (M(\rho ))$ of $\mathrm{Frac} (\widehat{\mathcal{O}}^{\mathrm{ur}}[[
\Gamma_{F,\, \max}]])$ satisfying
\begin{equation} \label{eq:interpolation_Lp_Artin}
\frac{\eta^\mathrm{gal} (L_{p,\Sigma_F} (M(\rho )))}{\bigl(\Omega^{w_\eta \mathsf{t}+2\mathsf{r}_\eta}_{\mathrm{CM}, p,F} \bigr)^{r(\rho )}} = 
i^{r(\rho)\lvert -w_\eta \mathsf{t}-\mathsf{r}_\eta\rvert}
\prod_{v\in \Sigma_{F,p}}\mathrm{Eul}_v (\rho\otimes \eta,0) 
 \dfrac{\Lambda(\rho\otimes \eta,0)}{\bigl(\Omega^{w_\eta\mathsf{t}+2\mathsf{r}_\eta}_{\mathrm{CM}, \infty, F} \bigr)^{r(\rho )}}
\end{equation}
with
\begin{align*}
& \mathrm{Eul}_v(\rho\otimes \eta,0)=
\begin{cases}
 L_{v^c}(\rho\otimes \eta,0)^{-1}L_v((\rho\otimes \eta)^\vee, 1)^{-1} & \text{if $\eta$ is unramified at $v$}, \\
\epsilon((\rho \otimes \eta)_v,\mathbf{e}_{F,v},dx_v)^{-1} & \text{ if $\eta$ is ramified at $v$}
\end{cases} 
\end{align*}
for any algebraic Hecke character $\eta \colon \mathbb{A}_F^\times/F^\times\rightarrow \overline{\mathbb{Q}}^\times$ such that 
\begin{enumerate}[label={\protect\upshape (\alph*)}]
 \item the associated Galois character $\eta^{\mathrm{gal}}$ factors through $\Gamma_{F,\, \max};$
 \item the infinity type $(w_\eta,\mathsf{r}_\eta)\in \mathbb{Z}\times \mathbb{Z}^{\Sigma_F}$ of $\eta$ satisfies both $-w_\eta-r_{\eta,\sigma}\leq -1$ and $r_{\eta,\sigma}\geq 0$ for every $\sigma\in \Sigma_F$. 
\end{enumerate}
\end{thm}

Here we use the symbol $L_{p,\Sigma_F}(M(\rho))$ to emphasise that it is a right object which should be called the $p$-adic $L$-function of the Artin motive $M(\rho)$. Note that the interpolation formula \eqref{eq:interpolation_Lp_Artin} complies the general formulation of Coates and Perrin-Riou \cite{cp89,coa89}; see also Remark \ref{rem:CPC}.

\begin{proof}
Set $G=\mathrm{Gal}(F_\rho/F)=G_F/\ker(\rho)$, which is a finite group by definition. Then, by Brauer's induction theorem \cite[Theorem (15.9)]{CR81}, the Artin representation $\rho$ is decomposed (as a virtual representation of $G$) into
\begin{equation}\label{equation:brauer_induction}
\rho= \sum^s_{j=1} a_j \, \mathrm{Ind}^G_{G_j} \psi_j^{\mathrm{gal}}, 
\end{equation}
where $G_j=\mathrm{Gal}(F_\rho/F_j)$ is a subgroup of $G$, $\psi_j^{\mathrm{gal}}$ is an abelian character of $G_j$ and  $a_j$ is an integer for each $j$. Each $F_j$ is then a CM field because it is an intermediate field of the extension $F_\rho/F$ of CM fields; see \cite[Lemma 18.2 (iv)]{shimura61} for details. Furthermore since the dimensions of the both sides of \eqref{equation:brauer_induction} obviously coincide, we obtain a basic equality
\begin{align}\label{eq:brauer_induction_dim}
 r(\rho)=\sum_{j=1}^s a_j (G:G_j) \dim \psi_j^\mathrm{gal} =\sum_{j=1}^s a_j [F_j:F].
\end{align}
For each $1\leq j\leq s$, the  $p$-ordinary CM type $\Sigma_F$ induces a unique $p$-ordinary CM type $\Sigma_{F_j}:=\{\tau\in I_{F_j} \mid \tau\rvert_F \in \Sigma_F\}$ of $F_j$. Due to Theorem~\ref{thm:KHT}, the $p$-adic $L$-function $L_{p,\Sigma_{F_j}}(\psi_j)$ with respect to the $p$-ordinary CM type $\Sigma_{F_j}$ uniquely exists as an element of $\widehat{\mathcal{O}}^\mathrm{ur}[[\Gamma_{F_j,\max}]]$. It is characterised by the interpolation property 
\begin{align}\label{equation:katz1}  %changed to align
\frac{\xi^{\mathrm{gal}}(L_{p,\Sigma_{F_j}}(\psi_j))}
{\Omega_{\mathrm{CM},p,F_j }^{w_\xi \mathsf{t}_j+2\mathsf{r}_\xi}} = 
\dfrac{(\mathfrak{r}_{F_j}^\times : \mathfrak{r}_{F_j^+}^\times )}{2^{d[F_j:F]}\sqrt{\lvert D_{F_j^+}\rvert} } i^{\lvert -w_\xi \mathsf{t}_j-\mathsf{r}_\xi\rvert}
 \prod_{\tilde{v}\in \Sigma_{F_j,p}}\mathrm{Eul}_{\tilde{v}}(\psi_j\xi,0)\frac{\Lambda(\psi_j\xi,0)}{\Omega_{\mathrm{CM},\infty ,F_j }^{w_\xi\mathsf{t}_j+2\mathsf{r}_\xi}}
\end{align}
for any algebraic Hecke character $\xi \colon \mathbb{A}_{F_j}^\times/F_j^\times\rightarrow \overline{\mathbb{Q}}^\times$ satisfying both (i) and (ii) in Theorem~\ref{thm:KHT}; note that $[F_j:\mathbb{Q}]=2d[F_j:F]$. Here $\mathsf{t}_j$ denotes  $(1,1,\ldots,1)\in \mathbb{Z}^{\Sigma_{F_j}}$. 
Now let us consider the composite map
\begin{align*}
\mathrm{pr}_j \colon \widehat{\mathcal{O}}^{\mathrm{ur}}[[\Gamma_{F_j ,\,\max}]] \twoheadrightarrow \widehat{\mathcal{O}}^{\mathrm{ur}}[[\mathrm{Gal}(F_j F_{\max}/F_j)]] \hookrightarrow  \widehat{\mathcal{O}}^{\mathrm{ur}}[[\Gamma_{F,\, \max}]]
\end{align*}
where the former map is the ring homomorphism induced by the quotient map of the Galois group 
$\Gamma_{F_j ,\,\max} =\mathrm{Gal}(F_{j,\,\max}/F_j) \twoheadrightarrow  \mathrm{Gal}(F_j F_{\max}/F_j)$, and 
the latter one is the ring homomorphism induced by the inclusion
$  \mathrm{Gal}(F_j F_{\max}/F_j) \subset \mathrm{Gal}(F_{\max}/F)=\Gamma_{F, \max}$.  Now we define $L_{p,\Sigma_F}(M(\rho))$ as
\begin{equation}\label{equation:katz3}
L_{p,\Sigma_F}(M(\rho ))=
\displaystyle \prod_{j=1}^s  \mathrm{pr}_j\left( \dfrac{2^{d[F_j:F]}\sqrt{\lvert D_{F_j^+}\rvert}}{(\mathfrak{r}_{F_j}^\times : \mathfrak{r}_{F_j^+}^\times)}L_{p,\Sigma_{F_j}}(\psi_j)\right)^{a_j}, 
\end{equation} 
which is a priori an element of the field of fractions of $\widehat{\mathcal{O}}^{\mathrm{ur}}[[\Gamma_{F,\, \max}]]$; note that both the absolute discriminant $D_{F_j^+}$ and the unit index $(\mathfrak{r}_{F_j}^\times :\mathfrak{r}_{F_j^+}^\times)$ are $p$-adic units due to the assumption (unr${}_{F,p}$) and unramifiedness of $\rho$ at $p$. In the rest of the proof, we assemble the interpolation formulae of the $p$-adic Hecke $L$-functions appearing in the right hand side of \eqref{equation:katz3}, and deduce the desired interpolation property of $L_{p,\Sigma_F}(M(\rho))$. 

\begin{component}[$\blacktriangleright$ Infinity types of $\psi_j\eta_j$] 
 Let $\eta\colon \mathbb{A}_F^\times/F^\times \rightarrow \mathbb{C}^\times$ be an algebraic Hecke character satisfying the conditions (a) and (b) in Theorem~\ref{theorem:CMp-adicL_Artin}. The restriction $\eta^{\mathrm{gal}}\vert_{G_{F_j}}$ of $\eta^{\mathrm{gal}}$ then corresponds to $\eta_j:=\eta\circ \mathrm{Norm}_{F_j/F}$ via global class field theory. Let $\psi_j$ denote the algebraic Hecke character corresponding to $\psi_j^\mathrm{gal}$ via global class field theory. Since the Norm map $\mathrm{Norm}_{F_j/F}$ induces 
\begin{align*}
 \mathbb{A}_{F_j}^{\infty,\times}=(\mathbb{C}^\times)^{\Sigma_{F_j}}\rightarrow  \mathbb{A}_{F}^{\infty,\times}=(\mathbb{C}^\times)^{\Sigma_{F}}\,; \, (x_\tau)_{\tau\in \Sigma_{F_j}} \mapsto \left(\prod_{\tau\vert_F=\sigma} x_\tau\right)_{\sigma\in \Sigma_F},
\end{align*}
one readily observes that the infinity type $(w_{\eta_j}, \mathsf{r}_{\eta_j})\in \mathbb{Z}\times \mathbb{Z}^{\Sigma_{F_j}}$ of $\psi_j\eta_j$ satisfies equalities $w_{\eta_j}=w_\eta$ and  $r_{\eta_j,\tau}=r_{\eta,\tau\vert_F}$ for each $\tau\in \Sigma_{F_j}$. From these formulae, we see that $\eta_j$ satisfies the condition (ii) of Theorem~\ref{thm:KHT} for $F_j$.  
 The condition (i) of Theorem~\ref{thm:KHT} is obviously fulfilled for $\eta_j$ by assumption, and hence we have 
\begin{align}
 \dfrac{\eta^{\mathrm{gal}} (L_{p,\Sigma_F}(M(\rho )))}
{\prod_{j=1}^s \bigl( \Omega_{\mathrm{CM},p,F_j }^{w_{\eta}\mathsf{t}_j+2\mathsf{r}_{\eta_j}}\bigr)^{a_j}} &= \prod_{j=1}^s \left( \dfrac{2^{d[F_j:F]}\sqrt{\lvert D_{F_j^+}\rvert}}{(\mathfrak{r}_{F_j}^\times: \mathfrak{r}_{F_j^+}^\times)}\dfrac{\eta^{\mathrm{gal}}\vert_{\mathrm{Gal}(\overline{F}/F_j)}(L_{p,\Sigma_{F_j}}(\psi_j))}{\Omega_{\mathrm{CM},p,F_j}^{w_{\eta}\mathsf{t}_j+2\mathsf{r}_{\eta_j}}}\right)^{a_j}
 \label{equation:specialization_p-adicL}\\ 
& =  \prod_{j=1}^s 
\left(  i^{\lvert -w_\eta \mathsf{t}_j-\mathsf{r}_{\eta_j}\rvert}
 \prod_{\tilde{v}\in \Sigma_{F_j,p}}  \mathrm{Eul}_{\tilde{v}}(\psi_j\eta_j,0)\frac{L_\infty(\psi_j\eta_j,0)L(\psi_j\eta_j,0)}{\Omega_{\mathrm{CM},\infty ,F_j }^{w_\eta\mathsf{t}_j+2\mathsf{r}_{\eta_j}}} \right)^{a_j} \nonumber
\end{align}
 by \eqref{equation:katz1} and \eqref{equation:katz3}. Here we have already substituted $w_{\eta_j}=w_\eta$. 
We shall modify each term of \eqref{equation:specialization_p-adicL} and show that \eqref{equation:specialization_p-adicL} is equivalent to the desired interpolation formula \eqref{eq:interpolation_Lp_Artin}.
\end{component}

\begin{component}[$\blacktriangleright$ $L$-values]
For each $j=1,2,\ldots,s$,  Mackey's decomposition theorem \cite[Theorem (10.13)]{CR81} provides a decomposition of $\mathrm{Ind}^G_{G_j}\, \psi_j^\mathrm{gal}$ appearing in  \eqref{equation:brauer_induction} as $G_{F_v}$-representations
\begin{align} \label{equation:Mackey}
 \mathrm{Ind}^G_{G_j}\, \psi_j^\mathrm{gal} \vert_{G_{F_v}} =\bigoplus_{[g]\in G_{F_j}\backslash G_F/G_{F_v}} \mathrm{Ind}^{G_{F_v}}_{G_{F_j}^g\cap G_{F_v}} \, (\psi_j^\mathrm{gal})^g 
\end{align}
where $G_{F_j}^g$ denotes the conjugate $g^{-1}G_{F_j}g$ of $G_{F_j}$ and the character $(\psi_j^\mathrm{gal})^g$ is defined as $(\psi^\mathrm{gal}_j)^g(x)=\psi_j^\mathrm{gal}(g xg^{-1})$ for $x\in G_{F_j}^g\cap G_{F_v}$. 
Combining \eqref{equation:Mackey} with the Tensor Product Theorem \cite[Corollary (10.20)]{CR81}, we obtain an eqality of virtual representations:
\begin{align} 
 & W\widehat{D}_{\mathrm{pst},v}(\rho\otimes_{\mathcal{E}} \eta^\mathrm{gal})_\mathbb{C}  \label{eq:Brauer_L}\\
&\;= \sum_{j=1}^s a_j \!\! \sum_{[g]\in G_{F_j}\backslash G_F/G_{F_v}}  \!\! \bigl(\mathrm{Ind}^{W_{F_v}}_{G_{F_j}^g\cap W_{F_v}} \, (\psi_j^\mathrm{gal})^g \bigr)\otimes_{\mathbb{C}} W\widehat{D}_{\mathrm{pst},v}(\eta^\mathrm{gal})_\mathbb{C}   \nonumber \\
&\;= \sum_{j=1}^s a_j \! \sum_{[g]\in G_{F_j}\backslash G_F/G_{F_v}} \! \mathrm{Ind}^{W_{F_v}}_{G_{F_j}^g \cap W_{F_v}}  \left( (\psi_j^\mathrm{gal})^g \otimes_{\mathbb{C}} W\widehat{D}_{\mathrm{pst},v}(\eta^\mathrm{gal})_\mathbb{C}\vert_{G_{F_j}^g\cap W_{F_v}} \right).  \nonumber
\end{align}
Now let $\tilde{v}_0$ denote a unique place of $F_j$ fixed by $G_{F_v}$ (we supress $j$ by abuse of notation). Set $W_{F_{j,\tilde{v}_0^g}}:=G_{F_j}\cap W_{F_v}^{g^{-1}}$, which is indeed regarded as the Weil group of $F_{j,\tilde{v}_0^g}$. If we regard $(\psi_j^\mathrm{gal})^g\otimes_{\mathbb{C}} W\widehat{D}_{\mathrm{pst},v}(\eta^\mathrm{gal})_{\mathbb{C}}\vert_{G_{F_j}^g\cap W_{F_v}}$ as a representation of $W_{F_{j,\tilde{v}_0^g}}$ via the isomorphism $W_{F_{j,\tilde{v}_0^g}} \xrightarrow{\, \sim \,} G_{F_j}^g\cap W_{F_v}\, ; \gamma \mapsto g^{-1}\gamma g$, we see that it is isomorphic to $\psi_j^\mathrm{gal}\otimes_{\mathbb{C}} W\widehat{D}_{\mathrm{pst},v}(\eta^\mathrm{gal})_\mathbb{C}\vert_{W_{F_{j,\tilde{v}_0^g}}}$. We thus obtain 
\begin{align} 
 \prod_{j=1}^s &\prod_{\tilde{v}\mid v} L((\psi_j\eta_j)_{\tilde{v}},s)^{a_j} =\prod_{j=1}^s \prod_{\tilde{v}\mid v} L(\psi_j^\mathrm{gal}\otimes_{\mathbb{C}} W\widehat{D}_{\mathrm{pst},v}(\eta^\mathrm{gal})_\mathbb{C}\vert_{W_{F_{j,\tilde{v}}}},s)^{a_j} \label{eq:Brauer_L_factor} \\
&=\prod_{j=1}^s\prod_{[g]} L( (\psi_j^\mathrm{gal})^g \otimes_{\mathbb{C}} W\widehat{D}_{\mathrm{pst},v}(\eta^\mathrm{gal})_\mathbb{C}\vert_{G_{F_j}^g\cap W_{F_v}},s)^{a_j}  \nonumber\\
&=\prod_{j=1}^s  \prod_{[g]}  L\left(  \mathrm{Ind}^{W_{F_v}}_{G_{F_j}^g \cap W_{F_v}}  \left((\psi_j^\mathrm{gal})^g \otimes_{\mathbb{C}} W\widehat{D}_{\mathrm{pst},v}(\eta^\mathrm{gal})_\mathbb{C}\vert_{G_{F_j}^g\cap W_{F_v}}\right),s\right)^{a_j}  \nonumber\\
&=L\left(\sum_{j=1}^s a_j  \sum_{[g]}   \mathrm{Ind}^{W_{F_v}}_{G_{F_j}^g \cap W_{F_v}} \left( (\psi_j^\mathrm{gal})^g \otimes_{\mathbb{C}} W\widehat{D}_{\mathrm{pst},v}(\eta^\mathrm{gal})_\mathbb{C}\vert_{G_{F_j}^g\cap W_{F_v}} \right) ,s\right)  \nonumber \\
&=L_v(\rho\otimes \eta,0)  \nonumber
\end{align}
where $\prod_{\tilde{v}\mid v}$ means the product over all the finite places $\tilde{v}$ of $F_j$ lying above $v$ and $\prod_{[g]}$ (resp.\ $\sum_{[g]}$) means the product (resp.\  the summation) over all double cosets $[g]$ of $G_{F_j}\backslash G_F/G_{F_v}$. The third equality follows due to inductivity of local $L$-fuctors (see \cite[Proposition 3.8 (ii)]{Deligne_constant}) and the last equality follows from \eqref{eq:Brauer_L}. By taking the product over all the finite places of $F$ and substituting $s=0$, we obtain $\prod_{j=1}^s L(\psi_j\eta_j,0)^{a_j}=L(\rho\otimes \eta,0)$ as desired. 
\end{component}

\begin{component}[$\blacktriangleright$ Modified $\infty$-Euler factors]
 By \eqref{eq:arch_L_eta}, the product of the modified $\infty$-Euler factors in the left-hand side of \eqref{equation:specialization_p-adicL} is rewritten as
\begin{align*}
 \prod_{j=1}^s i^{a_j\lvert -w_\eta\mathsf{t}_j-\mathsf{r}_{\eta_j}\rvert} L_\infty(\psi_j\eta_j,0)^{a_j} &=\prod_{j=1}^s \prod_{\tau\in \Sigma_{F_j}} i^{a_j(-w_\eta-r_{\eta,\tau})}\Gamma_\mathbb{C}(w_\eta+r_{\eta_j,\tau})^{a_j} \\
&=\prod_{j=1}^s \prod_{\sigma\in \Sigma_F}\prod_{\substack{\tau\in \Sigma_{F_j} \\ \tau\vert_F=\sigma}} \{i^{-w_\eta-r_{\eta,\tau}}\Gamma_\mathbb{C}(w_\eta+r_{\eta_j,\tau})\}^{a_j}. %need revise: place of \{
\end{align*}
Since the number of $\tau\in\Sigma_{F_j}$ satisfying $\tau\vert_F=\sigma$ equals $[F_j:F]$ for any $\sigma \in \Sigma_F$, we have 
\begin{align*}
 \prod_{j=1}^s i^{a_j\lvert -w_\eta\mathsf{t}_j-\mathsf{r}_{\eta_j}\rvert} L_\infty(\psi_j\eta_j,0)^{a_j} &=\prod_{j=1}^s \prod_{\sigma\in \Sigma_F}\prod_{\substack{\tau\in \Sigma_{F_j} \\ \tau\vert_F=\sigma}} \{i^{-w_\eta-r_{\eta,\tau\vert_F}}\Gamma_\mathbb{C}(w_\eta+r_{\eta,\tau\vert_F})\}^{a_j} \\
&=\prod_{\sigma\in \Sigma_F}\{i^{-w_\eta-r_{\eta,\sigma}}\Gamma_\mathbb{C}(w_\eta+r_{\eta,\sigma})\}^{\sum_{j=1}^s a_j[F_j:F]},
\end{align*}
which coincides with $i^{r(\rho) \lvert -w_\eta\mathsf{t}-\mathsf{r}_\eta\rvert}L_\infty(\rho\otimes \eta,0)$ by \eqref{eq:arch_L_rho} and \eqref{eq:brauer_induction_dim}.
\end{component}

\begin{component}[$\blacktriangleright$ Periods]
 Similarly to the computation of the modified $\infty$-Euler factors, we have
\begin{align*}
 \prod_{j=1}^s \bigl( \Omega_{\mathrm{CM},?,F_j }^{w_{\eta}\mathsf{t}_j+2\mathsf{r}_{\eta_j}}\bigr)^{a_j} 
&= \prod_{j=1}^s \prod_{\tau \in \Sigma_{F_j}} \bigl(\Omega_{\mathrm{CM},?,F_j ,\tau }^{w_{\eta} +2{r}_{\eta_j,\tau}}  \bigr)^{a_j}= 
\prod_{j=1}^s \prod_{\sigma \in \Sigma_F} \prod_{\substack{ \tau \in \Sigma_{F_j}\\   \tau \vert_F =\sigma }} \bigl(\Omega_{\mathrm{CM},?,F_j ,\tau }^{w_{\eta} +2{r}_{\eta_j,\tau}}  \bigr)^{a_j} \\
&= \prod_{j=1}^s \prod_{\sigma \in \Sigma_F} \prod_{\substack{ \tau \in \Sigma_{F_j}\\  \tau \vert_F =\sigma }} \bigl(\Omega_{\mathrm{CM},?,F ,\tau\vert_F }^{w_{\eta} +2{r}_{\eta,\tau\vert_F}}  \bigr)^{a_j}=\prod_{\sigma\in \Sigma_F} \bigl(\Omega_{\mathrm{CM},?,F,\sigma}^{w_\eta+2r_{\eta,\sigma}}\bigr)^{\sum_{j=1}^s a_j[F_j:F]}
\end{align*}
for each $?\in\{\infty,p\}$, which is equal to $\bigl(\Omega_{\mathrm{CM},?,F}^{w_\eta\mathsf{t}+2\mathsf{r}_\eta}\bigr)^{r(\rho)}$ due to \eqref{eq:brauer_induction_dim}. The third equality follows from Lemma \ref{lemma:deShalit}. 
\end{component}

\begin{component}[$\blacktriangleright$ Modified $p$-Euler factors at unramified places]

Let $v\in \Sigma_{F,p}$ be a place of $F$ lying above $(p)$ and suppose that $\eta$ is unramified at $v$. Then, since $\mathrm{Eul}_{\tilde{v}}(\psi_j\eta_j,0)$ is defined as the product %added %for arXiv version
$L_{\tilde{v}^c}(\psi_j\eta_j,0)^{-1}L_{\tilde{v}}((\psi_j\eta_j)^{-1},1)^{-1}$ for each $\tilde{v}\in \Sigma_{F_j,p}$ lying above $v$, we have 
\begin{align*}
 \prod_{j=1}^s \prod_{\substack{\tilde{v}\in \Sigma_{F_j,p} \\ \tilde{v}\mid v}} \mathrm{Eul}_{\tilde{v}}(\psi_j\eta_j,0)^{a_j}&= \prod_{j=1}^s \prod_{\substack{\tilde{v}\in \Sigma_{F_j,p} \\ \tilde{v}\mid v}} L_{\tilde{v}^c}(\psi_j\eta_j,0)^{-a_j}L_{\tilde{v}}((\psi_j\eta_j)^{-1},1)^{-a_j}\\
&=L ((\rho\otimes \eta)_{v^c},0)^{-1} L ((\rho \otimes \eta)_v^\vee ,1)^{-1} =\mathrm{Eul}_v(\rho\otimes \eta, 0)
\end{align*}
due to \eqref{eq:Brauer_L_factor}; note that \eqref{equation:brauer_induction} implies  $(\rho\otimes \eta)^{\vee}=\sum_{j=1}^s a_j \, \mathrm{Ind}^{G}_{G_j} (\psi_j^\mathrm{gal}\eta_j^\mathrm{gal})^{-1}$. 
\end{component}

\begin{component}[$\blacktriangleright$ Modified $p$-Euler factors at ramified places] 
Let $v\in \Sigma_{F,p}$ be a place of $F$ lying above $(p)$ and suppose that  $\eta$ is ramified at $v$. In this situation, $\mathrm{Eul}_v(\rho\otimes \eta,0)$  and $\mathrm{Eul}_{\tilde{v}}(\rho\otimes \eta,0)$ for $\tilde{v}\in \Sigma_{F_j,p}$ lying above $v$ are defined as $\epsilon((\rho\otimes \eta)_v,\mathbf{e}_{F,v},dx_v)$ and  $\epsilon((\psi_j\eta_j)_{\tilde{v}},\mathbf{e}_{F_j,\tilde{v}},dx_{\tilde{v}})$ respectively. In the following, we shall verify the equality among the local constants
\begin{align} \label{eq:Brauer_induction_epsilon}
\prod_{j=1}^s\prod_{\substack{\tilde{v}\in \Sigma_{F_j,p} \\ \tilde{v}\mid v}}  \epsilon((\psi_j\eta_j)_{\tilde{v}},\mathbf{e}_{F_j,\tilde{v}},dx_{\tilde{v}})^{a_j}=\epsilon((\rho\otimes \eta)_v,\mathbf{e}_{F,v},dx_v),
\end{align}
which implies the desired equality
\begin{align} \label{eq:Brauer_induction_A}
 \prod_{j=1}^s\prod_{\substack{\tilde{v}\in \Sigma_{F_j,p} \\ \tilde{v}\mid v}} \mathrm{Eul}_{\tilde{v}}(\psi_j\eta_j,0) =\mathrm{Eul}_v(\rho\otimes \eta,0).
\end{align}
Now let us prove \eqref{eq:Brauer_induction_epsilon}. Note that $\psi_j^\mathrm{gal}$  is unramified at each $\tilde{v}\in \Sigma_{F_j,p}$ since $F/F_j$ is unramified at $\tilde{v}$ by construction.  Hence, if we let $\mathfrak{P}_{\tilde{v}}^{e(\eta_{j,\tilde{v}})}$ denotes the conductor of the $\tilde{v}$ component of $\eta_j=\eta\circ \mathrm{Norm}_{F_j/F}$, the left-hand side of \eqref{eq:Brauer_induction_epsilon} is calculated as 
\begin{align} 
\prod_{j=1}^s &\prod_{\substack{\tilde{v}\in \Sigma_{F_j,p} \\ \tilde{v}\mid v}} \epsilon ((\psi_j\eta_j)_{\tilde{v}}, \mathbf{e}_{F_j,\tilde{v}},dx_{\tilde{v}})^{a_j}  \label{equation:epsilonfactor_into_two_terms}  \\
&= 
\prod_{j=1}^s \prod_{\substack{\tilde{v}\in \Sigma_{F_j,p} \\ \tilde{v}\mid v}}\epsilon ((\eta_j)_{\tilde{v}}, \mathbf{e}_{F_j,\tilde{v}},dx_{\tilde{v}})^{a_j} \psi_{j,\tilde{v}}(\varpi_v^{e(\eta_{j,\tilde{v}})})^{a_j} \nonumber \\ 
&=\prod_{j=1}^s\prod_{\substack{\tilde{v} \in \Sigma_{F_j , p} \\ \tilde{v} \mid v}}  
\psi_{j,\tilde{v}}(\varpi_v^{e(\eta_{j,\tilde{v}})})^{a_j} 
\left(\int_{\mathfrak{P}_{\tilde{v}}^{-e(\eta_{j,\tilde{v}})}} \eta_{j,\tilde{v}}^{-1}(x)\mathbf{e}_{F,v}(x)dx_{\tilde{v}}\right)^{a_j} 
\nonumber \\ 
&=\prod_{j=1}^s\prod_{\substack{\tilde{v} \in \Sigma_{F_j , p} \\ \tilde{v} \mid v}}  \psi_{j,\tilde{v}}(\varpi_v^{e(\eta_{j,\tilde{v}})})^{a_j} \nonumber \\
&\hspace*{5em}\cdot \left(\int_{\mathfrak{r}_{F_{j,\tilde{v}}}} \eta_{j,\tilde{v}}^{-1}(\varpi_v^{-e(\eta_{j,\tilde{v}})}x)\mathbf{e}_{F,v}(\varpi_v^{-e(\eta_{j,\tilde{v}})}x)d(\varpi_v^{-e(\eta_{j,\tilde{v}})}x_{\tilde{v}})\right)^{a_j} \nonumber \\
&=\prod_{j=1}^s \prod_{\substack{\tilde{v} \in \Sigma_{F_j , p} \\ \tilde{v} \mid v}} \psi_{j,\tilde{v}}(\varpi_v^{e(\eta_{j,\tilde{v}})})^{a_j} \eta_{j,\tilde{v}}(\varpi_v^{e(\eta_{j,\tilde{v}})})^{a_j} \nonumber \\
&\hspace*{10em}\cdot\left(\sum_{x\in (\mathfrak{r}_{F_j}/\mathfrak{P}_{\tilde{v}}^{e(\eta_{j,\tilde{v}})})^\times}
\eta_{j,\tilde{v}}^{-1}(x)\mathbf{e}_{F_j,\tilde{v}}(\varpi_v^{-e(\eta_{j,\tilde{v}})}x)\right)^{a_j}.  \nonumber
\end{align}
 Here $\varpi_v$ is a uniformiser of $F_v$ and we also regard it as a uniformiser of $F_{j,\tilde{v}}$  using unramifiedness of $F_{j,\tilde{v}}/F_v$. The first equality follows from \cite[(5.5.3)]{Deligne_constant}, and the second equality follows from \cite[(3.4.3.2)]{Deligne_constant}.
Note that $e(\eta_{j,\tilde{v}})$ coincides with $e(\eta_v)$, the exponent of the conductor  of $\eta_v$, since the norm map $\mathrm{Norm}_{F_{j,\tilde{v}}/F_v}$ induces a surjection $1+\mathfrak{P}_{\tilde{v}}^k\twoheadrightarrow 1+\mathfrak{P}_v^k$
 for every natural number $k$. See \cite[Chapitre V, Proposition 3]{serre04}; recall again that $F_{j,\tilde{v}}/F_v$ is unramified. 

We may calculate the product concerning $\eta_{j,\tilde{v}}$'s further by using a generalisation of Hasse and Davenport's relation for characters with prime power conductor, whose proof shall be given in Appendix~\ref{app:DH}. Indeed, we may apply Theorem~\ref{theorem:D-H_relation_overW_n} by setting 
\begin{align*}
 s&=f_{\tilde{v}\mid v}, &  \mathbb{F}_q &= \mathfrak{r}_{F} /\mathfrak{P}_v, & \mathbb{F}_{q^s}&=\mathfrak{r}_{F_j}/\mathfrak{P}_{\tilde{v}}, \\ \chi &= \eta_v,  &
n&=e (\eta_v), & W_n (\mathbb{F}_q ) &= \mathfrak{r}_{F}/\mathfrak{P}_v^{e(\eta_v)}, \\  W_n(\mathbb{F}_{q^s}) &=\mathfrak{r}_{F_j}/\mathfrak{P}_{\tilde{v}}^{e(\eta_v)}, & \psi_{n,\mathbb{F}_q}(x) &=\mathbf{e}_{F,v}(\varpi_v^{-e(\eta_v)} x) 
\end{align*}
where $f_{\tilde{v}\vert v}=[F_{j,\tilde{v}}:F_v]$ denotes the inertia degree of $\tilde{v}\vert v$. Note that $\mathbf{e}_{F_j,\tilde{v}}(\varpi_v^{-e(\eta_v)}x)$ coincides with $\mathbf{e}_{F,v}(\varpi_v^{-e(\eta_v)}\mathrm{Tr}_{F_{j,\tilde{v}}/F_v}(x))$ since $\varpi_v$ is contained in  $F_v$, and thus one may regard $\mathbf{e}_{F_j,\tilde{v}}(\varpi_v^{-e(\eta_v)}x)$ as $\psi_{n,\mathbb{F}_{q^s}}(x)$ appearing in Appendix~\ref{app:DH}. Applying Theorem~\ref{theorem:D-H_relation_overW_n}, we have 
\begin{align}
& \prod_{j=1}^s \prod_{\substack{\tilde{v}\in \Sigma_{F_j,p} \\ \tilde{v}\mid v}} \eta_{v}\circ \mathrm{Norm}_{F_{j,\tilde{v}}}(\varpi_v^{e(\eta_v)})^{a_j}
\left( 
\sum_{x\in  (\mathfrak{r}_{F_j}/\mathfrak{P}_{\tilde{v}}^{e(\eta_v)})^\times}
\eta_{j,\tilde{v}}^{-1}(x)\mathbf{e}_{F_j,\tilde{v}} 
(\varpi_v^{-e(\eta_v)}x) 
\right) ^{ a_j} \label{equation:localepsilon1} \\ 
& =  \prod_{j=1}^s \prod_{\substack{\tilde{v} \in \Sigma_{F_j , p}
\\ \tilde{v}\mid v}} \eta_v(\varpi_v^{e(\eta_v)})^{a_jf_{\tilde{v}\vert v}} \nonumber
\\
&\hspace*{4em} \cdot   
\left( (-1)^{e (\eta_v)(f_{\tilde{v}\vert v}-1)}  \left(\sum_{x\in (\mathfrak{r}_{F}/\mathfrak{P}_v^{e(\eta_v)})^\times}\eta_v^{-1}(x)\mathbf{e}_{F,v}(\varpi_v^{-e(\eta_v)}x) \right)^{f_{\tilde{v}\vert v}}\right)^{ a_j} \nonumber \\
&\hspace{25em}  (\text{by Theorem~\ref{theorem:D-H_relation_overW_n}}) \nonumber \\ 
&=  
(-1)^{\sum_{j=1}^s a_j \sum_{\tilde{v}\mid v} e(\eta_v) (f_{\tilde{v}\mid v}-1)}
\nonumber \\
&\hspace*{4em} \cdot   
\left(  \eta_v(\varpi_v^{e(\eta_v)})
 \sum_{x\in (\mathfrak{r}_{F}/\mathfrak{P}_v^{e(\eta_v)})^\times}
\eta_v^{-1} (x)\mathbf{e}_{F,v}(\varpi_v^{-e(\eta_v)}x)\right)^{\sum_{j=1}^s a_j\sum_{\tilde{v}\mid v} f_{\tilde{v}\mid v} } \nonumber \\
&=(-1)^{\sum_{j=1}^sa_j \sum_{\tilde{v}\mid v}e(\eta_v) (f_{\tilde{v}\mid v}-1)}\epsilon(\eta_v,\mathbf{e}_{F,v}a,dx_v)^{r(\rho)}, \nonumber 
\end{align}
where $\sum_{\tilde{v}\mid v}$ denotes the summation over $\tilde{v}\in \Sigma_{F_i,p}$ lying above $v$. The last equality follows from \cite[(5.5.3)]{Deligne_constant} and the basic equality on extension degrees 
\begin{align*}
 r(\rho) = \sum_{j=1}^s a_j[F_j:F]=\sum_{j=1}^s \sum_{\tilde{v}\mid v} a_j f_{\tilde{v}\mid v},
\end{align*}
which follows from unramifiedness of $F_{j,\tilde{v}}/F_v$. Although there appears an unexpected signature $(-1)^{\sum_{j=1}^s a_j\sum_{\tilde{v}\mid v}e(\eta_v)f_{\tilde{v}\mid v}}$, it will be finally cancelled out with the signature appearing in the computation of the product concerning $\psi_j$'s.

Next, in order to calculate the product concerning $\psi_j$'s in \eqref{equation:epsilonfactor_into_two_terms}, we study the action of $\mathrm{Frob}_v$ on each component 
\begin{align*}
 \mathrm{Ind}^{G_{F_v}}_{G^g_{F_j}\cap G_{F_v}}(\psi_j^\mathrm{gal})^g:=\mathrm{Hom}_{\overline{\mathbb{Q}}[G_{F_j}^g\cap G_{F_v}]} (\overline{\mathbb{Q}}[G_{F_v}], \overline{\mathbb{Q}}((\psi_j^\mathrm{gal})^g))
\end{align*}
appearing in the right-hand side of \eqref{equation:Mackey}.  Let $\tilde{v}_0$ be a unique place of $F_j$ fixed by $G_{F_v}$, as before. Then $G_{F_j}^g\cap G_{F_v}$ is naturally regarded as the conjugate $G_{F_{j,\tilde{v}_o^g}}^g$ of the decomposition subgroup  $G_{F_{j,\tilde{v}_0^g}}$ of $G_{F_j}$ at $\tilde{v}_0^g$; in particular, the degree of $\mathrm{Ind}_{G_{F_j}^g \cap G_{F_v}}^{G_{F_v}} (\psi_j^\mathrm{gal})^g$ coincides with the inertia degree $f_{\tilde{v}_0^g\mid v}$, and $g\bigl(\mathrm{Frob}_v^{f_{\tilde{v}_0^g\mid v}}\bigr)g^{-1}$ is identified with $\mathrm{Frob}_{\tilde{v}_0^g}$. Now, for $k=0,1,\ldots,f_{\tilde{v}_0^g\mid v}$, define $\phi_k\colon \overline{\mathbb{Q}}[G_{F_v}]\rightarrow \overline{\mathbb{Q}}((\psi_j^\mathrm{gal})^g)$ as a $\overline{\mathbb{Q}}[G_{F_j}^g\cap G_{F_v}]$-equivariant homomorphism sending $\mathrm{Frob}_v^{-k}$ to $1$ and $\mathrm{Frob}_v^{-k'}$ to $0$ for $k'=0,1,\ldots,f_{\tilde{v}_0^g\mid v}-1$ with $k'\neq k$. Then $\phi_0,\phi_1,\ldots,\phi_{f_{\tilde{v}_0^g\mid v}-1}$ forms a basis of $\mathrm{Ind}^{G_{F_v}}_{G^g_{F_j}\cap G_{F_v}}(\psi_j^\mathrm{gal})^g$, and the matrix presentation of the action of $\mathrm{Frob}_v$ on it is given by 
\begin{align*}	
\begin{bmatrix}
  & (\psi_j^\mathrm{gal})^g(\mathrm{Frob}_v^{f_{\tilde{v}_0^g\mid v}}) \\ 
I_{f_{\tilde{v}_0^g\mid v}-1} &  \\  
\end{bmatrix}=
\begin{bmatrix}
  &  \psi_j^\mathrm{gal}(\mathrm{Frob}_{\tilde{v}_0^g})\\ 
I_{f_{\tilde{v}_0^g\mid v}-1} &  \\  
\end{bmatrix}
\end{align*}
where $I_{f_{\tilde{v}_0^g\mid v}-1}$ denotes the identity matrix of degree $f_{\tilde{v}_0^g\mid v}-1$. 
Taking the determinant, we obtain an equality
\begin{align} \label{equation:Frob_action_induced}
 \det \bigl(\mathrm{Frob}_v;\, \mathrm{Ind}_{G_{F_j}^g\cap G_{F_v}}^{G_{F_v}}\, (\psi_j^\mathrm{gal})^g\bigr) &=(-1)^{f_{\tilde{v}_0^g\mid v}-1} \psi_{j,\tilde{v}_0^g}(\varpi_v).
\end{align}
Combining \eqref{equation:Frob_action_induced} with Mackey decomposition \eqref{equation:Mackey}, we may calculate as
\begin{align}
 \prod_{j=1}^s &\prod_{\substack{\tilde{v} \in \Sigma_{F_i , p} \\ \tilde{v}\mid v}}  \psi_{j,\tilde{v}} (\varpi_v^{e(\eta_v)})^{ a_i} 
=\prod_{j=1}^s \prod_{[g]\in G_{F_j}\backslash G_F/G_{F_v}}  \psi_{j,\tilde{v}_0^g} (\varpi_v^{e(\eta_v)})^{ a_j}  \label{equation:localepsilon2} \\
&=\prod_{j=1}^s \prod_{[g]\in G_{F_j}\backslash G_F/G_{F_v}} \!\! \left\{(-1)^{e(\eta_v)(f_{\tilde{v}_0^g\mid v}-1)} \! \det (\mathrm{Frob}_v^{e(\eta_v)}; \mathrm{Ind}_{G_{F_j}^g\cap G_{F_v}}^{G_{F_v}}(\psi_j^\mathrm{gal})^g) \!\right\}^{a_j} \nonumber\\
&\hspace*{27em} (\text{by \eqref{equation:Frob_action_induced}}) \nonumber \\
&= (-1)^{\sum_{j=1}^sa_j \sum_{\tilde{v}\mid v} e(\eta_v)(f_{\tilde{v}\mid v}-1)} \nonumber \\
&\hspace*{6.5em} \cdot  \det \left(\mathrm{Frob}_v^{e(\eta_v)};\, \sum_{j=1}^s a_j\sum_{[g]\in G_{F_j}\backslash G_F/G_{F_v}}\mathrm{Ind}_{G_{F_j}^g\cap G_{F_v}}^{G_{F_v}}(\psi_j^\mathrm{gal})^g\right) \nonumber\\
&= (-1)^{\sum_{j=1}^sa_j \sum_{\tilde{v}\mid v} e(\eta_v)(f_{\tilde{v}\mid v}-1)} \det\left(\mathrm{Frob}_v^{e(\eta_v)} \,;\, V_\rho\vert_{G_{F_v}}\right). \nonumber
\end{align} 
The last equality follows from \eqref{equation:brauer_induction} and \eqref{equation:Mackey}. We here emphasise that the signature appearing in \eqref{equation:localepsilon2} completely coincides with that in \eqref{equation:localepsilon1}.
Therefore, by putting \eqref{equation:epsilonfactor_into_two_terms}, \eqref{equation:localepsilon1} and \eqref{equation:localepsilon2} together, we finally obtain 
\begin{align*}
 \prod_{j=1}^s \prod_{\substack{\tilde{v}\in \Sigma_{F_j,p} \\ \tilde{v}\mid v}} \epsilon((\psi_j\eta_j)_{\tilde{v}},\mathbf{e}^{-1}_{F_j,\tilde{v}},dx_{\tilde{v}})^{a_j}&=\epsilon(\eta_v,\mathbf{e}_{F,v}^{-1},dx_v)^{r(\rho)}\det\bigl(\mathrm{Frob}_v^{e(\eta_v)}\, ; \, V_\rho\bigr) \\
&=\epsilon((\rho\otimes \eta)_v,\mathbf{e}_{F,v}^{-1},dx_v),
\end{align*}
due to \cite[(5.5.3)]{Deligne_constant}.
\end{component}
We thus complete the proof of Theorem \ref{theorem:CMp-adicL_Artin}. 
\end{proof}

\begin{rem}[Uniqueness of the $p$-adic $L$-function]
 The interpolation formula \eqref{eq:interpolation_Lp_Artin} uniquely characterises the $p$-adic Artin $L$-function $L_{p,\Sigma_F}(M(\rho))$ due to the $p$-adic identity theorem \cite[Theorem (4.1.2)]{katz}; indeed we may readily extend \cite[Theorem (4.1.2)]{katz} to elements of the field of fractions of the Iwasawa algebra $\widehat{\mathcal{O}}^\mathrm{ur}[[\Gamma_{F,\, \max}]]$.  In particular $L_{p,\Sigma_F}(M(\rho))$ does not depend on any Brauer decomposition \eqref{equation:brauer_induction} of $\rho$ which we choose in the proof of Theorem~\ref{theorem:CMp-adicL_Artin}, because the interpolation formula \eqref{eq:interpolation_Lp_Artin} does not contain any data concerning a particular choice of the decomposition \eqref{equation:brauer_induction}. 
\end{rem}

The integrality conjecture below is regarded as a $p$-adic counterpart of Artin's conjecture for (complex) Artin $L$-functions.

\begin{conj}[$p$-adic Artin conjecture] \label{conjecture:integrality}

The $p$-adic Artin  $L$-function $L_{p,\Sigma_F} (M(\rho ))$, which is constructed  as an element of $\mathrm{Frac} (\widehat{\mathcal{O}}^{\mathrm{ur}}[[\Gamma_{F,\, \max}]])$ in Theorem \ref{theorem:CMp-adicL_Artin}, 
is indeed an element of $\widehat{\mathcal{O}}^{\mathrm{ur}}[[\Gamma_{F,\, \max}]]$. 

\end{conj}

\section[Integrality of the $p$-adic Artin $L$-function   and the Iwasawa main conjecture]{Integrality of the $p$-adic Artin $L$-function  and the Iwasawa main conjecture}

In Section~\ref{ssc:IMC_MAIN}, we formulate the Iwasawa main conjecture for Artin representations on CM fields and state the main result of the present article (Theorem~\ref{theorem:integrality_of_CMp-adicL_Artin}). After studying several algebraic properties of Selmer groups in Section~\ref{ssc:alg_pre}, we shall give the proof of Theorem~\ref{theorem:integrality_of_CMp-adicL_Artin} in Section~\ref{ssc:integrality}.

\subsection{Iwasawa main conjectures and main results} \label{ssc:IMC_MAIN}
First of all, let us introduce the notion of several Selmer groups concerning Artin motives. 
Let $\rho\colon G_F\rightarrow {\rm Aut}_{E}(V_\rho)$ be an Artin representation, and take a finite extension $\mathcal{E}$ of $\mathbb{Q}_p$ containing $\boldsymbol{\iota}\circ \iota_\infty(E)$. In the following, we use the same symbol $V_\rho$ for its scalar extension $V_\rho\otimes_E \mathcal{E}$ and regard it as a $p$-adic representation for simplicity. 
Let $\mathcal{O}$ be the ring of integers of $\mathcal{E}$. 
Choose a $G_{F}$-stable $\mathcal{O}$-lattice $T$ of $V_\rho$ and set $A =T \otimes_{\mathbb{Z}_p} \mathbb{Q}_p /\mathbb{Z}_p$.  
Let $S_F$ be the set of places of $F$ consisting of all the archimedean places, the places dividing $(p)$ and those at which $\rho$ ramifies. For any algebraic extension $F'$ of $F$, let $S_{F'}$ be the set of places of $F'$ lying above those contained in $S_F$. 
When $[F':F]$ is finite, we define the {\em $(\Sigma_{F,p}\text{-ramified})$ Selmer group} $\mathrm{Sel}_{A} (F') $ of $A$ as 
\begin{equation}\label{equation:definitionSelmergp_rho}
\mathrm{Sel}_{A,\Sigma_F} (F') = \mathrm{Ker} 
\! \left[ \mathrm{loc}_A\colon
H^1 (F'_{S_{F'}} /F' ,A ) \rightarrow 
 \underset{w \in S_{F'} \setminus \Sigma_{F',p},\, w\nmid \infty}{\prod} 
 H^1 (F^{\prime\, \mathrm{ur}}_w ,A)
 \right],
\end{equation}
where $F'_{S_{F'}}$ is the maximal Galois extension of $F'$ unramified outside the places of $S_{F'}$, $F_w^{\prime\, \mathrm{ur}}$ is the maximal unramified extension of $F_w'$ and $\Sigma_{F',p}$ is the set of places of $F'$ lying above those contained in $\Sigma_{F,p}$. 
When $F'$ is an algebraic extension of $F$ of infinite degree, we define $\mathrm{Sel}_{A,\Sigma_F} (F') $ by taking the inductive limit of the Selmer groups 
defined over intermediate extensions of finite degree over $F$. 
Similarly, when $F'$ is a finite extension of $F$, we define the {\em $(\Sigma_{F,p}\text{-ramified})$ strict Selmer group} of $A$ as 
\begin{equation} \label{eq:str_Sel}
\mathrm{Sel}^{\mathrm{str}}_{A,\Sigma_F} (F') = \mathrm{Ker} \!
\left[  \mathrm{loc}_A^{\mathrm{str}}\colon
H^1 (F'_{S_{F'}}/F' ,A )  \rightarrow \!
 \underset{w \in S_{F'} \setminus \Sigma_{F',p},\, w\nmid \infty}{\prod} 
 H^1 (F^{\prime}_w ,A)
 \right].
\end{equation}
By taking the inductive limit, we may also define the strict Selmer group $\mathrm{Sel}^{\mathrm{str}}_{A,\Sigma_F} (F') $ even when $F'$ is an algebraic extension of $F$ of infinite degree. Using Shapiro's lemma \cite[Proposition (1.6.4)]{NSW}, we readily observe that $\mathrm{Sel}_{A,\Sigma_F}(F_{\max})$ and $\mathrm{Sel}^{\mathrm{str}}_{A,\Sigma_F}(F_{\max})$ are respectively isomorphic to the corresponding Selmer groups
\begin{align}
 \mathrm{Sel}_{\mathcal{A},\Sigma_F} (F) &= \mathrm{Ker} 
\! \left[ 
H^1 (F_{S_{F}} /F ,\mathcal{A} ) \rightarrow 
 \underset{v \in S_{F} \setminus \Sigma_{F,p},\, v\nmid \infty}{\prod} 
 H^1 (F_v^{\mathrm{ur}} ,\mathcal{A})
 \right],  \label{eq:Sel_def} \\
 \mathrm{Sel}^{\mathrm{str}}_{\mathcal{A},\Sigma_F} (F) &= \mathrm{Ker} 
\! \left[ 
H^1 (F_{S_{F}} /F ,\mathcal{A} ) \rightarrow 
 \underset{v \in S_{F} \setminus \Sigma_{F,p},\, v\nmid \infty}{\prod} 
 H^1 (F_v ,\mathcal{A})
 \right]  \label{eq:str_Sel_def}
\end{align}
of $\mathcal{A}:=\mathcal{T}\otimes_{\mathcal{O}[[\Gamma_{F,\,\max}]]}\mathcal{O}[[\Gamma_{F,\,\max}]]^\vee$, where $\mathcal{T}$ is defined as $T\otimes_{\mathcal{O}} \mathcal{O}[[\Gamma_{F,\,\max}]]^\sharp$ on which $G_F$ acts diagonally; refer to Section~\ref{sc:Introduction} for the definition $\mathcal{O}[[\Gamma_{F,\,\max}]]^\sharp$.

\begin{rem}
 The strict Selmer groups defined as \eqref{eq:str_Sel} (or \eqref{eq:str_Sel_def}) is the same as those proposed at \cite[Definition 3.19]{haraochiai}, although their definition looks a bit different. See also the proof of Lemma~\ref{lemma:Sel=selstr}.
\end{rem}

The {\em cotorsionness} of Selmer groups is one of the fundamental problem in Iwasawa theory, but in our situation this is always fulfilled.

\begin{pro}[Cotorsionness of Selmer groups] \label{prop:Sel_cotorsion}
  Both the Selmer groups $\mathrm{Sel}_{A,\Sigma_F}(F_{\max})$ and $\mathrm{Sel}_{A,\Sigma_F}^\mathrm{str}(F_{\max})$ are cofinitely generated cotorsion $\mathcal{O}[[\Gamma_{F,\, \max}]]$-modules for any Artin representation $\rho \colon G_F\rightarrow \mathrm{Aut}_{E}\, V_\rho$. %for arXiv version: change the order
\end{pro}

\begin{proof}
Since $\mathrm{Sel}_{A,\Sigma_F}^{\mathrm{str}}(F_{\max})$ is a submodule of $\mathrm{Sel}_{A,\Sigma_F}(F_{\max})$, it suffices to verify the cotorsionness of $\mathrm{Sel}_{A,\Sigma_F}(F_{\max})$. Let $F_\rho$ be %denote %for arXiv version
the field corresponding to the kernel of $\rho$, as before. Since $F_{\max}F_{\rho}/F_{\max}$ is a finite extension, the restriction %map %for arXiv version 
$
\mathrm{Sel}_{A,\Sigma_F } (F_{\max}) \rightarrow \mathrm{Sel}_{A,\Sigma_F } (F_{\max}F_{\rho})
$
has an $\mathcal{O}$-cofinitely generated kernel. Hence $\mathrm{Sel}_{A,\Sigma_F } (F_{\max})^\vee$ is $\mathcal{O}[[\Gamma_{F,\, \max}]]$-torsion if $\mathrm{Sel}_{A,\Sigma_F} (F_{\max}F_{\rho})^\vee$ is $\mathcal{O}[[\mathrm{Gal}(F_{\max}F_{\rho} /F_{\rho})]]$-torsion.
Since $G_{F_{\max}F_{\rho}}$ trivially acts on $A$  by definition, we observe by taking the Pontrjagin dual of \eqref{equation:definitionSelmergp_rho} that $\mathrm{Sel}_{A,\Sigma_F}(F_{\max}F_\rho)^\vee$ is isomorphic to  $\mathrm{Gal}(M^\rho_{\Sigma_F}/F_{\max}F_\rho)$, the Galois group of the maximal abelian pro-$p$ extension of $F_{\max}F_\rho$ unramified outside the places lying above those contained in $\Sigma_{F,p}$. The Galois group $\mathrm{Gal}(M^\rho_{\Sigma_F}/F_{\max}F_{\rho})$ is then a torsion module over 
$\mathcal{O}[[\mathrm{Gal}(F_{\max}F_{\rho}  /F_{\rho})]]$ by \cite[Theorem 1.2.2 (iii)]{HT-aIMC}, as desired.
\end{proof}

Now let us fomulate the Iwasawa main conjecture for CM fields. We first consider the case where the Artin representation $\rho$ is abelian.

\begin{conj}[Iwasawa main conjecture for $\psi$]\label{conjecture:iwasawamainconjecture}
Let $F$ be a CM field and consider a finite Hecke character$\psi^\mathrm{gal}\colon G_F\rightarrow \mathrm{Aut}_{\mathcal{E}}\, V_\psi$ %a character of finite order %for arxiv version
such that the field $F_\psi$ corresponding to the kernel of $\psi^\mathrm{gal}$ is linearly disjoint from $F_{\max}$. 
Choosing a $\mathcal{E}$-basis $e_\psi$ of $V_\psi\cong \mathcal{E}$, define the standard $G_F$-stable $\mathcal{O}$-lattice of $V_\psi$ by $T_\psi:=\mathcal{O} e_\psi$, and set $A_\psi =T_\psi  \otimes_{\mathbb{Z}_p} \mathbb{Q}_p /\mathbb{Z}_p$.  Let $\mathrm{Sel}_{A_{\psi },\Sigma_F}(F_{\max})$ be the Selmer group defined as above. 
Then the equality 
\begin{align*}
 (L_{p,\Sigma_F}(\psi )) = \mathrm{char}_{\widehat{\mathcal{\mathcal{O}}}^{\mathrm{ur}}[[\Gamma_{F,\, \max}]]} 
 \bigl(\mathrm{Sel}_{A_{\psi },\Sigma_F}(F_{\max})^\vee \widehat{\otimes}_{\mathcal{O}}\widehat{\mathcal{O}}^{\mathrm{ur}} \bigr)
\end{align*}
of principal ideals of $\widehat{\mathcal{O}}^{\mathrm{ur}}[[\Gamma_{F,\, \max}]]$ should hold where the $p$-adic Hecke $L$-function $L_{p,\Sigma_F}(\psi )$ is defined as in Theorem~\ref{thm:KHT}.
\end{conj}

Note that the Pontrjagin dual of the Selmer group $\mathrm{Sel}_{A_\psi,\Sigma_F}(F_{\max})$ is pseudo-isomorphic to $\mathrm{Gal}(M_{\Sigma_F}F_\psi/F_{\max}F_\psi)_\psi$ introduced in Section~\ref{sc:Introduction} under the situation of Conjecture~\ref{conjecture:iwasawamainconjecture} (see \cite[Proposition 3.16]{haraochiai} for example), and thus Conjecture~\ref{conjecture:iwasawamainconjecture} is equivalent to Main conjecture proposed by Hida and Tilouine in \cite{HT-aIMC}.
The following conjecture is a direct generalisation of Conjecture~\ref{conjecture:iwasawamainconjecture} for  higher-dimensional Artin representations.

\begin{conj}[Iwasawa main conjecture for $\rho$]\label{conjecture:iwasawamainconjectureartin}
Let $F$ be a CM field and $\rho\colon G_F\rightarrow \mathrm{Aut}_{\mathcal{E}}(V_\rho)$ an Artin representation of $G_F$ such that the field $F_\rho$ corresponding to the kernel of $\rho$ is linearly disjoint from $F_{\max}$. Let $\mathcal{O}$ be the ring of integers of $\mathcal{E}$. Choose a $G_{F}$-stable $\mathcal{O}$-lattice $T$ of $V_\rho$ and set $A =T \otimes_{\mathbb{Z}_p} \mathbb{Q}_p /\mathbb{Z}_p$. Let $\mathrm{Sel}_{A,\Sigma_F}(F_{\max})$ be the Selmer group defined as above. 
Then the equality
\begin{equation}\label{equation:IMC_for_Artinrep}
(L_{p ,\Sigma_F} (M(\rho ))) = 
\mathrm{char}_{\widehat{\mathcal{O}}^\mathrm{ur}[[\Gamma_{F,\, \max}]]} \bigl(\mathrm{Sel}_{A,\Sigma_F} (F_{\max})^\vee 
\widehat{\otimes}_{\mathcal{O}} \widehat{\mathcal{O}}^{\mathrm{ur}} \bigr).
\end{equation}
of principal fractional ideals of $\widehat{\mathcal{O}}^\mathrm{ur}[[\Gamma_{F,\, \max}]]$ should hold. 
\end{conj}

\begin{rem}
\begin{enumerate}
\item 
The left-hand side of \eqref{equation:IMC_for_Artinrep} is defined independently of the choice of a lattice $T$, but 
the right-hand side of \eqref{equation:IMC_for_Artinrep}  a priori depends on the choice of $T$. 
In fact, we can check that the right-hand side of \eqref{equation:IMC_for_Artinrep} is independent of the choice of a lattice $T$ by 
Lemma \ref{lemma:mu_lattice}. 
\item 
A variant of Conjecture \ref{conjecture:iwasawamainconjectureartin} formulated for $\mathrm{Sel}^\mathrm{str}_{A,\Sigma_F} (F_{\max})$ in place of $\mathrm{Sel}_{A,\Sigma_F} (F_{\max})$  is equivalent to Conjecture \ref{conjecture:iwasawamainconjectureartin} by Lemma \ref{lemma:Sel=selstr}.
\item 
Note that, in Theorem \ref{theorem:CMp-adicL_Artin}, the $p$-adic Artin $L$-function $L_{p,\Sigma_F} (M(\rho ))$ is constructed as an element of 
$\mathrm{Frac} (\widehat{\mathcal{O}}^{\mathrm{ur}}[[\Gamma_{F,\, \max}]])$.  
Since the right-hand side of \eqref{equation:IMC_for_Artinrep} is an integral ideal of 
$\widehat{\mathcal{O}}^\mathrm{ur}[[\Gamma_{F,\, \max}]]$, the validity of \eqref{equation:IMC_for_Artinrep} 
implies that $L_{p,\Sigma_F} (M(\rho ))$ is an integral element of 
$\widehat{\mathcal{O}}^{\mathrm{ur}}[[\Gamma_{F,\, \max}]]$. 
\end{enumerate}
\end{rem}

We are now ready to state the main result of the present article.

\begin{thm}[Main Theorem]\label{theorem:integrality_of_CMp-adicL_Artin}
Let $M(\rho)$ be the Artin motive corresponding to an Artin representation $\rho$ of $G_F$ which is unramified at any prime ideal lying above $(p)$, and suppose that the field $F_\rho$ corresponding to the kernel of $\rho$ is also a CM field. Assume further that the Iwasawa main conjecture $($Conjecture $\ref{conjecture:iwasawamainconjecture})$ is true for any intermediate field $K$ of $F_\rho /F$ and branch characters factoring through $\mathrm{Gal}(F_\rho/K)$.
Then Conjectures $\ref{conjecture:integrality}$ and $\ref{conjecture:iwasawamainconjectureartin}$ hold true for~$\rho$.
\end{thm}

\begin{rem}[On the assumption of Theorem~\ref{theorem:integrality_of_CMp-adicL_Artin}] \label{rem:failure_ab_IMC}
Concerning Conjecture~\ref{conjecture:iwasawamainconjecture} for an intermediate field $K$ of $F_\rho/F$ and a branch character $\psi^\mathrm{gal}$ factoring through $\mathrm{Gal}(F_\rho/K)$, the known results which we may apply are only \cite[Theorem 8.17]{Hsieh14} and \cite[Theorem 8.18]{Hsieh14} at the present. However, the former result \cite[Theorem 8.17]{Hsieh14} does not work at all in our situation since  the branch character $\psi^\mathrm{gal}$ will never become anticyclotomic; indeed $\psi^{\mathrm{gal}}(cgc^{-1})=\psi^{\mathrm{gal}}(g)$ holds for every $g\in \mathrm{Gal}(F_\rho/K)$ by Lemma~\ref{lem:CM_ext}. Meanwhile, the latter result \cite[Theorem 8.18]{Hsieh14} requires that $F$ is a composite field of a totally real field and an imaginary quadratic field $\mathsf{M}$, and the field $K_\psi$ corresponding to the kernel of $\psi^\mathrm{gal}$ should be abelian over $\mathsf{M}$. This too strict constraint makes it difficult to apply \cite[Theorem 8.18]{Hsieh14} to characters $\psi_j^\mathrm{gal}$ appearing in the Brauer decomposition \eqref{equation:brauer_induction}. After all it seems hard to give an explicit non-commutative Artin representation $\rho$ satisfying Conjectures $\ref{conjecture:integrality}$ and $\ref{conjecture:iwasawamainconjectureartin}$ with current knowledge of the (abelian) Iwasawa main conjecture of CM fields.
\end{rem}

\begin{rem}[On the case $p=2$]
Although we assume that the prime number $p$ is {\em odd} throughout the present article, we would like to briefly explain here the situation of the case $p=2$. 
Indeed, the only issue concerning the prime $2$ is that, due to Hasse's unit index theorem \cite[Satz~14]{Hasse}, both the numerator and the denominator of the modification factor $\frac{2^{d[F_j:F]}\sqrt{\lvert D_{F_j^+}\rvert}}{(\mathfrak{r}_{F_j}^\times :\, \mathfrak{r}_{F^+_j}^\times)}$ appearing in \eqref{equation:katz3} might be divisible by $2$ even if $F$ is assumed to be absolutely unramified at $(2)$. Therefore we can still verify by the same proof that the $2$-adic Artin $L$-function $L_{2,\Sigma_F}(M(\rho))$ 
is an element of $\widehat{\mathcal{O}}^\mathrm{ur}[[\Gamma_{F,\, \max}]]\otimes_{\mathbb{Z}}\mathbb{Q}$ and the main conjecture \eqref{equation:IMC_for_Artinrep} holds true up to $\mu$-invariants if Conjecture~\ref{conjecture:iwasawamainconjecture} holds true for $p=2$.
\end{rem}

\subsection{Algebraic preliminaries} \label{ssc:alg_pre}

Before proving Theorem~\ref{theorem:integrality_of_CMp-adicL_Artin}, we prepare several notation and lemmas. 

\begin{lem}\label{lemma:Sel=selstr}
 As before, choose a $G_{F}$-stable $\mathcal{O}$-lattice $T$ of the Artin representation $V_\rho$ and set $A =T \otimes_{\mathbb{Z}_p} \mathbb{Q}_p /\mathbb{Z}_p$.  
Then the natural surjection  
$$
\mathrm{Sel}_{A,\Sigma_F} (F_{\max})^\vee \twoheadrightarrow \mathrm{Sel}^{\mathrm{str}}_{A,\Sigma_F} (F_{\max})^\vee
$$
is a pseudo-isomorphism of $\mathcal{O}[[\Gamma_{F,\, \max}]]$-modules. 
Also for any finite extension $F'$ of $F$,  
the natural surjection
\begin{align*}
 \mathrm{Sel}_{A,\Sigma_F} (F'F_{\max})^\vee &\twoheadrightarrow \mathrm{Sel}^{\mathrm{str}}_{A,\Sigma_F} (F'F_{\max})^\vee \\ 
( \mathrm{resp}.\ \mathrm{Sel}_{A,\Sigma_F} (F'_{\max})^\vee &\twoheadrightarrow \mathrm{Sel}^{\mathrm{str}}_{A,\Sigma_F} (F'_{\max})^\vee  \quad 
)
\end{align*}
is a pseudo-isomorphism of $\mathcal{O}[[\mathrm{Gal}(F'F_{\max} /F')]]$-modules $($resp.\ a pseudo-isomorphism of $\mathcal{O}[[\Gamma_{F',\, \max}]]$-modules$)$. 
\end{lem}

To prove Lemma~\ref{lemma:Sel=selstr}, we introduce notation 
\begin{align*}
\mathrm{Loc}^i_M (K) &= \underset{\substack{w \in S_K\setminus \Sigma_{K,p} \\ w\nmid p}}{\prod} 
H^i (K_w^{\mathrm{ur}} ,M), & \mathrm{Loc}^{\mathrm{str},i}_M (K) &= \underset{\substack{w \in S_K\setminus \Sigma_{K,p} \\ w\nmid p}}{\prod} 
 H^i (K_w ,M)
\end{align*}
on local conditions for $i=0,1$ and $2$, where $K$ is an intermediate field of $F_{\max}/F$ and $M$ is an arbitrary $\mathrm{Gal}(F_{S}/F)$-subquotient of $A$. 

\begin{proof}[Proof of Lemma~$\ref{lemma:Sel=selstr}$]
Consider the commutative diagram with exact rows: 
\begin{equation}
\begin{CD}
0 @>>> \mathrm{Sel}_{A,\Sigma_F} (F_{\max}) @>>> H^1 (F_S / F_{\max} ,A)  @>{\mathrm{loc}_A}>> \mathrm{Loc}^1_{A} (F_{\max})  \\ 
@. @A{f}AA @AA{\rotatebox{90}{$=$}}A @AA{h}A\\ 
0 @>>> \mathrm{Sel}^{\mathrm{str}}_{A,\Sigma_F} (F_{\max}) @>>> H^1 (F_S / F_{\max} ,A)  @>>{\mathrm{loc}^{\mathrm{str}}_{A}}> \mathrm{Loc}^{\mathrm{str},1}_{A} (F_{\max}). 
\\ 
\end{CD}
\end{equation}
The maps $\mathrm{loc}_A$ and $\mathrm{loc}_A^\mathrm{str}$ are the restriction morphisms. By construction, the left vertical map $f$ is injective. In order to verify the assertion for $\mathrm{Sel}_{A,\Sigma_F} (F_{\max})^\vee \twoheadrightarrow \mathrm{Sel}^{\mathrm{str}}_{A,\Sigma_F} (F_{\max})^\vee$, it suffices to show that $(\mathrm{Coker}\, f)^\vee$ is a pseudonull $\mathcal{O}[[\Gamma_{F,\, \max}]]$-module. 
The right vertical map $h$ is also induced by the restriction maps and thus has the kernel isomorphic to
\begin{align*}
 \mathrm{Ker}(h)
  & \cong \underset{\substack{ v \in S_F \\  v \nmid p\infty}}{\prod} \prod_{\substack{w\mid v \\ w: \text{ place of } F_{\max}}} H^1_{\mathrm{ur}}(F_{\max,w},A) \times  \underset{\substack{v \in \Sigma^c_{F,p}}}{\prod}  H^1_{\mathrm{ur}} (F_v ,\mathcal{A}).
\end{align*}
Here we rewrite the local condition at $v\in \Sigma_{F,p}^c$ as in \eqref{eq:Sel_def} and \eqref{eq:str_Sel_def}. Then the unramified cohomology $H^1_{\mathrm{ur}}(F_{\max,w}, A)=H_0(F_v^{\mathrm{ur}}/F_{\max,w},A^{I_w})$ is trivial for any $v \in S_F$ with $v\nmid p\infty$ because $\mathrm{Gal}(F_v^{\mathrm{ur}}/F_{\max,w})\cong \prod_{\ell \neq p}\mathbb{Z}_{\ell}$ is a pro-prime-to-$p$ group whereas $A^{I_w}$ is a discrete $p$-group. 
Now consider the case $v\in \Sigma_{F,p}^c$. Then the $\mathbb{Z}_p$-rank of the decomposition subgroup of $\Gamma_{F,\,\max}$ at $v$ is equal to or greater than $2$ by \cite[Lemma~3.1 (i)]{BCGKST}, which acts trivially on $H^1_{\mathrm{ur}}(F_v,\mathcal{A})$ by the very definition $H_0(F_v^{\mathrm{ur}}/F_{\max,w},A^{I_w})$ of the unramified cohomology group. Therefore $H^1_{\mathrm{ur}}(F_v,\mathcal{A})$  admits two coprime annihilators in $\mathcal{O}[[\Gamma_{F,\,\max}]]$, which implies that $(\mathrm{Ker}\, h)^\vee$ is $\mathcal{O}[[\Gamma_{F,\,\max}]]$-pseudonull. Then $(\mathrm{Coker}\, f)^\vee$ is also $\mathcal{O}[[\Gamma_{F,\,\max}]]$-pseudonull as required since it is a subquotient of $(\mathrm{Ker}\, h)^\vee$. One may similarly the other assertions, also using \cite[Lemma~3.1 (i)]{BCGKST}. 
\end{proof}

Recall that $F_S$ denotes the maximal Galois extension of $F$ unramified outside the places contained in $S_F$.

\begin{lem}[Surjectivity of the localisation map]\label{lemma:surjectivity_localization_map}
Let $F'$ be a finite extension of $F$ %contained %for arXiv version
in $F_S$, and suppose that $F'$ is also a CM field. 
For a $p$-adic Artin representation $\tau \colon  G_{F'} \rightarrow \mathrm{Aut}_{\mathcal{E}}\, V_{\tau}$ of $G_{F'}$, let $\mathcal{O}$ denote the ring of integers of $\mathcal{E}$. 
Choose a $G_{F'}$-stable $\mathcal{O}$-lattice $T$ and set $A=V_\tau/T$. 
Then the localisation map $\mathrm{loc}_A^\mathrm{str}\colon H^1 (F_S / F'_{\max} ,A) \rightarrow \mathrm{Loc}^{\mathrm{str},1}_A (F'_{\max}) $ defining $\mathrm{Sel}_{A,\Sigma_{F'}}^{\mathrm{str}}(F'_{\max})$ is surjective. 
\end{lem}
\begin{proof}
For any finite extension $K$ of $F'$ contained in $F_S$, set  
\begin{align*}
 \overline{\mathrm{Loc}^{\mathrm{str},1}_A (K)} = \underset{\substack{\lambda \in S_K \\  \lambda \nmid p\infty}}{\prod} 
 \frac{H^1 (K_{\lambda} ,A)}{H^1_{\mathrm{ur}} (K_{\lambda} ,A)_{\mathrm{div}}}  \times \underset{\substack{\mathfrak{p}\in \Sigma^c_{p ,K}}}{\prod}  H^1 (K_{\mathfrak{p}} ,A)
\end{align*}
where $H^1_{\mathrm{ur}} (K_{\lambda} ,A)_{\mathrm{div}}$ denotes the maximal divisible subgroup of 
$H^1_{\mathrm{ur}} (K_{\lambda} ,A)$. We introduce a similar local condition 
\begin{align*}
 \overline{\mathrm{Loc}^{\mathrm{str},1}_{A^\ast} (K)}
 =  \underset{\substack{\lambda \in S_K \\  \lambda \nmid p\infty}}{\prod} 
 \frac{H^1 (K_{\lambda} ,A^\ast)}{H^1_{\mathrm{ur}} (K_{\lambda} ,A^\ast)_{\mathrm{div}}} \times \underset{\substack{\mathfrak{p}\in \Sigma_{p ,K}}}{\prod}  H^1 (K_{\mathfrak{p}} ,A^\ast)
\end{align*}
for the Kummer dual $A^\ast = \mathrm{Hom}_{\mathbb{Z}_p} (T ,\mu_{p^\infty})$ of $T$. We define a Bloch--Kato type Selmer group $\mathrm{Sel}^{\mathrm{BK}}_{A,\Sigma_K} (K)$ 
(resp. $\mathrm{Sel}^{\mathrm{BK}}_{A^\ast,\Sigma_K^c} (K)$) to be the kernel of the localisation map 
 \begin{align*}
  H^1 (K_S / K ,A) \xrightarrow{\;\overline{\mathrm{loc}}^\mathrm{str}_A\;} \overline{\mathrm{Loc}^{\mathrm{str},1}_A (K)} \quad 
 \text{(resp.\ } H^1 (K_S / K ,A^\ast) \xrightarrow{\;\overline{\mathrm{loc}}^\mathrm{str}_{A^\ast}\;} \overline{\mathrm{Loc}^{\mathrm{str},1}_{A^\ast} (K)}).
 \end{align*}
For a finite Galois module $A^\ast [p^m]$,  
we define $\overline{\mathrm{Loc}^{\mathrm{str},1}_{A^\ast [p^m]} (K)}$ as
\begin{align*}
\overline{\mathrm{Loc}^{\mathrm{str},1}_{A^\ast [p^m]} (K)}
 =  \underset{\substack{\lambda \in S_K \\  \lambda \nmid p\infty}}{\prod} 
 \frac{H^1 (K_{\lambda} ,A^\ast [p^m])}{\iota_m ^{-1} (H^1_{\mathrm{ur}} (K_{\lambda} ,A^\ast )_{\mathrm{div}})} \times 
 \underset{\substack{\mathfrak{p}\in \Sigma_{p ,K}}}{\prod}  H^1 (K_{\mathfrak{p}} ,A^\ast [p^m])
\end{align*}
using the %natural %for arXiv version
map $\iota_m \colon H^1 (K_{\lambda} ,A^\ast [p^m]) 
\rightarrow H^1 (K_{\lambda} ,A^\ast )$ induced by the inclusion $A^*[p^m]\hookrightarrow A^*$. 
Then we define  $\mathrm{Sel}^{\mathrm{BK}}_{A^\ast [p^m],\Sigma_F^c} (K)$
to be the kernel of the localisation map 
 \begin{align*} %for arXiv version   align*
  \mathrm{loc}^\mathrm{str}_{A^*[p^m]}\colon H^1 (K_S / K ,A^\ast [p^m]) \rightarrow \overline{\mathrm{Loc}^{\mathrm{str},1}_{A^\ast [p^m]} (K)}.
 \end{align*}
By the Poitou--Tate global duality theorem \cite[(8.6.10)]{Neukirch}, we have the following exact sequence: 
\begin{align} 
 \mathrm{Sel}^{\mathrm{BK}}_{A,\Sigma_F} (F'_{\max})  \longrightarrow 
H^1 (F_S /F'_{\max},A) &\xrightarrow{\;\overline{\mathrm{loc}}^\mathrm{str}_A\;}\nonumber  \\
 \varinjlim_{K}\,  & \overline{\mathrm{Loc}^{\mathrm{str},1}_A (K)}
 \longrightarrow  \left( \varprojlim_{K,m}  \mathrm{Sel}^{\mathrm{BK}}_{A^\ast [p^m],\Sigma_{K}^c} (K) \right)^\vee.\label{equation:surjectiveity_localization}
\end{align}
Here $K$ runs through all finite extensions of $F'$ contained in $F'_{\max}$. Note that there is a natural surjection $q\colon \varinjlim_{K} \overline{\mathrm{Loc}^{\mathrm{str},1}_A (K)} \twoheadrightarrow \mathrm{Loc}^{\mathrm{str},1}_A (F'_{\max})$ satisfying ${\mathrm{loc}^\mathrm{str}_A}=q\circ \overline{\mathrm{loc}}^\mathrm{str}_A$ by definition. Hence, for surjectivity of $\mathrm{loc}_A^\mathrm{str}$, it suffices to verify that $\overline{\mathrm{loc}}_A^\mathrm{str}$ is surjective, or in other words, that the last term of \eqref{equation:surjectiveity_localization} vanishes. Furthermore we have an isomorphism 
\begin{align}\label{equation:compactSel_discreteSel}
\left( \varprojlim_{K,m}  \mathrm{Sel}^{\mathrm{BK}}_{A^\ast [p^m],\Sigma_K^c} (K) \right)^\vee  
\cong \mathrm{Hom}_{\mathcal{O}[[\Gamma_{F',\, \max}]]} \big((\mathrm{Sel}^{\mathrm{BK}}_{A^\ast,\Sigma_F^c } (F'_{\max})^\vee )^\iota , \mathcal{O}[[\Gamma_{F',\, \max}]]\big)
\end{align}
concerning the last term of \eqref{equation:surjectiveity_localization}, where the symbol  $(\ \ )^\iota$ denotes the twisted $\mathcal{O}[[\Gamma_{F',\, \max}]]$-module on which each  $g\in \Gamma_{F',\, \max}$ acts via $g^{-1}$. 
The verification of \eqref{equation:compactSel_discreteSel} goes in the same way as \cite[Lemma 4.11]{ochiai2006} if one replaces the ring $\mathcal{R} /(p^n,H)$  in {\em loc.\ cit.}\ with $\mathcal{O}/(\varpi^n) [\Gamma_{F',\, \max}/U]$, where $\varpi$ is a uniformiser of $\mathcal{O}$ and $U$ is an open subgroup of $\Gamma_{F',\, \max}$. 
From \eqref{equation:compactSel_discreteSel}, we readily observe that the triviality of the last term of \eqref{equation:surjectiveity_localization} is deduced from the $\mathcal{O}[[\Gamma_{F',\, \max}]]$-cotorsionness of $\mathrm{Sel}^{\mathrm{BK}}_{A^*,\Sigma_{F}^c}(F'_{\max})$ because $\mathcal{O}[[\Gamma_{F',\, \max}]]$ has no torsion. Meanwhile $\mathrm{Sel}^{\mathrm{BK}}_{A^*,\Sigma_F^c}(F'_{\max})$ is a submodule of $\mathrm{Sel}_{A^*,\Sigma_F^c}(F'_{\max})$ by construction, and thus the cotorsionness of the former module follows from the cotorsionness of the latter module. Similarly to Proposition~\ref{prop:Sel_cotorsion}, we may deduce the cotorsionness of $\mathrm{Sel}_{A^*,\Sigma_F^c}(F'_{\max})$ from \cite[Theorem~1.2.2 (iii)]{HT-aIMC} noting that $F'_{\max}F'_\tau(\mu_p)/F'_{\max}$ is a finite extension and $G_{F'_{\max}F'_\tau(\mu_p)}$ acts trivially on $A^*$. We thus complete the proof.
\end{proof}

We define the {\em $\mu$-invariant} $\mu (X)$  of a finitely generated torsion $\mathcal{O}[[\Gamma_{F,\, \max}]]$-module $X$ to be the length of the $\mathcal{O}[[\Gamma_{F,\, \max}]]_{(\varpi)}$-module 
$X_{(\varpi)}$, where $\varpi$ is a uniformiser of $\mathcal{O}$ and $(\ \ )_{(\varpi)}$ denotes the localisation at the prime ideal 
of $\mathcal{O}[[\Gamma_{F,\, \max}]]$ generated by $\varpi$. The next lemma plays an important role in the algebraic part of our arguments, since most $G_F$-stable lattices $T$ of an Artin representation $V_\rho$ are not preserved under the virtual decomposition \eqref{eq:Brauer_rho}. 

\begin{lem}[lattice invariance of the $\mu$-invariant]\label{lemma:mu_lattice}
Let $\rho \colon G_{F} \rightarrow \mathrm{Aut}_{\mathcal{E}}\, V_{\rho}$ be a $p$-adic Artin representation of $G_{F}$ and $\mathcal{O}$ the ring of integers of $\mathcal{E}$. For two $G_F$-stable $\mathcal{O}$-lattices $T$ and $T'$ of $V_{\rho}$, consider $A=T\otimes_{\mathbb{Z}_p} \mathbb{Q}_p /\mathbb{Z}_p$ and  $A'=T' \otimes_{\mathbb{Z}_p} \mathbb{Q}_p /\mathbb{Z}_p$. 
Then the $\mu$-invariant of $\mathrm{Sel}_{A,\Sigma_F} (F_{\max}) $ coincides with that of $\mathrm{Sel}_{A',\Sigma_F} (F_{\max}) $. 
\end{lem}

Before the proof of Lemma~\ref{lemma:mu_lattice}, we define
\begin{align*}
 \mathcyr{Sh}^i (K,S,M) = \mathrm{Ker} \left[ H^i (K_S /K ,M) \longrightarrow \bigoplus_{w \in S} H^i (K_w ,M)\right] \quad (i=1,2)
\end{align*} 
for a number field $K$, a finite set $S$ of places of $K$, and a discrete $G_K$-module $M$. By taking the inductive limit, we extend the definition of $\mathcyr{Sh}^i (K,S,M)$ to an algebraic extention $K$ of $\mathbb{Q}$ of infinite degree, as before.

\begin{proof}
By Lemma \ref{lemma:Sel=selstr}, it suffices to prove the anologous assertion for 
the $\mu$-invariants of $\mathrm{Sel}^{\mathrm{str}}_{A,\Sigma_F} (F_{\max}) $ and 
$\mathrm{Sel}^{\mathrm{str}}_{A',\Sigma_F} (F_{\max}) $. 
Replacing $T$ with an appropriate homothetic lattice, we may assume
 without loss of generality that there is a $G_F$-equivariant surjection $A
 \twoheadrightarrow A'$. Then we obtain the following fundamental commutative diagram with exact rows:
\begin{equation}\label{equation:fundamentaldiagram_mu}
\begin{CD}
0 @>>> \mathrm{Sel}^{\mathrm{str}}_{A,\Sigma_F} (F_{\max}) @>>> H^1 (F_S / F_{\max} ,A)  @>{\mathrm{loc}^\mathrm{str}_A}>> \mathrm{Loc}^{\mathrm{str},1}_A (F_{\max})  \\ 
@. @V{f_A\vert_{\mathrm{Sel}}}VV @VV{f_A}V @VV{g_A}V \\ 
0 @>>> \mathrm{Sel}^{\mathrm{str}}_{A',\Sigma_F} (F_{\max}) @>>> H^1 (F_S / F_{\max} ,A')  @>>{\mathrm{loc}^\mathrm{str}_{A'}}> \mathrm{Loc}^{\mathrm{str},1}_{A'} (F_{\max}). 
\\ 
\end{CD}
\end{equation}
Here the vertical maps $f_A$ and $g_A$ are induced from the surjection
 $A\twoheadrightarrow A'$. Applying Lemma \ref{lemma:surjectivity_localization_map} to the case where $F'=F$ and $\tau =\rho$, we observe that both the rightmost horizontal arrows $\mathrm{loc}_{A}^{\mathrm{str}}$ and $\mathrm{loc}_{A'}^{\mathrm{str}}$ in \eqref{equation:f_A} are surjective. Note that, since both $\mathrm{Sel}_{A,\Sigma_F}^{\mathrm{str}}(F_{\max})$ and $\mathrm{Sel}_{A',\Sigma_F}^{\mathrm{str}}(F_{\max})$ are $\mathcal{O}[[\Gamma_{F,\, \max}]]$-cotorsion by Proposition~\ref{prop:Sel_cotorsion}, the kernel and the cokernel of $f_A\vert_{\mathrm{Sel}}$ are also $\mathcal{O}[[\Gamma_{F,\, \max}]]$-cotorsion. Furthermore, since both $f_A\otimes \mathcal{E}$ and $g_A\otimes \mathcal{E}$ induces isomorphisms, all of $\mathrm{Ker}\, f_A$, $\mathrm{Coker}\, f_A$, $\mathrm{Ker}\, g_A$ and $\mathrm{Coker}\, g_A$ are $p$-cotorsion. Hence we obtain an equality among $\mu$-invariants 
\begin{align}
\mu (\mathrm{Sel}^{\mathrm{str}}_{A,\Sigma_F}& (F_{\max})^\vee)-\mu (\mathrm{Sel}^{\mathrm{str}}_{A',\Sigma_F} (F_{\max})^\vee)  = \mu((\mathrm{Ker}\, f_A\vert_{\mathrm{Sel}})^\vee) -\mu ((\mathrm{Coker}\, f_A\vert_{\mathrm{Sel}})^\vee) \label{equation:snake_mu-invariants} \\
&= \mu ((\mathrm{Ker}\, f_A)^\vee) - \mu ((\mathrm{Coker}\, f_A)^\vee)  - \mu ((\mathrm{Ker}\, g_A)^\vee) + \mu ((\mathrm{Coker}\, g_A)^\vee) \nonumber
\end{align}
by applying the snake lemma to the fundamental diagram \eqref{equation:fundamentaldiagram_mu}, taking the Pontrjagin dual and localisation at $(\varpi)$, and taking the $\mu$-invariants. Now let us study the differences %for arXiv version: ann the differences
 $\mu((\mathrm{Coker}\, f_A)^\vee)-\mu((\mathrm{Ker}\, f_A)^\vee)$ and $\mu((\mathrm{Coker}\, g_A)^\vee)-\mu((\mathrm{Ker}\, g_A)^\vee)$. 

\begin{component}[$\blacktriangleright$ The kernel and cokernel of $f_A$]
 The vertical map $f_A$ in \eqref{equation:fundamentaldiagram_mu} fits into the long exact sequence
 \begin{align}  
  \begin{tikzcd}[ampersand replacement=\&]
  \& H^0(F_S / F_{\max} ,A) \ar{r}{f_A^0} \ar[draw=none]{d}[name=X,anchor=center]{} \& H^0 (F_S /
  F_{\max} ,A') \ar[rounded corners,
  to path={
  -- ([xshift=2ex]\tikztostart.east)
  |- (X.center) \tikztonodes
  -| ([xshift=-2ex]\tikztotarget.west)
  -- (\tikztotarget)}]{dll}{} \\
 H^1 (F_S / F_{\max} ,C) \rar \& H^1 (F_S /
  F_{\max} ,A) \ar{r}{f_A} \ar[draw=none]{d}[name=Y,anchor=center]{} \& H^1 (F_S / F_{\max} ,A') \ar[rounded corners,
  to path={
  -- ([xshift=2ex]\tikztostart.east)
  |- (Y.center) \tikztonodes
  -| ([xshift=-2ex]\tikztotarget.west)
  -- (\tikztotarget)}]{dll}{}\\
H^2 (F_S / F_{\max} ,C) \rar \& H^2 (F_S /
  F_{\max} ,A) \&
 \end{tikzcd} \label{equation:f_A}
  \end{align}
where $C$ denotes the kernel of $A \twoheadrightarrow A'$; note that it is a finite $G_F$-module. 
Since both $ H^0 (F_S / F_{\max} ,A)$ and  $ H^0 (F_S / F_{\max} ,A')$ are cofinitely generated $\mathcal{O}$-modules whose corank are equal to $\mathrm{dim}_{\mathcal{E}}H^0 (F_S / F_{\max} ,V_{\rho})$, %for arXiv version
the cokernel of the top horizontal map $f_A^0$ in \eqref{equation:f_A} is $\mathcal{O}$-cofinitely generated and $\mathcal{O}$-cotorsion. In particular it is finite $\mathcal{O}[[\Gamma_{F,\, \max}]]$-module, and thus its localisation at $(\varpi)$ is trivial. Combining this observation with \eqref{equation:f_A}, we have 
\begin{equation}\label{equation:ker_fa}
(\mathrm{Ker}\, f_A)^\vee_{(\varpi)}=H^1(F_S/F_{\max},C)_{(\varpi)}^\vee.
\end{equation}
Next we shall prove that $H^2(F_S/F_{\max},A)$ is a trivial module to deduce  
\begin{equation}\label{equation:coker_fa}
(\mathrm{Coker}\, f_A)^\vee=H^2(F_S/F_{\max},C)^\vee.
 \end{equation}
To achieve this purpose, we will observe that $H^2(F_S/F_{\max}, A)$ is both cotorsion free and cotorsion over $\mathcal{O}[[\Gamma_{F,\,\max}]]$. For cotorsion freeness, recall that Shapiro's lemma \cite[Proposition (1.6.4)]{NSW} implies an isomorphism
\begin{equation}\label{equation:shapiro_h^2_A}
H^2 (F_S / F_{\max} ,A) \cong H^2 (F_S / F , \mathcal{T}^\vee),
 \end{equation}
where $\mathcal{T}^\vee$ is the Pontrjagin dual of a free $\mathcal{O}[[\Gamma_{F,\, \max}]]$-module $\mathcal{T}:=T\otimes_{\mathcal{O}}\mathcal{O}[[\Gamma_{F,\, \max}]]^\sharp$.  Using the fact that the $p$-cohomological dimension of $\mathrm{Gal}(F_S/F)$ is less than or equal to $2$ \cite[Proposition (8.3.18)]{NSW}, we readily observe that the multiplication-by-$x$ map on the divisible module $\mathcal{T}^\vee$ induces an surjective endomorphism of $H^2(F_S/F,\mathcal{T}^\vee)$ for any non-zero  $x\in \mathcal{O}[[\Gamma_{F,\, \max}]]$. Taking the Pontrjagin dual and using \eqref{equation:shapiro_h^2_A}, we conclude that $H^2(F_S/F_{\max}, A)^\vee$ is $\mathcal{O}[[\Gamma_{F,\,\max}]]$-torsion free, as desired.

For cotorsionness of $H^2(F_S/F_{\max},A)$, note that the cohomology group %for arXiv version
$H^2(F_S/F_{\max},A)$ is isomorphic to $\mathcyr{Sh}^2 (F_{\max} ,S_{\max},A)$ for a finite set $S_{\max}$ of all the places of $F_{\max}$ lying above those contained in $S_F$. Indeed, we have an exact sequence 
\begin{equation}\label{equation:h^2_torsion}
0 \longrightarrow \mathcyr{Sh}^2 (F_{\max} ,S_{\max},A) \longrightarrow H^2 (F_S / F_{\max} ,A) \longrightarrow  \prod_{w \in S_{\max}} \!\!
H^2 ( F_{\max,w} ,A)
\end{equation}
by the very definition of $\mathcyr{Sh}^2 (F_{\max} ,S_{\max},A)$, but the last term of \eqref{equation:h^2_torsion} vanishes since the $p$-cohomological dimension of the absolute Galois group of  $F_{\max,w}$ equals %is %for arXiv version
 $1$ for every $w \in S_{\max}$. We may deduce this from \cite[Chapitre II, Proposition~12]{serre97} noting that the absolute Galois group of the residue field of $F_{\max,w}$ is isomorphic to a pro-prime-to-$p$ group $\prod_{\ell\neq p}\mathbb{Z}_\ell$ for finite every $w\in S_{\max}$. 
Therefore it suffices to verify that $\mathcyr{Sh}^2 (F_{\max} ,S_{\max},A)$ is cotorsion, which follows from cotorsionness of $\mathcyr{Sh}^1 (F_{\max} ,S_{\max},A^*)$ for the Kummer dual $A^*:=\mathrm{Hom}_{\mathrm{cts}}(T,\mu_{p^\infty})$ due to \cite[Proposition 4.4]{gr-coh}. 
We may indeed deduce $\mathcal{O}[[\Gamma_{F,\,\max}]]$-cotorsionness of $\mathcyr{Sh}^1 (F_{\max} ,S_{\max},A^*)$ from the fact that $\mathrm{Sel}_{A^*,\Sigma_F}(F_{\max})$ is cotorsion (Proposition~\ref{prop:Sel_cotorsion}), based upon an argument similar to the proof of Lemma~\ref{lemma:surjectivity_localization_map}.

\end{component}
 
\begin{component}[$\blacktriangleright$ The kernel and cokernel of $g_A$]

By a similar argument using the long exact sequence corresponding to \eqref{equation:f_A}, we readily obtain isomorphisms
\begin{align} \label{eq:ker_coker_ga} 
(\mathrm{Ker}\, g_A)_{(\varpi)}^\vee &\cong \mathrm{Loc}_C^{\mathrm{str},1}(F_{\max})^\vee_{(\varpi)}, & (\mathrm{Coker}\, g_A)^\vee &\cong \mathrm{Loc}_C^{\mathrm{str},2}(F_{\max})^\vee.
\end{align}
Indeed, we may deduce the triviality of $\mathrm{Loc}_A^{\mathrm{str},2}(F_{\max})$ directly from the fact that the $p$-cohomological dimension of $G_{F_{\max, w}}$ is equal to $1$ for every finite $w\in S_{\max}$, as we have seen.

\end{component}

\begin{component}[$\blacktriangleright$ Analysis of the difference of $\mu$-invariants]

Combining \eqref{equation:snake_mu-invariants} with \eqref{equation:ker_fa}, \eqref{equation:coker_fa} and \eqref{eq:ker_coker_ga}, we obtain
\begin{align}
\mu (\mathrm{Sel}^{\mathrm{str}}_A (F_{\max}))-\mu (\mathrm{Sel}^{\mathrm{str}}_{A'} (F_{\max})) &= 
\mu (H^1 (F_S / F_{\max} ,C)) - \mu (H^2 (F_S / F_{\max} ,C))  \nonumber\\
&\qquad \qquad  - \bigl(\mu (\mathrm{Loc}^{\mathrm{str},1}_C (F_{\max})) - \mu (\mathrm{Loc}^{\mathrm{str},2}_C (F_{\max}))\bigr). \label{equation:comparison_of_mu-invariants}
\end{align}
In order to verify that the right-hand side of \eqref{equation:comparison_of_mu-invariants} is equal to $0$, it suffices to check that the ratio of $\frac{\# H^1 (F_S / K ,C)}{ \# H^2 (F_S / K ,C)}$ to $\frac{\# \mathrm{Loc}^{\mathrm{str},1}_{C} (K)}{ \# \mathrm{Loc}^{\mathrm{str},2}_{C} (K)}$ is bounded when $K$ runs through  
the set of finite  extensions of $F$ contained in $F_{\max}$. Take such a finite extension $K/F$. Then the global Euler--Poincar\'{e} characteristic formula of the Galois cohomology \cite[(8.7.4)]{NSW} implies   
\begin{equation}\label{equation:local_behaviour1}
\frac{\# H^1 (F_S / K ,C)}{\# H^0 (F_S / K ,C) \# H^2 (F_S / K ,C)} = 
\underset{\substack{v:\text{finite place of $K$} \\  v \vert \infty}}{\prod}  \# H^0 (K_v , C) = \# C ^{\frac{[K: \mathbb{Q}]}{2}}. 
\end{equation}
Meanwhile the local Euler--Poincar\'{e} characteristic formula of the Galois cohomology \cite[Theorem (7.3.1)]{NSW} implies
\begin{align}
\frac{\# \mathrm{Loc}^{\mathrm{str},1}_{C} (K)}{\# \mathrm{Loc}^{\mathrm{str},0}_{C} (K) \# \mathrm{Loc}^{\mathrm{str},2}_{C} (K)} &= 
\underset{\substack{v\in \Sigma^c_{K,p}}}{\prod}   \frac{\#  H^1 (K_{v},C)}
{\#  H^0 (K_v ,C) \#  H^2 (K_v,C)}  \nonumber \\
&= \underset{\substack{v\in \Sigma^c_{K,p}}}{\prod} \# C^{[K_{v}:\mathbb{Q}_p]} 
= \# C ^{\frac{[K: \mathbb{Q}]}{2}}. \label{equation:local_behaviour2a} 
\end{align}

Compairing \eqref{equation:local_behaviour1} with \eqref{equation:local_behaviour2a}, we have
\begin{align*}
\frac{1}{\# C}\leq \frac{\# H^1 (F_S / K ,C)}{ \# H^2 (F_S / K ,C)} \left(\frac{\# \mathrm{Loc}^{\mathrm{str},1}_{C} (K)}{ \# \mathrm{Loc}^{\mathrm{str},2}_{C} (K)}\right)^{-1} \!\!  =\dfrac{\# \mathrm{Loc}^{\mathrm{str},0}_C(K)}{\# H^0(F_S/K,C)} \leq \# C^{\# S_{\max}},
\end{align*}
which implies that the ratio under consideration is bounded with respect to $K$, as desired. \qedhere

\end{component}
\end{proof}

  \subsection{Integrality of the $p$-adic Artin $L$-function} \label{ssc:integrality}

Recall that we have defined $L_{p,\Sigma_F}(M(\rho ))$ by using a Brauer decomposition \eqref{equation:brauer_induction} as in \eqref{equation:katz3}. 

\begin{pro} \label{prop:decomp_char}
The characteristic ideal of the Pontrjagin dual of the Selmer group $\mathrm{Sel}_{A,\Sigma_F}(F_{\max})$ admits a decomposition
\begin{align}\label{equation:Selmer_by_Brauer_induction}
\mathrm{char}_{\mathcal{O}[[\Gamma_{F,\,\max}]]}\mathrm{Sel}_{A,\Sigma_F}(F_{\max})^\vee  
=
\prod^s_{j=1}  \left( 
\mathrm{char}_{\mathcal{O}[[\mathrm{Gal} (F_j F_{\max} /F_j)]]}\mathrm{Sel}_{A_{\psi_j},\Sigma_{F_j}}(F_j F_{\max})^\vee 
\right)^{a_j} 
\end{align}
corresponding to \eqref{equation:katz3}. Here the product of the characteristic ideals appearing in the right-hand side of \eqref{equation:Selmer_by_Brauer_induction} is taken in the field of fractions of $\mathcal{O}[[\Gamma_{F,\,\max}]]$ via natural inclusions $\mathcal{O}[[\mathrm{Gal}(F_jF_{\max}/F_j)]]\hookrightarrow \mathcal{O}[[\Gamma_{F,\,\max}]]$. 
\end{pro}

\begin{proof}
 Taking Shapiro's lemma \cite[Proposition (1.6.4)]{NSW} into accounts, we may rewrite \eqref{equation:Selmer_by_Brauer_induction}  as
\begin{align}\label{equation:Selmer_by_Brauer_induction2}
\mathrm{char}_{\mathcal{O}[[\Gamma_{F,\,\max}]]}\mathrm{Sel}_{A,\Sigma_F}(F_{\max})^\vee 
=  
\prod^s_{j=1} \left( 
\mathrm{char}_{\mathcal{O}[[\Gamma_{F,\,\max}]]}\mathrm{Sel}_{\mathrm{Ind}^{G_F} _{G_{F_j}}A_{\psi_j},\Sigma_F}( F_{\max})^\vee 
\right)^{a_j}. 
\end{align}
Without loss of generality, we may rearrange the subindices $\{j \mid 1\leq j\leq s\}$ so that $a_1 ,\ldots , a_{s'}$ ($s'\leq s$) are positive and $a_{s'+1} ,\ldots , a_{s}$ are negative. Then the equality \eqref{equation:Selmer_by_Brauer_induction2} is rewritten as 
\begin{multline}\label{equation:Selmer_by_Brauer_induction3}
\mathrm{char}_{\mathcal{O}[[\Gamma_{F,\,\max}]]}\mathrm{Sel}_{A \oplus \bigoplus^{s}_{j=s'+1} \mathrm{Ind}^{G_F} _{G_{F_j}}A_{\psi_j}^{\oplus (-a_j)},\Sigma_F}(F_{\max})^\vee
\\ =  
\mathrm{char}_{\mathcal{O}[[\Gamma_{F,\,\max}]]}\mathrm{Sel}_{ \bigoplus^{s'}_{j=1} \mathrm{Ind}^{G_F}_{G_{F_j}}A_{\psi_j}^{\oplus a_j},\Sigma_F}( F_{\max})^\vee. 
\end{multline}
Here recall that the discrete representation $A$ appearing in the left-hand side of \eqref{equation:Selmer_by_Brauer_induction3} is defined for a chosen $G_F$-stable $\mathcal{O}$-lattice $T$ of $V_\rho$, whereas each $\mathrm{Ind}^{G_F}_{G_{F_j}}A_{\psi_j}$ appearing in the both sides of \eqref{equation:Selmer_by_Brauer_induction3} is defined for the $G_F$-stable $\mathcal{O}$-lattice $\mathrm{Ind}^{G_F}_{G_{F_j}}T_{\psi_j}$ of $\mathrm{Ind}^{G_F}_{G_{F_j}}V_{\psi_j}$ induced by a standard lattice  $T_{\psi_j}=\mathcal{O}e_{\psi_j}$ of the $1$-dimensional representation $V_{\psi_j}$. 
Then we firstly observe that \eqref{equation:Selmer_by_Brauer_induction3} holds true up to $\mu$-invariants since the both sides of \eqref{equation:Selmer_by_Brauer_induction3} are characteristic ideals corresponding to the Selmer groups of the {\em same} representation $V_{\rho'}:=\bigoplus_{j=1}^{s'}\mathrm{Ind}^{G_F}_{G_{F_j}}V_{\psi_j}^{\oplus a_j}$ with respect to the {\em different} $G_F$-stable $\mathcal{O}$-lattices 
$
T \oplus \bigoplus^{s}_{j=s'+1} \mathrm{Ind}^{G_F}_{G_{F_j}}T_{\psi_j}^{\oplus (-a_j)}
$
and $
\bigoplus^{s'}_{j=1} \mathrm{Ind}^{G_F}_{G_{F_j}}T_{\psi_j}^{\oplus a_j}
$. 
But Lemma~\ref{lemma:mu_lattice} guarantees that the equality \eqref{equation:Selmer_by_Brauer_induction3} holds true even without the ambiguity of $\mu$-ivvariants. 
We thus obtain \eqref{equation:Selmer_by_Brauer_induction}, as desired.
\end{proof}

\begin{proof}[Proof of Theorem $\ref{theorem:integrality_of_CMp-adicL_Artin}$]
Recall that we assume the validity of Conjecture \ref{conjecture:iwasawamainconjecture} 
\begin{align} \label{eq:IMC_each_j}
 (L_{p,\Sigma_{F_j}}(\psi_j)) = \mathrm{char}_{\widehat{\mathcal{O}}^{\mathrm{ur}} [[\Gamma_{F_j,\, \max}]]} \left(
\mathrm{Sel}_{A_{\psi_j,\Sigma_{F_j}}}(F_{j ,\, \max})^\vee \widehat{\otimes}_{\mathcal{O}} \widehat{\mathcal{O}}^{\mathrm{ur}}\right) 
\end{align}
 for each $j$.
Combining \eqref{eq:IMC_each_j} with the defining equality \eqref{equation:katz3} of $L_{p,\Sigma_F}(M(\rho))$, we have 
\begin{align}\label{equation:equality using L_p(M(rho))}
(L_{p,\Sigma_F}(M(\rho ))) = \prod^s_{j=1} \mathrm{pr}_j 
\left(
\mathrm{char}_{\widehat{\mathcal{O}}^{\mathrm{ur}} [[\Gamma_{F_j,\, \max}]]} 
\left(\mathrm{Sel}_{A_{\psi_j},\Sigma_{F_j}}(F_{j ,\,\max})^\vee \widehat{\otimes}_{\mathcal{O}} \widehat{\mathcal{O}}^{\mathrm{ur}} \right)
\right)^{a_j}.
\end{align}   
Thus, to prove the assertion, it suffices to verify the {\em descent equality}
\begin{multline}\label{equation:selmer_descent}
\mathrm{char}_{\mathcal{O}[[\mathrm{Gal} (F_j F_{\max} /F_j)]]}\mathrm{Sel}_{A_{\psi_j},\Sigma_{F_j}}(F_j F_{\max})^\vee  \\
= 
 \mathrm{char}_{\mathcal{O}[[\mathrm{Gal} (F_{j,\,\max} /F_j)]]} \mathrm{Sel}_{A_{\psi_j},\Sigma_{F_j}}(F_{j ,\,\max})^\vee  \mod\, \mathfrak{A}_{r_j}
\end{multline}
of the characteristic ideal for each $j$, where $\mathfrak{A}_{r_j}$ denotes the kernel of $\mathrm{pr}_j$. 
Indeed we can deduce the main equality \eqref{equation:IMC_for_Artinrep} of Conjecture \ref{conjecture:iwasawamainconjectureartin} combining \eqref{equation:selmer_descent} with Proposition~\ref{prop:decomp_char} and \eqref{equation:equality using L_p(M(rho))}. Furthermore \eqref{equation:selmer_descent} is equivalent to
\begin{multline}\label{equation:selmer_str_descent}
\mathrm{char}_{\mathcal{O}[[\mathrm{Gal} (F_j F_{\max} /F_j)]]}\mathrm{Sel}^{\mathrm{str}}_{A_{\psi_j},\Sigma_{F_j}}(F_j F_{\max})^\vee \\
= 
\mathrm{char}_{\mathcal{O}[[\mathrm{Gal} (F_{j,\,\max} /F_j)]]} 
\mathrm{Sel}^{\mathrm{str}}_{A_{\psi_j},\Sigma_{F_j}}(F_{j ,\,\max})^\vee \mod\, \mathfrak{A}_{r_j}
\end{multline}
by Lemma \ref{lemma:Sel=selstr}, and thus we shall verify \eqref{equation:selmer_str_descent} in the rest of the proof.

In order to study specialisation of strict Selmer groups effectively, we %hereafter %for arXiv version
use the presentation $\mathrm{Sel}^{\mathrm{str}}_{\mathcal{A}_j,\Sigma_{F}}(F_j)$ of the strict Selmer group introduced in \eqref{eq:str_Sel_def} instead of $\mathrm{Sel}^{\mathrm{str}}_{A_{\psi_j},\Sigma_{F_j}}(F_{j,\,\max})$; recall that $\mathcal{A}_j$ is defined as $\mathcal{T}_j\otimes_{\mathcal{O}[[\Gamma_{F_j,\,\max}]]}\mathcal{O}[[\Gamma_{F_j,\,\max}]]^\vee$ for $\mathcal{T}_j:=T_{\psi_j}\otimes_{\mathcal{O}}\mathcal{O}[[\Gamma_{F_j,\,\max}]]^\sharp$.
By construction, the Galois group $\mathrm{Gal}(F_{j ,\,\max} /F_j F_{\max})$ is a free $\mathbb{Z}_p$-module of rank $r_j$, where $r_j$ is defined as $d( [F_j:F] -1) + \delta_{F_j} -\delta_F$; %for arXiv version
here $\delta_F $ and $\delta_{F_j}$ respectively denote the Leopoldt $p$-defects of $F$ and $F_j$. Note that $\delta_{F_j}\geq \delta_F$ holds because $F_j$ is an extension of $F$. 
Let us take a $\mathbb{Z}_p$-basis $\gamma_1 , \dotsc , \gamma_{r_j}$ of $\mathrm{Gal}(F_{j ,\,\max} /F_j F_{\max})$, and set $x_\nu=\gamma_\nu-1$ for $\nu=1,\dotsc,r_j$. Then $x_1,x_2,\dotsc,x_{r_j}$ forms a regular sequence of $\mathcal{O}[[\Gamma_{F_j,\,\max}]]$ generating the kernel of $\mathrm{pr}_j$, or equivalently, the kernel of the quotient map $\mathcal{O}[[\Gamma_{F_j,\,\max} ]]\twoheadrightarrow \mathcal{O}[[\mathrm{Gal} (F_j F_{\max} /F_j )]]$. For each $\nu=1,\dotsc,r_j$, let $\mathfrak{A}_\nu$ denote the ideal of $\mathcal{O}[[\Gamma_{F_j,\, \max}]]$ generated by $x_1 ,\dotsc, x_\nu$, and write $\mathcal{A}_j[\mathfrak{A}_\nu]$ for the maximal $\mathfrak{A}_\nu$-torsion submodule of $\mathcal{A}_j$.
We can define the strict Selmer group $\mathrm{Sel}^{\mathrm{str}}_{\mathcal{A}_j[\mathfrak{A}_\nu],\Sigma_{F_j}}(F_j)$ for the discrete representation $\mathcal{A}_j[\mathfrak{A}_\nu]$ of $G_{F_j}$. 
In order to verify \eqref{equation:selmer_str_descent}, it suffices to show that there exists a $\mathbb{Z}_p$-basis $\gamma_1,\ldots,\gamma_{r_j}$ of $\mathrm{Gal}(F_{j,\, \max}/F_jF_{\max})$ satisfying the following two conditions for each $\nu=1,2,\ldots,r_j$ (put $\mathfrak{A}_0=(0)$ as convention);
\begin{enumerate}[label={[\Roman*]}]
\item (control theorem) 
both %for arXiv version
the kernel and the cokernel of the map 
\begin{align*}
 s_\nu\colon \mathrm{Sel}^{\mathrm{str}}_{\mathcal{A}_j [\mathfrak{A}_\nu ],\Sigma_{F_j}} (F_j ) \longrightarrow \mathrm{Sel}^{\mathrm{str}}_{\mathcal{A}_j [\mathfrak{A}_{\nu-1} ],\Sigma_{F_j}} (F_j )
[x_\nu]
\end{align*}
induced by the inclusion $\mathcal{A}_j[\mathfrak{A}_{\nu-1}]\hookrightarrow \mathcal{A}_j[\mathfrak{A}_\nu]$ are %both 
copseudonull (as $\mathcal{O}[[ \Gamma_{F_j,\, \max}]]/\mathfrak{A}_\nu$-modules);
\item 
the Pontrjagin dual 
$\mathrm{Sel}^{\mathrm{str}}_{\mathcal{A}_j [\mathfrak{A}_\nu ],\Sigma_{F_j}} (F_j )^\vee$ of the strict Selmer group has no nontrivial pseudonull $\mathcal{O}[[\Gamma_{F_j,\max} ]]/\mathfrak{A}_\nu$-submodules. 
\end{enumerate}
Indeed, after verifying existence of such tuples $\gamma_1,\ldots,\gamma_{r_j}$, we may deduce 
\begin{multline*}
\mathrm{char}_{\mathcal{O}[[\Gamma_{F_j,\, \max}]]/\mathfrak{A}_{\nu-1}} 
\mathrm{Sel}^{\mathrm{str}}_{\mathcal{A}_j[\mathfrak{A}_{\nu-1}],\Sigma_{F_j}}(F_j)^\vee  \mod\, (x_\nu, x_{\nu+1},\ldots,x_{r_j}) \\
\stackrel{\text{[II]}}{=}\mathrm{char}_{\mathcal{O}[[\Gamma_{F_j,\, \max}]]/\mathfrak{A}_{\nu}} \mathrm{Sel}^{\mathrm{str}}_{\mathcal{A}_j[\mathfrak{A}_{\nu-1}],\Sigma_{F_j}}(F_j)[x_\nu]^\vee \mod\, (x_{\nu+1},\ldots, x_{r_j}) \\
\stackrel{\text{[I]}}{=}\mathrm{char}_{\mathcal{O}[[\Gamma_{F_j,\, \max}]]/\mathfrak{A}_{\nu}} \mathrm{Sel}^{\mathrm{str}}_{\mathcal{A}_j[\mathfrak{A}_\nu]}(F_j)^\vee \mod\, (x_{\nu+1},\ldots, x_{r_j})
\end{multline*}
for each $\nu=1,2,\ldots, r_j$  by inductively applying \cite[Lemma~3.1]{ochiai2005}; see also the explanation given in \cite[Section 3.4.5]{haraochiai}. The desired equality \eqref{equation:selmer_str_descent} is just a combination of all these equalities; recall that $\mathrm{Sel}^{\mathrm{str}}_{\mathcal{A}_j[\mathfrak{A}_{r_j}],\Sigma_{F_j}}(F_j)$ is isomorphic to $\mathrm{Sel}^{\mathrm{str}}_{A_{\psi_j}, \Sigma_{F_j} }(F_j F_{\max})$. 

\begin{component}[$\blacktriangleright$ {\rm [I]} for every regular sequence]
 Here we shall show that the assertion [I] holds true for  {\em every} regular sequence $x_1,\dotsc,x_{r_j}\in \mathcal{O}[[\Gamma_{F_j,\max} ]]$ of the form $x_\nu =\gamma_\nu -1$ ($\nu=1,2,\ldots,r_j$) generating the kernel $\mathfrak{A}_{r_j}$ of $\mathrm{pr}_j$. Let us consider the commutative diagram
\begin{align*}
\xymatrix{H^1(F_{j,\, S_{F_j}},\mathcal{A}_j[\mathfrak{A}_\nu]) \ar[rrr]^-{\mathrm{loc}_{\mathcal{A}_j[\mathfrak{A}_\nu]}^{\mathrm{str}}} \ar@{->>}[d]_-{t_\nu} &&& \mathrm{Loc}_{\mathcal{A}_j[\mathfrak{A}_\nu],\Sigma_{F_j}}^{\mathrm{str},1}(F_j) \ar@{->>}[d]^-{u_\nu} \\ 
H^1 (F_{j, S_{F_j}}/F_j , \mathcal{A}_j[\mathfrak{A}_{\nu-1}] )[x_\nu] \ar[rrr]^-{\mathrm{loc}_{\mathcal{A}_j[\mathfrak{A}_{\nu-1}]}^{\mathrm{str}}} &&& \mathrm{Loc}_{\mathcal{A}_j[\mathfrak{A}_{\nu-1}],\Sigma_{F_j}}^{\mathrm{str},1}(F_j)[x_\nu]} 
\end{align*}
where $\mathrm{Loc}^{\mathrm{str},1}_{\mathcal{A}_j [\mathfrak{A}_\nu ]} (F_j )$ denotes $\prod_{w\in S_{F_j}\setminus \Sigma_{F_j,p},\; w\nmid \infty} H^1(F_{j,\, w}, \mathcal{A}_j[\mathfrak{A}_\nu])$ as in Section~\ref{ssc:alg_pre}. Recall from \eqref{eq:str_Sel} that $\mathrm{Sel}_{\mathcal{A}_j[\mathfrak{A}_\nu],\Sigma_{F_j}}^{\mathrm{str}}(F_j)$ is defined as the kernel of $\mathrm{loc}_{\mathcal{A}_j[\mathfrak{A}_\nu]}^{\mathrm{str}}$. Both $t_\nu$ and $u_\nu$ are induced by the the short exact sequence 
\begin{align} \label{eq:xnu_ses}
 \xymatrix{0\ar[r] & \mathcal{A}_j[\mathfrak{A}_\nu] \ar[r] & \mathcal{A}_j[\mathfrak{A}_{\nu-1}]\ar[r]^-{\times x_\nu} & \mathcal{A}_j[\mathfrak{A}_{\nu-1}] \ar[r] & 0},
\end{align} 
and $s_\nu$ is just a morphism induced by $t_\nu$. In order to verify the control theorem [I], it suffices to show that both $\mathrm{Ker}\, t_\nu$ and $\mathrm{Ker}\, u_\nu$ are copseudonull $\mathcal{O}[[\Gamma_{F_j,\max}]]/\mathfrak{A}_{\nu}$-modules because the snake lemma implies that $\mathrm{Ker}\, s_\nu$ is a submodule of $\mathrm{Ker}\, t_\nu$ and $\mathrm{Coker}\, s_\nu$ is a subquotient of $\mathrm{Ker}\, u_\nu$. 

Firstly, the Pontrjagin dual $(\mathrm{Ker}\, t_\nu)^\vee$ of the kernel of $t_\nu$ is described as 
\begin{align*}
  \left(H^0( F_{j,S_{F_j}}/F_j,\, \mathcal{A}_j[\mathfrak{A}_{\nu-1}])/x_{\nu}\right)^\vee \cong H_0(F_{j,S_{F_j}}/F_j,\, \mathcal{T}_j/\mathfrak{A}_{\nu-1}\mathcal{T}_j)[x_\nu]
\end{align*}
by the long exact sequence associated with \eqref{eq:xnu_ses} and the Pontrjagin duality. Now take a sufficiently small open subgroup $U$ of $\mathrm{Gal} (F_{j, S_{F_j}} /F_j) $ so that it acts trivially on 
$T_{\psi_j}$, and let $F_U$ be an intermediate field of $F_{j,\max}/F_j$ corresponding to $U/\bigl(\mathrm{Gal}(F_{j,S_{F_j}}/F_{j,\, \max})\cap U\bigr)$. Then $H_0(F_{j,S_{F_j}}/F_j,\mathcal{T}_j/\mathfrak{A}_{\nu-1}\mathcal{T}_j)[x_\nu]$ is, by definition, realised as a subquotient of a finitely generated 
$\mathcal{O}$-module $T_{\psi_j} \otimes_{\mathcal{O}} \mathcal{O}[\mathrm{Gal} (F_U /F_j) ]^\sharp $. This observation implies that $(\mathrm{Ker}\, t_\nu)^\vee$ is $\mathcal{O}[[\Gamma_{F_j,\, \max}]]/\mathfrak{A}_{\nu}$-pseudonull since the Krull dimension of $\mathcal{O}[[\Gamma_{F_j,\,\max} ]]/\mathfrak{A}_{\nu}$ 
is greater than or equal to $2$. 
 
Next, the Pontrjagin dual $(\mathrm{Ker}\, u_\nu)^\vee$ of the kernel of $u_\nu$ is described as 
\begin{align}\label{equation:cokernel:controltheorem}
\prod_{\substack{w\in S_{F_j}\setminus \Sigma_{F_j,p} \\ w\nmid \infty}} \left( H^0(F_{j,w}, \mathcal{A}_j[\mathfrak{A}_{\nu-1}])/x_\nu\right)^\vee \cong \prod_{\substack{w\in S_{F_j}\setminus \Sigma_{F_j,p} \\ w\nmid \infty}} H_0(F_{j,w}, \mathcal{T}_j/\mathfrak{A}_{\nu-1}\mathcal{T}_j)[x_\nu]
\end{align}
in the same way as $(\mathrm{Ker}\, t_\nu)^\vee$. We claim that the $w$-component of \eqref{equation:cokernel:controltheorem} is trivial if $w$ is not contained in $\Sigma_{F_j,p}^c$. Note that $\mathrm{Frob}_w$ forms part of a $\mathbb{Z}_p$-basis of $\mathrm{Gal}(F_jF_{\max}/F_j)$ since $F_jF_{\max}/F_j$ is unramified at $w$. In particular $x_1,x_2,\ldots,x_{\nu}, x_w^F$ forms a regular sequence of $\mathcal{O}[[\Gamma_{F_j.\,\max}]]$ for every $\nu=1,2,\ldots, r_j$ if we set $x_w^F:=\mathrm{Frob}_w-1$. Now let $x_w^I\mathcal{O}$ denote the ideal of $\mathcal{O}$ generated by $\psi_j^{\mathrm{gal}}(g)-1$ for all $g\in I_w$. If $x_w^I$ is a unit, the $I_w$-coinvariant of $\mathcal{T}_j/\mathfrak{A}_{\nu-1}\mathcal{T}_j$ vanishes and so does $H_0(F_{j,w},\mathcal{T}_j/\mathfrak{A}_{\nu-1}\mathcal{T}_j)[x_\nu]$. Otherwise we readily observe that $x_1,x_2,\ldots, x_\nu, x_w^F, x_w^I$ forms a regular sequence of $\mathcal{O}[[\Gamma_{F_j,\, \max}]]$ contained in its maximal ideal, and thus $x_1, \ldots,x_{\nu-1}, x_w^I, x_w^F,x_\nu$ is also a regular sequence. Then, since $H_0(F_{j,w},\mathcal{T}_j/\mathfrak{A}_{\nu}\mathcal{T}_{j})$ is isomorphic to $\mathcal{O}[[\Gamma_{F_j,\max}]]/(x_1,\ldots,x_{\nu-1},x_w^I,x_w^F)$ as an $\mathcal{O}[[\Gamma_{F_j,\, \max}]]$-module, we conclude that its $x_\nu$-torsion part should vanish, due to regularity of the sequence $x_1,\ldots,x_{\nu-1}, x_w^I,x_w^F, x_\nu$. 
 
To verify %complete verification of  %for arXiv version
the pseudonullity of $(\mathrm{Ker}\, u_\nu)^\vee$, we finally prove that $H_0(F_{j,w},\mathcal{T}_j/\mathfrak{A}_{\nu-1}\mathcal{T}_j)[x_\nu]$ is $\mathcal{O}[[\Gamma_{F_j,\, \max}]]/\mathfrak{A}_\nu$-pseudonull for each $w\in \Sigma_{F_j,p}^c$. Let $F_{j,\max}^{(\nu)}$ denotes the subfield of $F_{j,\max}$ corresponding to $\langle \gamma_1,\gamma_2,\ldots, \gamma_\nu\rangle$. Recall that, for the unique place $v$ of $F$ lying below $w$, the decomposition subgroup of $\Gamma_{F,\, \max}$ at $v$ has $\mathbb{Z}_p$-rank equal to or greater than $2$ by \cite[Lemma~3.1 (i)]{BCGKST}. This readily implies that the $\mathbb{Z}_p$-rank of the decomposition subgroup of $\mathrm{Gal}(F_{j,\max}^{(\nu)}/F_j)$ at $w$ is also equal to or greater than $2$, which acts trivially on the $D_w$-coinvariant $H_0(F_{j,w},\mathcal{T}_j/\mathfrak{A}_{\nu}\mathcal{T}_j)$ of $\mathcal{T}_j/\mathfrak{A}_\nu\mathcal{T}_j$. Therefore $H_0(F_{j,w}, \mathcal{T}_j/\mathfrak{A}_\nu \mathcal{T}_j)$ is $\mathcal{O}[[\mathrm{Gal}(F_{j,\max}/F_j)]]/\mathfrak{A}_\nu$-pseudonull for each $\nu=0,1,\ldots,r_j$. Then, by \cite[Lemma 3.1]{ochiai2005}, we have 
\begin{multline*}
 \mathrm{char}_{\mathcal{O}[[\Gamma_{F_j,\, \max}]]/\mathfrak{A}_\nu} \bigl(H_0(F_{j,w}, \mathcal{T}_j/\mathfrak{A}_{\nu-1}\mathcal{T}_j)[x_\nu]\bigr) \\
=\mathrm{char}_{\mathcal{O}[[\Gamma_{F_j,\, \max}]]/\mathfrak{A}_\nu} \bigl(H_0(F_{j,w}, \mathcal{T}_j/\mathfrak{A}_{\nu-1}\mathcal{T}_j)/x_\nu H_0(F_{j,w}, \mathcal{T}_j/\mathfrak{A}_{\nu-1}\mathcal{T}_j)\bigr) \\
=\mathrm{char}_{\mathcal{O}[[\Gamma_{F_j,\, \max}]]/\mathfrak{A}_\nu} \bigl(H_0(F_{j,w}, \mathcal{T}_j/\mathfrak{A}_\nu\mathcal{T}_j)\bigr)=0.
\end{multline*}
Thus $H_0(F_{j,w},\mathcal{T}_j/\mathfrak{A}_{\nu-1}\mathcal{T}_j)[x_\nu]$ is $\mathcal{O}[[\Gamma_{F_j,\, \max}]]/\mathfrak{A}_\nu$-pseudonull, as desired. 
\end{component}

\begin{component}[$\blacktriangleright$ Inductive choice of $\gamma_\nu$ satisfying {\em [II]}] 
%Now %for arXiv version
Let us discuss how to choose a $\mathbb{Z}_p$-basis $\gamma_1,\gamma_2,\ldots,\gamma_{r_j}$ of $\mathrm{Gal}(F_{j,w}/F_jF_{\max})$ so that the resulting sequence $x_1,x_2,\ldots, x_{r_j}$ satisfies [II] (and also [I]). We here follow the argument of \cite[Section~3.4.5]{haraochiai}. 
It suffices to choose $\gamma_1,\gamma_2,\ldots,\gamma_{r_j}$  so that both the  conditions $(\Gamma 1)_\nu$ and $(\Gamma 2)_\nu$ introduced in \cite[p.68]{haraochiai} are fulfilled for each $\nu=1,2,\ldots, r_j$; indeed we have verified in \cite[Proposition 3.35]{haraochiai} that the conditions $(\Gamma 1)_\nu$ and $(\Gamma 2)_\nu$ for each $\nu$ imply the assertion [II], the proof of which is heavily based upon \cite[Theorem~4.1.1]{Greenberg16}. 
But if a part of $\mathbb{Z}_p$-basis  $\gamma_1 , \gamma_2,\ldots, \gamma_{\nu-1}$ of $\mathrm{Gal}(F_{j,\max}/F_jF_{\max})$ satisfies the conditions $(\Gamma 1)_1$ to $(\Gamma 1)_{\nu-1}$ 
(resp.\ the conditions $(\Gamma 2)_1$ to $(\Gamma 2)_{\nu-1}$), any $\gamma_1,\gamma_2,\ldots,\gamma_\nu$  forming a part of $\mathbb{Z}_p$-basis of $\mathrm{Gal}(F_{j,\max}/F_jF_{\max})$ satisfies the condition $(\Gamma 1)_{\nu}$ (resp.\ $(\Gamma 2)_\nu$) except for finitely many choices of $\gamma_{\nu}$ by \cite[Lemma 4.1.1 and Remark 2.1.3]{gr-coh} (resp.\ \cite[Proposition 3.37]{haraochiai}). Since there are  infinitely many choices for $\mathbb{Z}_p$-bases of $\mathrm{Gal}(F_{j ,\,\max} /F_j F_{\max})$, we can find 
$\gamma_1,\gamma_2,\ldots,\gamma_{r_j}$ so that $x_1,x_2,\ldots,x_{r_j}$ forms a regular sequence of $\mathcal{O}[[\mathrm{Gal} (F_{j ,\,\max} /F_j) ]]$ satisfying both the conditions [I] and [II]. 
This completes the proof of Theorem \ref{theorem:integrality_of_CMp-adicL_Artin}. \qedhere

\end{component}
\end{proof}

\begin{rem}
 In the final step of the specialisation procedures, we have used the assumption on the nontriviality of the cyclotomic $p$-adic $L$-function $\mathcal{L}^\mathrm{cyc}_p(f_\eta)$ in \cite[Section 3.4.6]{haraochiai}. However, the nontriviality of the corresponding $p$-adic $L$-function $L_{p,\Sigma_{F_j}}(\psi_j) \mod\mathfrak{A}_{r_j}$ follows from the interpolation formula \eqref{eq:p-adic_Hecke}, and thus we do not need any assumption like (NV${}_{\mathcal{L}_p^\mathrm{cyc}(f_\eta)}$).
\end{rem}

\appendix
\section{Davenport--Hasse relation over rings of truncated Witt vectors} \label{app:DH}
In the appendix, we extend so-called Davenport--Hasse relation of Gauss sums over finite fields \cite[(0.8)]{DH} to rings of truncated Witt vectors. For a finite field $\mathbb{F}$ and a positive integer $n$, let $W_n(\mathbb{F})$ denote the ring of $n$-truncated Witt vectors of $\mathbb{F}$. Then, for a finite extension $\mathbb{E}/\mathbb{F}$, the norm map $\mathrm{Nr}_{\mathbb{E}/\mathbb{F}}\colon \mathbb{E}^\times \rightarrow \mathbb{F}^\times$ and  the trace map $\mathrm{Tr}_{\mathbb{E}/\mathbb{F}}\colon \mathbb{E} \rightarrow \mathbb{F}$ respectively induce $\mathrm{Nr}_{\mathbb{E}/\mathbb{F}}\colon W_n(\mathbb{E})^\times \rightarrow W_n(\mathbb{F})^\times$ and  $\mathrm{Tr}_{\mathbb{E}/\mathbb{F}}\colon W_n(\mathbb{E}) \rightarrow W_n(\mathbb{F})$ in a functorial way. Hereafter we always identify $W_n(\mathbb{F}_p)$ with $\mathbb{Z}/p^n\mathbb{Z}$ via the canonical isomorphism between them.

 We  %hereafter %need revise
fix a finite field $\mathbb{F}_q$ of characteristic $p$ and consider  its extension $\mathbb{F}_{q^s}$ of degree $s$ 
for a natural number $s$. Let $\chi \colon W_n (\mathbb{F}_q )^\times \rightarrow \mathbb{C}^\times$ be %an arbitrary multiplicative %for arXiv version
a multiplicative character. 
For each $n\in \mathbb{N}$, we fix a standard additive character $\psi_n^\circ \colon W_n ( \mathbb{F}_p ) \longrightarrow \mathbb{C}^\times$ sending $x \in W_n ( \mathbb{F}_p )= \mathbb{Z}/p^n\mathbb{Z}$ to $\zeta^x_{p^n}=\exp(2\pi i/p^n)$, a primitive $p^n$-th root of unity. For each  $\nu\in \mathbb{N}$, we define a multiplicative character  $\chi_{\mathbb{F}_{q^\nu}}\colon W_n ( \mathbb{F}_{q^\nu} )^\times \rightarrow \mathbb{C}^\times $ 
to be $\chi \circ \mathrm{Nr}_{\mathbb{F}_{q^\nu}/\mathbb{F}_q}$. 

Now we define the {\em Gauss sum} $\tau (\chi) $ over $W_n(\mathbb{F}_q)$ to be 
\begin{equation}\label{equation:definition-of-gausssum}
\tau (\chi) = \sum_{x \in W_n (\mathbb{F}_q )} \chi (x) \psi_{n,\mathbb{F}_q} (x), 
\end{equation}
where we extend $\chi$ to whole $W_n (\mathbb{F}_q )$ by setting $\chi (x) =0$ if $x\in W_n (\mathbb{F}_q )$ is not invertible. We also define the Gauss sum $\tau(\chi_{\mathbb{F}_{q^s}})$ over $W_n(\mathbb{F}_{q^s})$ in a similar manner replacing $\chi$ and $\psi_n$ with $\chi_{\mathbb{F}_{q^s}}$ and $\psi_{n,\mathbb{F}_{q^s}}$, respectively. The claim which we shall verify in the appendix is the following.

\begin{thm}\label{theorem:D-H_relation_overW_n}
Let the notation be as above. Then we have the following equality of Gauss sums$:$ 
\begin{align} \label{eq:DH_p^n}
\tau (\chi_{\mathbb{F}_{q^s}})  = (-1)^{n(s-1)} \tau (\chi )^s. 
\end{align}
\end{thm}
Theorem \ref{theorem:D-H_relation_overW_n} is an extension of the classical theorem due to Davenport and Hasse \cite[(0.8)]{DH} in the sense that the special case $n=1$ of Theorem \ref{theorem:D-H_relation_overW_n} recovers their original statement. The first-named author leaned from Mahiro Atsuta that Daichi Takeuchi has also verified Theorem~\ref{theorem:D-H_relation_overW_n} independently when the base field is a prime field (that is, the case where $q$ equals $p$); see \cite[Appendix B]{ADK}.

\medskip
We here summarise basic ingredients which we use throughout the appendix.

\begin{itemize}[leftmargin=2em]
 \item[--] Let $\mathbb{F}$ be a finite extension of $\mathbb{F}_p$. For each $r=1,\ldots, n-1$, we obtain a short exact sequence
\begin{align}\label{eq:decomp_Wn_ses}
 \begin{tikzcd}[ampersand replacement=\&, contains/.style = {draw=none,"\in" description,sloped}]
1 \ar[r] \& 1+ p^{n-r}W_n(\mathbb{F}) \ar[r] \& W_n(\mathbb{F})^\times  \ar[r, "\; \mathrm{pr}^n_{n-r}\; "] \& W_{n-r}(\mathbb{F})^{\times}  \ar[r] \& 1 
\end{tikzcd}
\end{align} 
induced from the projection $\mathrm{pr}^n_{n-r}\colon W_n(\mathbb{F})\twoheadrightarrow W_{n-r} (\mathbb{F})$. Moreover there is a bijection $W_r(\mathbb{F})\rightarrow 1+p^{n-r}W_n(\mathbb{F})\,;\, w\mapsto 1+p^{n-r}\widetilde{w}$ for a lift  $\widetilde{w}$ of $w$ with respect to the narural surjection %added %for arXiv version
 $\mathrm{pr}^n_{r}\colon W_n(\mathbb{F})\twoheadrightarrow W_{r} (\mathbb{F})$. As is well known, the value $1+p^{n-r}\widetilde{w}$ does not depend on the choice of lifts $\widetilde{w}$; indeed, we have 
\begin{align*}
 1+p^{n-r}\widetilde{w}=1+(\underbrace{0,0,\ldots,0}_{n-r},w_0^{p^{n-r}},w_1^{p^{n-r}},\ldots,w_{r-1}^{p^{n-r}})
\end{align*} 
for $w=(w_0,w_1,\ldots,w_{r-1})\in W_r(\mathbb{F})$. Now let us fix a lift $\widetilde{z}$ of each $z \in W_{n-r}(\mathbb{F})$ to $W_n(\mathbb{F})$. This enables us to describe  $W_n(\mathbb{F}_{q^s})^\times$ as 
\begin{align} \label{eq:decomp_Wn}
W_n(\mathbb{F}_{q^s})^\times=\{\widetilde{z}(1+p^{n-r}\widetilde{w}) \mid w\in W_r(\mathbb{F}),\, z\in W_{n-r}(\mathbb{F})^\times\}.
\end{align}
For later convenience, we choose a lift $\widetilde{z}$ of $z$ so that it is contained in $W_n(\mathbb{F}')$ if $z$ is an element of $W_{n-r}(\mathbb{F}')$ for a subfield $\mathbb{F}'$ of $\mathbb{F}$.

\item[--] For each $r=1,2,\ldots,n-1$, let $\psi_r\colon W_r(\mathbb{F}_p)\cong \mathbb{Z}/p^r\mathbb{Z}\rightarrow \mathbb{C}^\times$ denote an additive character defined as $w\mapsto \zeta_{p^r}^w$ for $\zeta_{p^r}:=\exp(2\pi i/p^r)$. Set $\psi_{r,\mathbb{F}}:=\psi_r\circ \mathrm{Tr}_{\mathbb{F}/\mathbb{F}_p}$ for a finite extension $\mathbb{F}$ of $\mathbb{F}_p$. %for arXiv version %change the order
Then any additive character of $W_r(\mathbb{F})$ is described as $x\mapsto \psi_{r,\mathbb{F}_q}(\alpha x)$ for a unique element $\alpha \in W_r(\mathbb{F})$. This is a consequence of perfectness of the trace pairing; see the following Lemma~\ref{lem:trace_pairing}. By definition we have an equality $\psi_{n,\mathbb{F}}(p^{n-r}\widetilde{w})=\psi_{r,\mathbb{F}}(w)$ for any $r=1,2,\ldots,n-1$ and any $w\in W_r(\mathbb{F})$.
\end{itemize}

\begin{lem} \label{lem:trace_pairing}
 Let $\mathbb{F}$ be a finite extension of $\mathbb{F}_p$ and $r$ a natural number. Then the trace pairing
\begin{align*} 
\langle \; , \; \rangle_r \colon  W_r(\mathbb{F}) \times W_r(\mathbb{F}) \longrightarrow W_r(\mathbb{F}_p) \, ;\; (x,y) \mapsto \mathrm{Tr}_{\mathbb{F}/\mathbb{F}_p}(xy)
\end{align*}
is a perfect pairing of $W_r(\mathbb{F}_p)$-modules.  
\end{lem}

\begin{proof}
 When $r$ is equal to $1$, perfectness of $\langle \;,\;\rangle_1$ is a consequence of separability of $\mathbb{F}/\mathbb{F}_p$ (see, for example,  the proof of \cite[Proposition~2.8]{Neukirch}). For $r>1$, consider a commutative diagram
\begin{align*}
\xymatrix@C-2pc{
p^{r-1}W_r(\mathbb{F}) \ar@{^(->}[d]  & \times & W_r(\mathbb{F})/pW_r(\mathbb{F}) \ar[rrrrr]  &&&&& p^{r-1}W_r(\mathbb{F}_p)  \ar@{^(->}[d]  \\
W_r(\mathbb{F}) \ar@{->>}[d] & \times & W_r(\mathbb{F}) \ar@{->>}[u] \ar[rrrrr]^-{\langle \; ,\;\rangle_r} &&&&& W_r(\mathbb{F}_p) \\
W_r(\mathbb{F})/p^{r-1}W(\mathbb{F}) & \times & pW_r(\mathbb{F}) \ar@{^(->}[u] \ar[rrrrr] &&&&& pW_{r}(\mathbb{F}_p). \ar@{^(->}[u]
}
\end{align*}
Here the left and middle colums are exact, and the top and bottom horizontal arrows denote the pairings induced by $\langle \;,\;\rangle_r$.
But for any natural number $j$ satisfying $1\leq j \leq r$, we have isomorphisms $W_r(\mathbb{F})/p^j W_r(\mathbb{F})\cong W_{j}(\mathbb{F})$ and $W_j(\mathbb{F})\xrightarrow{\,\sim\,} p^{r-j}W_r(\mathbb{F})\,;\, w \mapsto p^{r-j}\widetilde{w}$, where $\widetilde{w}$ is an arbitrary lift of $w$ to $W_r(\mathbb{F})$. Via these isomorphisms, the top and bottom pairings in the diagram are identified with $\langle \; ,\;\rangle_1$ and $\langle \;, \;\rangle_{r-1}$, respectively. 

Now let $\Phi_j \colon W_j(\mathbb{F})\rightarrow W_j(\mathbb{F})^*:=\mathrm{Hom}_{W_j(\mathbb{F}_p)}(W_j(\mathbb{F}),W_j(\mathbb{F}_p))$ be a $W_j(\mathbb{F}_p)$-homomorphism induced by the pairing $\langle \; , \;\rangle_j$. By the  discussion above, we have a commutative diagram of abelian groups with exact rows
\begin{align*}
\xymatrix{0\ar[r] & \mathbb{F} \ar[r] \ar[d]_{\Phi_1} & W_r(\mathbb{F}) \ar[r] \ar[d]^{\Phi_r} & W_{r-1}(\mathbb{F}) \ar[r] \ar[d]^{\Phi_{r-1}} & 0 \\
0\ar[r] & \mathbb{F}^* \ar[r]  & W_r(\mathbb{F})^* \ar[r]  & W_{r-1}(\mathbb{F})^* \ar[r]  & 0.}
\end{align*}
By induction hypothesis, the left and right vertical maps are isomorphisms, and thus the five lemma implies that  $\Phi_r$ is also an isomorphism.
\end{proof}

Let us prove Theorem~\ref{theorem:D-H_relation_overW_n}. Since the case where $n=1$ is nothing but the original Davenport--Hasse relation \cite[(0.8)]{DH}, it suffices to prove the statement assuming $n>1$. Contrary to the original proof of Davenport and Hasse for the case $n=1$ (using a product decomposition of a certain generating function), we explicitly evaluate the Gauss sums $\tau(\chi_{\mathbb{F}_{q^s}})$ and $\tau(\chi)$ utilising Lamprecht's techniques \cite[Sektionen 2.3--2.5]{Lam}, and directly compare them. A similar calculation is also developed by Mauclaire \cite{Mau1, Mau2} for Gauss sums over $W_n(\mathbb{F}_p)=\mathbb{Z}/p^n\mathbb{Z}$ (see also an exposition of \cite[Section~1.6]{BEW}, especially Theorem 1.6.4).

\begin{proof}[Proof of Theorem~{\protect\rm \ref{theorem:D-H_relation_overW_n}}]
The proof goes in a different way depending on the {\em parity} of $n$. 

\smallskip
\noindent {\bf \underline{\bf Case 1}.\ $n$: even} \;  Firstly, we assume that $n$ is even and set $n=2r$ with a positive integer $r$. Using \eqref{eq:decomp_Wn}, we have 
\begin{align*}
\tau (\chi_{\mathbb{F}_{q^s}})  &= \sum_{z \in W_r(\mathbb{F}_{q^s})^{\times}} \sum_{w \in W_{r}(\mathbb{F}_{q^s})} 
 \chi_{\mathbb{F}_{q^s}} (\widetilde{z} (1+p^r\widetilde{w})) \psi_{n,\mathbb{F}_{q^s}} (\widetilde{z} (1+p^r \widetilde{w})) \\ 
 & = \sum_{z \in W_r(\mathbb{F}_{q^s})^{\times}} \chi_{\mathbb{F}_{q^s}}(\widetilde{z}) \psi_{n,\mathbb{F}_{q^s}}(\widetilde{z})\sum_{w \in W_{r}(\mathbb{F}_{q^s})} 
\chi\circ \mathrm{Nr}_{\mathbb{F}_{q^s}/\mathbb{F}_q}(1+p^r\widetilde{w})\psi_{n,\mathbb{F}_{q^s}}(p^r\widetilde{z}\widetilde{w}) \\ 
 & = \sum_{z \in W_r(\mathbb{F}_{q^s})} \chi_{\mathbb{F}_{q^s}}(\widetilde{z}) \psi_{n,\mathbb{F}_{q^s}}(\widetilde{z}) \sum_{w \in W_{r}(\mathbb{F}_{q^s})} 
\chi(1+p^r\mathrm{Tr}_{\mathbb{F}_{q^s}/\mathbb{F}_q}(\widetilde{w}))\psi_{r,\mathbb{F}_{q^s}}(zw).
\end{align*}
Here the first equality is just the definition of $\tau(\chi_{\mathbb{F}_{q^s}})$, and the middle equality follows from  multiplicativity of $\chi_{\mathbb{F}_{q^s}}$ and  additivity of $\psi_{n,\mathbb{F}_{q^s}}$.
At the last equality, we use the fact that $\chi_{\mathbb{F}_{q^s}}(\widetilde{z})=0$ if $z$ is not a unit of $W_r(\mathbb{F}_{q^s})$, and the equality
\begin{align*}
 \mathrm{Nr}_{\mathbb{F}_{q^s}/\mathbb{F}_q}(1+p^r\widetilde{w})&=\prod_{\sigma\in \mathrm{Gal}(\mathbb{F}_{q^s}/\mathbb{F}_q)} (1+p^r\widetilde{w}^\sigma) \\
&=1+p^r\sum_{\sigma\in \mathrm{Gal}(\mathbb{F}_{q^s}/\mathbb{F}_q)}\widetilde{w}^\sigma=1+p^r\mathrm{Tr}_{\mathbb{F}_{q^s}/\mathbb{F}_q} (\widetilde{w})
\end{align*}
which holds because $p^{2r}=p^n=0$ in $W_n(\mathbb{F}_{q^s})$. By the same reason, we readily observe that $x\mapsto \chi(1+p^r \widetilde{x})$ is an additive character of $W_r(\mathbb{F}_{q})$. Thus, by Lemma~\ref{lem:trace_pairing}, there exists a unique element $-\varepsilon_\chi\in W_r(\mathbb{F}_q)$ satisfying
\begin{align*}
 \tau(\chi_{\mathbb{F}_{q^s}}) & = \sum_{z \in W_r(\mathbb{F}_{q^s})} \chi_{\mathbb{F}_{q^s}}(\widetilde{z}) \psi_{n,\mathbb{F}_{q^s}}(\widetilde{z}) \sum_{w \in W_{r}(\mathbb{F}_{q^s})} 
\psi_{r,\mathbb{F}_q}(-\varepsilon_\chi\mathrm{Tr}_{\mathbb{F}_{q^s}/\mathbb{F}_q}(w))\psi_{r,\mathbb{F}_{q^s}}(zw) \\
& = \sum_{z \in W_r(\mathbb{F}_{q^s})
} \chi_{\mathbb{F}_{q^s}}(\widetilde{z}) \psi_{n,\mathbb{F}_{q^s}}(\widetilde{z}) \sum_{w \in W_{r}(\mathbb{F}_{q^s})} 
\psi_{r,\mathbb{F}_{q^s}}((z-\varepsilon_\chi)w).
\end{align*}
Note that the character sum with respect to $w$ equals  $\# W_r(\mathbb{F}_{q^s})=q^{rs}$ when $z$ coincides with $\varepsilon_\chi$, and vanishes otherwise. We thus have
\begin{align}
  \tau(\chi_{\mathbb{F}_{q^s}})&=q^{rs}\chi_{\mathbb{F}_{q^s}}(\widetilde{\varepsilon}_\chi)\psi_{n,\mathbb{F}_{q^s}}(\widetilde{\varepsilon}_\chi)
=\bigl(q^r\chi(\widetilde{\varepsilon}_\chi)\psi_{n,\mathbb{F}_q}(\widetilde{\varepsilon}_\chi)\bigr)^s. \label{eq:even_qs}
 \end{align}
Here recall that the lift $\widetilde{\varepsilon}_\chi$ of $\varepsilon_\chi\in W_r(\mathbb{F}_q)$ is chosen so that $\widetilde{\varepsilon}_\chi\in W_n(\mathbb{F}_q)$. By a very similar calculation, we also have
\begin{align} 
 \tau(\chi) & = \sum_{z \in W_r(\mathbb{F}_{q})} \chi(\widetilde{z}) \psi_{n,\mathbb{F}_{q}}(\widetilde{z}) \sum_{w \in W_{r}(\mathbb{F}_{q})} 
\psi_{r,\mathbb{F}_q}(-\varepsilon_\chi w)\psi_{r,\mathbb{F}_{q}}(zw) \label{eq:even_q} \\ 
& = \sum_{z \in W_r(\mathbb{F}_{q})
} \chi(\widetilde{z}) \psi_{n,\mathbb{F}_{q}}(\widetilde{z}) \sum_{w \in W_{r}(\mathbb{F}_{q})} 
\psi_{r,\mathbb{F}_{q}}((z-\varepsilon_\chi)w)=q^r\chi(\widetilde{\varepsilon}_\chi)\psi_{n,\mathbb{F}_q}(\widetilde{\varepsilon}_\chi). \nonumber
 \end{align}
Comparing \eqref{eq:even_qs} with \eqref{eq:even_q}, we obtain the desired equality \eqref{eq:DH_p^n}; note that %we have  
$(-1)^{n(s-1)}=1$ holds  %for arXiv version
in Case 1 because $n$ is assumed to be even. 

\bigskip
\noindent {\bf \underline{\bf Case 2}.\ $n$: odd} \;  We next assume that $n$ is odd and set $n=2r+1$ with a positive integer $r$. Let $\nu$ be any natural number satisfying $1\leq \nu \leq s$. 
Similarly to Case 1, we have 
\begin{align*}
\tau (\chi_{\mathbb{F}_{q^\nu}}) & = \sum_{z \in W_{r+1}(\mathbb{F}_{q^\nu})^{\times}} \sum_{w \in W_{r}(\mathbb{F}_{q^\nu})} 
 \chi_{\mathbb{F}_{q^\nu}} (\widetilde{z} (1+p^{r+1}\widetilde{w})) \psi_{n,\mathbb{F}_{q^\nu}} (\widetilde{z} (1+p^{r+1} \widetilde{w})) \quad  (\text{by \eqref{eq:decomp_Wn}})\\ 
 & = \sum_{z \in W_{r+1}(\mathbb{F}_{q^\nu})^{\times}} \chi_{\mathbb{F}_{q^\nu}}(\widetilde{z}) \psi_{n,\mathbb{F}_{q^\nu}}(\widetilde{z})\\
&\hspace*{7em} \cdot  \sum_{w \in W_{r}(\mathbb{F}_{q^\nu})} 
\chi\circ \mathrm{Nr}_{\mathbb{F}_{q^\nu}/\mathbb{F}_q}(1+p^{r+1}\widetilde{w})\psi_{n,\mathbb{F}_{q^\nu}}(p^{r+1}\widetilde{z}\widetilde{w}) \\ 
 & = \sum_{z \in W_{r+1}(\mathbb{F}_{q^\nu})} \chi_{\mathbb{F}_{q^s}}(\widetilde{z}) \psi_{n,\mathbb{F}_{q^\nu}}(\widetilde{z}) \\
&\hspace*{5em} \cdot  \sum_{w \in W_{r}(\mathbb{F}_{q^\nu})} 
\chi(1+p^{r+1}\mathrm{Tr}_{\mathbb{F}_{q^\nu}/\mathbb{F}_q}(\widetilde{w}))\psi_{r,\mathbb{F}_{q^\nu}}(\mathrm{pr}^{r+1}_r(z)w).
\end{align*}
Here $\mathrm{pr}^{r+1}_r\colon W_{r+1}(\mathbb{F}_{q^\nu})\twoheadrightarrow W_r(\mathbb{F}_{q^\nu})$ is the natural projection. As in Case 1, a map sending $x\mapsto \chi(1+p^{r+1}\widetilde{x})$ defines an additive character of $W_r(\mathbb{F}_{q})$. There thus exists a unique element $-\varepsilon_\chi\in W_r(\mathbb{F}_q)$ by Lemma~\ref{lem:trace_pairing}, and we have
\begin{align*}
 \tau(\chi_{\mathbb{F}_{q^\nu}}) & = \sum_{z \in W_{r+1}(\mathbb{F}_{q^\nu})} \chi_{\mathbb{F}_{q^\nu}}(\widetilde{z}) \psi_{n,\mathbb{F}_{q^s}}(\widetilde{z}) \\
&\hspace*{5em} \cdot  \sum_{w \in W_{r}(\mathbb{F}_{q^\nu})} 
\psi_{r,\mathbb{F}_q}(-\varepsilon_\chi\mathrm{Tr}_{\mathbb{F}_{q^\nu}/\mathbb{F}_q}(w))\psi_{r,\mathbb{F}_{q^\nu}}(\mathrm{pr}^{r+1}_r(z)w) \\
& = \sum_{z \in W_{r+1}(\mathbb{F}_{q^\nu})
} \chi_{\mathbb{F}_{q^\nu}}(\widetilde{z}) \psi_{n,\mathbb{F}_{q^\nu}}(\widetilde{z}) \sum_{w \in W_{r}(\mathbb{F}_{q^\nu})} 
\psi_{r,\mathbb{F}_{q^\nu}}((\mathrm{pr}^{r+1}_r(z)-\varepsilon_\chi)w) \\
&=q^{r\nu}\sum_{\substack{z \in W_{r+1}(\mathbb{F}_{q^\nu}) \\ \mathrm{pr}^{r+1}_r(z)=\varepsilon_\chi}} \chi_{\mathbb{F}_{q^\nu}}(\widetilde{z}) \psi_{n,\mathbb{F}_{q^\nu}}(\widetilde{z}).
\end{align*}
Note that, if $\varepsilon_\chi$ is not a unit of $W_n(\mathbb{F}_q)$, we have $\chi_{\mathbb{F}_{q^\nu}}(\widetilde{z})=0$ for any $z\in W_{r+1}(\mathbb{F}_{q^\nu})$ satisfying $\mathrm{pr}^{r+1}_r(z)=\varepsilon_\chi$. In the case, the desired equality \eqref{eq:DH_p^n} trivially holds because the both sides of \eqref{eq:DH_p^n} are reduced to $0$. We thus assume in the following that $\varepsilon_\chi$ is a unit. Due to \eqref{eq:decomp_Wn_ses} with $n$ and $r$ replaced by $r+1$ and $1$ respectively, we obtain an equality of sets
\begin{align*}
 \{ \widetilde{z}\in W_n(\mathbb{F}_{q^\nu}) \mid z\in W_{r+1}(\mathbb{F}_{q^\nu}),\, \mathrm{pr}^{r+1}_r(z)=\varepsilon_\chi \}=\{\widetilde{\varepsilon}_\chi(1+p^r[\delta]) \mid \delta\in \mathbb{F}_{q^\nu}\},
\end{align*}
where $[\, \cdot\,] \colon \mathbb{F}_q^\times \rightarrow W_n(\mathbb{F}_q)^\times$ denotes the Teichm\"uller lift. Therefore we have
\begin{align}
  \tau(\chi_{\mathbb{F}_{q^\nu}})&=q^{r\nu}\sum_{\delta\in \mathbb{F}_{q^\nu}} \chi_{\mathbb{F}_{q^\nu}}(\widetilde{\varepsilon}_\chi (1+p^r[\delta]))\psi_{n,\mathbb{F}_{q^\nu}}(\widetilde{\varepsilon}_\chi(1+p^r[\delta])) \label{eq:odd_qs} \\
&=q^{r\nu}\chi_{\mathbb{F}_{q^\nu}}(\widetilde{\varepsilon}_\chi)\psi_{n,\mathbb{F}_{q^\nu}}(\widetilde{\varepsilon}_\chi) \sum_{\delta\in \mathbb{F}_{q^\nu}} \chi_{\mathbb{F}_{q^\nu}}(1+p^r[\delta])\psi_{n,\mathbb{F}_{q^\nu}}(p^r\widetilde{\varepsilon}_\chi[\delta])\nonumber  \\
&=\bigl(q^r\chi(\widetilde{\varepsilon}_\chi)\psi_{n,\mathbb{F}_{q}}(\widetilde{\varepsilon}_\chi)\bigr)^\nu \sum_{\delta\in \mathbb{F}_{q^\nu}} \chi_{\mathbb{F}_{q^\nu}}(1+p^r[\delta])\psi_{r+1,\mathbb{F}_{q^\nu}}(\widetilde{\varepsilon}_\chi[\delta]). \nonumber
 \end{align}
Here we use the same symbol $[\, \cdot\,]$ for the Teichm\"uller lift with values in $W_{r+1}(\mathbb{F}_{q^\nu})$ by abuse of notation. 
To verify the desired equality \eqref{eq:DH_p^n}, it suffices to prove the following claim.

\begin{claim}
Let the notation be as above. Then the following equality holds$:$
 \begin{multline} \label{eq:use_DH}
 \sum_{\delta\in \mathbb{F}_{q^s}} \chi_{\mathbb{F}_{q^s}}(1+p^r[\delta])\psi_{r+1,\mathbb{F}_{q^s}}(\widetilde{\varepsilon}_\chi[\delta]) =(-1)^{s-1}\left(\sum_{\delta\in \mathbb{F}_q} \chi(1+p^r[\delta])\psi_{r+1,\mathbb{F}_{q}}(\widetilde{\varepsilon}_\chi[\delta]) \right)^s.
\end{multline}
\end{claim}

Indeed, if we admit the equality \eqref{eq:use_DH}, we can deduce \eqref{eq:DH_p^n} as
\begin{align*}
 \tau(\chi_{\mathbb{F}_{q^s}}) &=(q^r \chi(\widetilde{\varepsilon}_\chi)\psi_{n,\mathbb{F}_q}(\widetilde{\varepsilon}_\chi))^s \sum_{\delta\in {\mathbb{F}_{q^s}}}\chi_{\mathbb{F}_{q^s}}(1+p^r[\delta])\psi_{r+1,\mathbb{F}_{q^s}}(\widetilde{\varepsilon}_\chi [\delta]) \\
&=(-1)^{s-1}(q^r \chi(\widetilde{\varepsilon}_\chi)\psi_{n,\mathbb{F}_q}(\widetilde{\varepsilon}_\chi))^s \left(\sum_{\delta\in{\mathbb{F}_q}}\chi(1+p^r[\delta])\psi_{r+1,\mathbb{F}_{q}}(\widetilde{\varepsilon}_\chi [\delta])\right)^s  \\ 
&=(-1)^{s-1}\tau(\chi).
\end{align*}
Here the first (resp.\ the third) equality is nothing but \eqref{eq:odd_qs} for  $\nu=s$ (resp.\ $\nu =1$) and the second equality follows from \eqref{eq:use_DH}. Note that we have $(-1)^{n(s-1)}=(-1)^{s-1}$ in Case 2 because $n$ is assumed to be odd. 

To prove the claim, we shall observe that the both sides of \eqref{eq:use_DH} are described in terms of {\em quadratic Gauss sums} over finite fields (or their variants when $p$ equals $2$), and then apply the classical Davenport--Hasse relation \cite[(0.8)]{DH} (or explicitly evaluate the partial sums when $p$ equals $2$) to obtain the equality \eqref{eq:use_DH}. 

\begin{proof}[Proof of {\protect\underline{\protect\bf Claim}}]
 As before, let $\nu$ be any integer satisfying $1\leq \nu\leq s$. First note that, for any $\delta_1, \delta_2\in \mathbb{F}_{q^\nu}$, we have 
\begin{align*} 
(1+p^r[\delta_1])&(1+p^r[\delta_2]) \\
&=1+p^r[\delta_1+\delta_2]+p^r([\delta_1]+[\delta_2]-[\delta_1+\delta_2])+p^{2r}[\delta_1\delta_2] \\
&=\bigl(1+p^r[\delta_1+\delta_2]\bigr)(1+p^{r+1}\epsilon_{\delta_1,\delta_2}+p^{2r}[\delta_1\delta_2]) \qquad \text{in }W_n(\mathbb{F}_{q^\nu})
\end{align*}
where $\epsilon_{\delta_1,\delta_2}$ is an element of $W_n(\mathbb{F}_{q^\nu})$ satisfying $[\delta_1]+[\delta_2]-[\delta_1+\delta_2]=p\epsilon_{\delta_1,\delta_2}$.
Using this equality, we can calculate as
\begin{align} 
 \chi_{\mathbb{F}_{q^\nu}}&(1+p^r[\delta_1])\psi_{r+1,\mathbb{F}_{q^\nu}}(\widetilde{\varepsilon}_\chi[\delta_1])\chi_{\mathbb{F}_{q^\nu}}(1+p^r[\delta_2])\psi_{r+1,\mathbb{F}_{q^\nu}}(\widetilde{\varepsilon}_\chi[\delta_2]) \label{eq:chi}  \\
&=\chi_{\mathbb{F}_{q^\nu}}\bigl(1+p^r[\delta_1+\delta_2]\bigr)\psi_{r+1,\mathbb{F}_{q^\nu}}\bigl(\widetilde{\varepsilon}_\chi[\delta_1+\delta_2]\bigr) \nonumber \\
&\hspace*{5em} \cdot \chi_{\mathbb{F}_{q^\nu}}\bigl(1+p^{r+1}(\epsilon_{\delta_1,\delta_2}+p^{r-1}[\delta_1\delta_2])\bigr)\psi_{r+1,\mathbb{F}_{q^\nu}}(p\widetilde{\varepsilon}_\chi \epsilon_{\delta_1,\delta_2}) \nonumber \\
&=\chi_{\mathbb{F}_{q^\nu}}\bigl(1+p^r[\delta_1+\delta_2]\bigr)\psi_{r+1,\mathbb{F}_{q^\nu}}\bigl(\widetilde{\varepsilon}_\chi[\delta_1+\delta_2]\bigr) \nonumber \\
&\hspace*{5em} \cdot \psi_{r,\mathbb{F}_{q^\nu}}\bigl(-\varepsilon_\chi (\epsilon_{\delta_1,\delta_2}+p^{r-1}[\delta_1\delta_2])\bigr)\psi_{r,\mathbb{F}_{q^\nu}}(\varepsilon_\chi \epsilon_{\delta_1,\delta_2}) \nonumber \\
&=\chi_{\mathbb{F}_{q^\nu}}\bigl(1+p^r[\delta_1+\delta_2]\bigr)\psi_{r+1,\mathbb{F}_{q^\nu}}\bigl(\widetilde{\varepsilon}_\chi[\delta_1+\delta_2]\bigr)\psi_{1,\mathbb{F}_{q^\nu}}(-\mathrm{pr}^r_1(\varepsilon_\chi) \delta_1\delta_2). \nonumber
 \end{align}
Here the first equality follows from the calculation above and the second equality follows from the definition of $\varepsilon_\chi$; namely $\chi(1+p^{r+1}x)=\psi_{r,\mathbb{F}_q}(-\varepsilon_\chi x)$ holds for any $x\in W_r(\mathbb{F}_q)$. Set $w_\chi:=\mathrm{pr}^r_1(\varepsilon_\chi)\in \mathbb{F}_q^\times$ for brevity.

\medskip
\noindent {\bf\underline{Case 2-a}. $p$: odd} \; For any $\delta_1,\delta_2\in \mathbb{F}_{q^\nu}$, we have 
\begin{align}\label{eq:p_odd_quad_exp}
 \psi_{1,\mathbb{F}_{q^\nu}}&(2^{-1}w_\chi\delta_1^2)
 \psi_{1,\mathbb{F}_{q^\nu}}(2^{-1}w_\chi\delta_2^2) =\psi_{1,\mathbb{F}_{q^\nu}}\bigl(2^{-1}w_\chi(\delta_1+\delta_2)^2\bigr)\psi_{1,\mathbb{F}_{q^\nu}}\bigl(-w_\chi\delta_1\delta_2\bigr)
 \end{align}
because we have $\delta_1^2 + \delta_2^2 = (\delta_1 + \delta_2 )^2 -2 \delta_1 \delta_2$ and $\psi_{1,\mathbb{F}_{q^\nu}}$ is an additive character. 
Comparing \eqref{eq:chi} with \eqref{eq:p_odd_quad_exp}, one observes that 
\begin{align*}
 \delta \mapsto \chi_{\mathbb{F}_{q^\nu}}(1+p^r[\delta])\psi_{r+1,\mathbb{F}_{q^\nu}}(\widetilde{\varepsilon}_\chi[\delta])\psi_{1,\mathbb{F}_{q^\nu}}(2^{-1}w_\chi\delta^2)^{-1}
\end{align*}
defines an additive character of $\mathbb{F}_{q^\nu}$. Hence, due to Lemma~\ref{lem:trace_pairing}, there exists a unique element $b_{(\nu)} \in \mathbb{F}_{q^\nu}$ satisfying
\begin{align} \label{eq:additive}
 \chi_{\mathbb{F}_{q^\nu}}(1+p^r[\delta])\psi_{r+1,\mathbb{F}_{q^\nu}}(\widetilde{\varepsilon}_\chi[\delta])\psi_{1,\mathbb{F}_{q^\nu}}(2^{-1}w_\chi\delta^2)^{-1}=\psi_{1,\mathbb{F}_{q^\nu}}(b_{(\nu)}\delta).
\end{align}
We now verify that $b_{(\nu)}$ coincides with $b_{(1)}$ and thus is an element of $\mathbb{F}_q^\times$. Let $\sigma$ denote the $q$-th power Frobenius automorphism. Then we have 
\begin{align} \label{eq:additive_Frob}
 \chi_{\mathbb{F}_{q^\nu}}(1+p^r[\delta^\sigma])\psi_{r+1,\mathbb{F}_{q^\nu}}(\widetilde{\varepsilon}_\chi[\delta^\sigma])\psi_{1,\mathbb{F}_{q^\nu}}(2^{-1}w_\chi (\delta^\sigma)^{2})^{-1}
=\psi_{1,\mathbb{F}_{q^\nu}}(b_{(\nu)}\delta^\sigma)= \psi_{1,\mathbb{F}_{q^\nu}}(b_{(\nu)}^{\sigma^{-1}}\delta)
\end{align}
for each $\delta\in \mathbb{F}_{q^\nu}$ by replacing $\delta$ with $\delta^\sigma$ in \eqref{eq:additive}. Furthermore the left-hand side of \eqref{eq:additive_Frob} coincides with that of \eqref{eq:additive}; indeed we have
\begin{align*}
 \chi_{\mathbb{F}_{q^\nu}}(1+p^r[\delta^\sigma])&\psi_{r+1,\mathbb{F}_{q^\nu}}(\widetilde{\varepsilon}_\chi[\delta^\sigma])\psi_{1,\mathbb{F}_{q^\nu}}(2^{-1}w_\chi (\delta^\sigma)^2)^{-1} \\
&=\chi_{\mathbb{F}_{q^\nu}}\bigl((1+p^r[\delta])^\sigma\bigr)\psi_{r+1,\mathbb{F}_{q^\nu}}(\widetilde{\varepsilon}_\chi[\delta]^\sigma)\psi_{1,\mathbb{F}_{q^\nu}}(2^{-1}w_\chi (\delta^2)^\sigma)^{-1} \\
&=\chi_{\mathbb{F}_{q^\nu}}(1+p^r[\delta]) \psi_{r+1,\mathbb{F}_{q^\nu}}(\widetilde{\varepsilon}_\chi^{\sigma^{-1}}[\delta])\psi_{1,\mathbb{F}_{q^\nu}}((2^{-1}w_\chi)^{\sigma^{-1}} \delta^2)^{-1}  \\
&=\chi_{\mathbb{F}_{q^\nu}}(1+p^r[\delta]) \psi_{r+1,\mathbb{F}_{q^\nu}}(\widetilde{\varepsilon}_\chi[\delta])\psi_{1,\mathbb{F}_{q^\nu}}((2^{-1}w_\chi) \delta^2)^{-1}. 
\end{align*}
Hence %, by \eqref{eq:additive} and \eqref{eq:additive_Frob}, 
we have $\psi_{1,\mathbb{F}_{q^\nu}}(b_{(\nu)}^{\sigma^{-1}}\delta)=\psi_{1,\mathbb{F}_{q^\nu}}(b_{(\nu)}\delta)$, or equivalently $\psi_{1,\mathbb{F}_{q^\nu}} \bigl((b_{(\nu)}-b_{(\nu)}^{\sigma^{-1}})\delta\bigr)=1$, for each $\delta \in \mathbb{F}_{q^\nu}$ by \eqref{eq:additive} and \eqref{eq:additive_Frob}. %for arXiv version
This implies the equality $b_{(\nu)}^{\sigma}=b_{(\nu)}$, and thus $b_{(\nu)}$ is in particular an element of $\mathbb{F}_q$. Next note that 
\begin{multline*}
 \chi_{\mathbb{F}_{q^\nu}}(1+p^r[\delta])\psi_{r+1,\mathbb{F}_{q^\nu}}(\widetilde{\varepsilon}_\chi[\delta])\psi_{1,\mathbb{F}_{q^\nu}}(-2^{-1}w_\chi\delta^2)^{-1} \\
=\bigl(\chi(1+p^r[\delta]) \psi_{r+1,\mathbb{F}_{q}}(\widetilde{\varepsilon}_\chi[\delta])\psi_{1,\mathbb{F}_{q}}(-2^{-1}w_\chi\delta^2)^{-1}\bigr)^\nu
\end{multline*}
holds for each $\delta\in \mathbb{F}_{q}$. Thus, substituting \eqref{eq:additive} into the both sides of the equality above, we obtain $\psi_{1,\mathbb{F}_{q^\nu}}(b_{(\nu)}\delta)=\psi_{1,\mathbb{F}_q}(b_{(1)}\delta)^\nu$, or equlvalently $\psi_{1,\mathbb{F}_q}\bigl( (\mathrm{Tr}_{\mathbb{F}_{q^\nu}/\mathbb{F}_q}(b_{(\nu)})-sb_{(1)})\delta \bigr)$ is equal to $1$ for each $\delta \in \mathbb{F}_q$. We thus have  $\mathrm{Tr}_{\mathbb{F}_{q^\nu}/\mathbb{F}_q}(b_{(\nu)}-b_{(1)})=0$ and, since $b_{(\nu)}-b_{(1)}$ is indeed an element of $\mathbb{F}_q$ as we have seen, we finally obtain the equality $b_{(\nu)}=b_{(1)}$. Hereafter we set $b=b_{(\nu)}=b_{(1)}$ for brevity.

\smallskip
Now let us return to the proof of \eqref{eq:use_DH}. We can calculate as
\begin{align*}
 &\!\!\!\!\! \sum_{\delta\in \mathbb{F}_{q^\nu}} \chi_{\mathbb{F}_{q^\nu}}(1+p^r[\delta])\psi_{r+1,\mathbb{F}_{q^\nu}}(\widetilde{\varepsilon}_\chi[\delta]) \\
 & = \sum_{\delta\in \mathbb{F}_{q^\nu}}\psi_{1,\mathbb{F}_{q^\nu}}(2^{-1}w_\chi\delta^2)\psi_{1,\mathbb{F}_{q^\nu}}(b \delta)  =\sum_{\delta\in \mathbb{F}_{q^\nu}} \psi_{1,\mathbb{F}_{q^\nu}}(2^{-1}w_\chi\delta^2+b\delta)  \\
&=\psi_{1,\mathbb{F}_{q^\nu}}(-2^{-1}w_\chi^{-1}b^2) \sum_{\delta\in \mathbb{F}_{q^\nu}} \psi_{1,\mathbb{F}_{q^\nu}}\left(2^{-1}w_\chi \delta^2 \right) \qquad (\text{replace $\delta$ by $\delta-w_\chi^{-1}b$}).
\end{align*}
Here the first equality follows from \eqref{eq:additive}, whereas the second and the third equalities follow rather straightforward calculation using additivity of the character $\psi_{r+1,\mathbb{F}_{q^\nu}}$. 
The sum $\sum_{\delta\in \mathbb{F}_{q^\nu}} \psi_{1,\mathbb{F}_{q^\nu}}\left(2^{-1}w_\chi \delta^2 \right)$ appearing above is known as the {\em quadratic Gauss sum} over the finite field $\mathbb{F}_{q^\nu}$. As is well known, the quadratic Gauss sum is presented as a usual Gauss sum for an appropriate quadratic character. In fact, it is rewritten as
\begin{align*}
 &\psi_{1,\mathbb{F}_{q^\nu}}(-2^{-1}w_\chi^{-1}b^2) \sum_{\delta\in \mathbb{F}_{q^\nu}} \psi_{1,\mathbb{F}_{q^\nu}}\left(2^{-1}w_\chi \delta^2 \right) \\
&=\psi_{1,\mathbb{F}_{q^\nu}}(-2^{-1}w_\chi^{-1}b^2) \left\{ \psi_{1,\mathbb{F}_{q^\nu}}(2^{-1}w_\chi \cdot 0)+2 \sum_{x\in \bigl(\mathbb{F}_{q^\nu}^\times \bigr)^2} \psi_{1,\mathbb{F}_{q^\nu}}(2^{-1}w_\chi x) \right\}\\
&=\psi_{1,\mathbb{F}_{q^\nu}}(-2^{-1}w_\chi^{-1}b^2) \left\{1+ \sum_{x\in \mathbb{F}_{q^\nu}^\times}\left(1+\left(\dfrac{\mathrm{Nr}_{\mathbb{F}_{q^\nu}/\mathbb{F}_p}(x)}{p}\right)\right) \psi_{1,\mathbb{F}_{q^\nu}}(2^{-1}w_\chi x) \right\}\\
&\stackrel{(*)}{=}\psi_{1,\mathbb{F}_{q^\nu}}(-2^{-1}w_\chi^{-1}b^2) \sum_{x\in \mathbb{F}_{q^\nu}^\times} \left(\dfrac{\mathrm{Nr}_{\mathbb{F}_{q^\nu}/\mathbb{F}_p}(x)}{p}\right) \psi_{1,\mathbb{F}_{q^\nu}}(2^{-1}w_\chi x) \\
&=\psi_{1,\mathbb{F}_{q^\nu}}(-2^{-1}w_\chi^{-1}b^2)\sum_{x\in \mathbb{F}_{q^\nu}^\times} \left(\dfrac{\mathrm{Nr}_{\mathbb{F}_{q^\nu}/\mathbb{F}_p} (2w_\chi^{-1}x)}{p}\right) \psi_{1,\mathbb{F}_{q^\nu}}(x) \qquad  (\text{replace $x$ by $2w_\chi^{-1}x$}) \\
&=\psi_{1,\mathbb{F}_q}(-2^{-1}w_\chi^{-1}b^2)^\nu \left(\dfrac{ \mathrm{Nr}_{\mathbb{F}_{q}/\mathbb{F}_p}(2w_\chi^{-1})}{p}\right)^\nu \tau\left( \left(\frac{\mathrm{Nr}_{\mathbb{F}_q/\mathbb{F}_p}(-)}{p}\right)_{\mathbb{F}_{q^\nu}}\right).
\end{align*}
Here $\left(\frac{-}{p}\right)$ denotes the Legendre symbol modulo $p$. At the third equality $(*)$, we subtract a trivial exponential sum  $\sum_{x\in \mathbb{F}_{q^\nu}}\psi_{1,\mathbb{F}_{q^\nu}}(2^{-1}w_\chi x)=0$ from the summation.
Since $\tau\left(\left(\frac{\mathrm{Nr}_{\mathbb{F}_q/\mathbb{F}_p}(-)}{p}\right)_{\mathbb{F}_{q^\nu}}\right)$ is a Gauss sum over the finite field $\mathbb{F}_{q^\nu}$, we can apply the original result of Davenport and Hasse \cite[(0.8)]{DH} to it and obtain an equality 
\begin{align*} %for arXiv version: alignised
\tau\left(\left(\frac{\mathrm{Nr}_{\mathbb{F}_q/\mathbb{F}_p}(-)}{p}\right)_{\mathbb{F}_{q^s}}\right)=(-1)^{s-1}\tau\left(\left(\frac{\mathrm{Nr}_{\mathbb{F}_q/\mathbb{F}_p}(-)}{p}\right)\right)^s, 
\end{align*}
which deduces validity of the desired equality \eqref{eq:use_DH}.

\medskip
\noindent {\bf\underline{Case 2-b}. $\boldsymbol{p=2}$} \; 
Consider $\psi_{2,\mathbb{F}_{q^\nu}} (-[w_\chi\delta^2]_2)$ for  $\delta\in \mathbb{F}_{q^\nu}$, where $[\,\cdot\,]_2 \colon \mathbb{F}_{q^\nu}\rightarrow W_2(\mathbb{F}_{q^\nu})$ denotes the Teichm\"uller lift. By definition,  we have
\begin{multline*}
 \psi_{2,\mathbb{F}_{q^\nu}} (-[w_\chi\delta_1^2]_2)\psi_{2,\mathbb{F}_{q^\nu}} (-[w_\chi\delta_2^2]_2) =\psi_{2,\mathbb{F}_{q^\nu}}(-[w_\chi \delta_1^2]_2-[w_\chi \delta_2^2]_2) \\
=\psi_{2,\mathbb{F}_{q^\nu}}\bigl(-[w_\chi (\delta_1+\delta_2)^2]_2\bigr) \psi_{2,\mathbb{F}_{q^\nu}}\bigl([w_\chi]([\delta_1+\delta_2]_2^2-[\delta_1^2]_2-[\delta_2^2]_2)\bigr). 
\end{multline*}
 Using the law of addition $(x_0,x_1)+(y_0,y_1)=(x_0+y_0,x_1+y_1-x_0y_0)$ in $W_2(\mathbb{F}_{q^\nu})$ and multiplicativity of $[\,\cdot\,]_2$, we readily obtain an equality $[\delta_1+\delta_2]^2_2-[\delta_1^2]_2-[\delta_2^2]_2=-2[\delta_1\delta_2]_2$  as 
\begin{align*}
[\delta_1+\delta_2]^2_2+2[\delta_1\delta_2]_2
&=(\delta_1^2+\delta_2^2,0)+(0,-\delta_1^2\delta_2^2)=(\delta_1^2+\delta_2^2,-\delta_1^2\delta_2^2)\\
&= (\delta_1^2,0)+(\delta_2^2,0)=[\delta_1^2]_2+[\delta_2^2]_2.
\end{align*}
We thus have
\begin{align}\label{eq:even_quad_exp}
  \psi_{2,\mathbb{F}_{q^\nu}} (-[w_\chi\delta_1^2]_2)\psi_{2,\mathbb{F}_{q^\nu}} (-[w_\chi\delta_2^2]_2) =\psi_{2,\mathbb{F}_{q^\nu}}\bigl(-[w_\chi(\delta_1+\delta_2)^2]\bigr)\psi_{1,\mathbb{F}_{q^\nu}}(-w_\chi \delta_1\delta_2).
\end{align}
Therefore, by \eqref{eq:chi} and \eqref{eq:even_quad_exp}, we see that there exists a unique element $b\in \mathbb{F}_q$ satisfying
\begin{align*}
  \chi_{\mathbb{F}_{q^\nu}}(1+p^r[\delta])\psi_{r+1,\mathbb{F}_{q^\nu}}(\widetilde{\varepsilon}_\chi[\delta])\psi_{2,\mathbb{F}_{q^\nu}}(-[w_\chi\delta^2])^{-1}=\psi_{2,\mathbb{F}_{q^\nu}}(2[b\delta]_2)
\end{align*}
similarly to the case where $p$ is odd. 
Now we define
\begin{align*}
 \sigma^{(2)}_{\nu}(-[w_\chi]_2):=\sum_{\delta\in \mathbb{F}_{q^\nu}}\psi_{2,\mathbb{F}_{q^\nu}}(-[w_\chi\delta^2]_2),
\end{align*}  
a partial sum of the quadratic Gauss sum $\sum_{\delta\in W_2(\mathbb{F}_{q^\nu})}\psi_{2,\mathbb{F}_{q^\nu}}(-[w_\chi\delta^2]_2)$. Note that $\sigma_\nu^{(2)}(-[w_\chi]_2)$ does not change even if we replace $[\delta^2]_2$ with an arbitrary lift of $\delta^2$ to $W_2(\mathbb{F}_{q^\nu})$ in each summand. Taking this fact into accounts, we can calculate as
\begin{align} 
 \sum_{\delta\in \mathbb{F}_{q^\nu}} \chi_{\mathbb{F}_{q^\nu}}&(1+2^r[\delta])\psi_{r+1,\mathbb{F}_{q^\nu}}(\widetilde{\varepsilon}_\chi[\delta]) =\sum_{\delta\in \mathbb{F}_{q^\nu}} \psi_{2,\mathbb{F}_{q^\nu}}(-[w_\chi \delta^2]_2+2[b\delta]_2) \label{eq:sum_p_2} \\
&=\sum_{\delta\in \mathbb{F}_{q^\nu}}\psi_{2,\mathbb{F}_{q^\nu}}\left(-[w_\chi]_2([\delta]_2-[bw_\chi^{-1}]_2)^2+[b^2w_\chi^{-1}]_2\right) \nonumber \\
&=\psi_{2,\mathbb{F}_{q^\nu}}([b^2w_\chi^{-1}]_2)  \sum_{\delta\in \mathbb{F}_{q^\nu}} \psi_{2,\mathbb{F}_{q^\nu}}\left(-[w_\chi \delta^2]_2 \right)\quad  \; (\text{replace $[\delta]_2$ by $[\delta]_2+[bw_\chi^{-1}]_2$}) \nonumber \\
&=\psi_{2,\mathbb{F}_{q^\nu}}([b^2w_\chi^{-1}]_2) \sigma^{(2)}_{\nu}(-[w_\chi]_2)=\psi_{2,\mathbb{F}_q}([b^2w_\chi^{-1}]_2)^\nu \sigma_\nu^{(2)}(-[w_\chi]_2).  \nonumber
\end{align}
At the third equality, the summation does not change after we replace $[\delta]_2$ by $[\delta]_2+[bw_\chi^{-1}]_2$ because $\{[\delta]_2+[bw_\chi^{-1}]_2\mid \delta\in \mathbb{F}_{q^\nu}\}$ gives a complete set of representatives of $W_2(\mathbb{F}_{q^\nu})/2W_2(\mathbb{F}_{q^\nu})$. The partial sum $\sigma_\nu^{(2)}(-[w_\chi]_2)$ has been calculated as $-\{-(1+i)\}^{\nu\; \mathrm{ord}_p(q)}$ (see \cite[(6.30)]{Lam} for example), and we thus obtain an equality $(-\sigma_1^{(2)}(-[w_\chi]))^s=-\sigma_s^{(2)}(-[w_\chi]_2)$, or equivalently 
\begin{align} \label{eq:eval_partial_sum}
 \sigma_s^{(2)}(-[w_\chi]_2)=(-1)^{s-1}\sigma_1^{(2)}(-[w_\chi]_2)^s.
\end{align}
The desired equality \eqref{eq:use_DH} easily follows from \eqref{eq:sum_p_2} and \eqref{eq:eval_partial_sum}.  
\end{proof}
Now we have verified \eqref{eq:DH_p^n} in all cases, which completes the proof of Theorem~\ref{theorem:D-H_relation_overW_n}.
\end{proof}

\subsubsection*{Acknowledgements} 
The first-named author would like to thank Mahiro Atsuta for valuable discussion especially on calculation of local epsilon factors of unramified Artin representations, which helps the authors to complete the interpolation formula of $L_{p,\Sigma_F}(M(\rho))$. The authors are also obliged to the anonymous referee for his or her suggestive comments, which enable the authors to improve the interpolation properties of the $p$-adic $L$-functions $L_{p,\Sigma_F}(\psi)$ and $L_{p,\Sigma_F}(M(\rho))$. The present work is supported by JSPS KAKENHI (Grant-in-Aid for Scientific Research (B): Grant Number 
JP23K23763, and Grant-in-Aid for Scientific Research (C): Grant Number JP22K03237).

\end{document}